\documentclass[a4paper,10pt,reqno]{amsart}
\usepackage{amsmath,amsfonts,amsthm}
\usepackage[mathscr]{eucal}
\usepackage{multirow}
\usepackage{graphicx}

\numberwithin{equation}{section}
\newtheorem{theorem}{Theorem}[section]
\newtheorem{lemma}[theorem]{Lemma}
\newtheorem{proposition}{Proposition}

\theoremstyle{definition}

\theoremstyle{remark}
\newtheorem{remark}{Remark}

\def\N{\mathbb{N}}

\def\Var{\text{Var}}
\def\Cov{\text{Cov}}

\def\g{\gamma}

\def\dl{\Delta}

\newcommand{\e}{\mathrm{e}}

\newcommand{\Z}{\mathbb Z}
\newcommand{\R}{\mathbb R}

\newcommand{\rd}{\;\mathrm{d}}
\def\ba{\frac{b}{a}}
\def\bba{\frac{2b}{a}}

\begin{document}

\title[Linear SFDE with average functional]{Long Memory and Financial Market Bubble Dynamics in Affine Stochastic Differential Equations with Average Functionals}

\author{John A. D. Appleby}
\address{Edgeworth Centre for Financial Mathematics, School of Mathematical Sciences, Dublin City University,
Ireland} \email{john.appleby@dcu.ie}
\urladdr{webpages.dcu.ie/\textasciitilde applebyj}

\author{John A. Daniels}
\address{Edgeworth Centre for Financial Mathematics, School of Mathematical Sciences, Dublin City University,
Ireland} \email{john.daniels2@mail.dcu.ie} 

\thanks{Both authors gratefully acknowledge Science Foundation Ireland for the support of this research
under the Mathematics Initiative 2007 grant 07/MI/008 ``Edgeworth
Centre for Financial Mathematics''. The second author is also
supported by the Irish Research Council for Science, Engineering and
Technology under the Embark Initiative scholarship scheme.}



\date{23 May 2013}

\begin{abstract}
In this paper we consider the growth, large fluctuations and memory
properties of an affine stochastic functional differential equation
with an average functional where the contributions of the average
and instantaneous terms are parameterised. An asymptotic analysis of
the solution of this equation is conducted for all values of the
parameters of the equation. When solutions are recurrent, we show
that the autocovariance function of the solution decays at a
polynomial rate, even though the solution is asymptotically equal to
another asymptotically stationary process whose autocovariance
function decays exponentially. It is shown that when solutions grow,
they do so at either a polynomial or exponential rate in time
depending on the sign of a parameter of the model, modulo some
exceptional parameter sets. On these exceptional sets, solutions are
recurrent on the real line with large fluctuations consistent with
the Law of the Iterated logarithm, or exhibit subexponential yet
superpolynomial growth.
\end{abstract}

\maketitle

\section{Introduction and overview}
In this paper, we consider the asymptotic behaviour of an affine
scalar stochastic functional differential equation where the average
of the process over its entire history appears on the right--hand
side. Accordingly, we study
\begin{equation} \label{eq.stochintro}
dX(t)=\left(aX(t)+b\frac{1}{1+t}\int_{-1}^t X(s)\,ds\right)\,dt +
\sigma\,dB(t), \quad t\geq 0,
\end{equation}
where $X$ is given by the continuous function $\psi$, defined on
$[-1,0]$, $B$ is a standard one--dimensional Brownian motion and
$\sigma\neq 0$. Here $a$ and $b$ are real parameters. There is a
unique strong solution of \eqref{eq.stochintro} which is a Gaussian
process. The goal of the paper is to describe for all pairs of the
parameters $a$ and $b$ the asymptotic behaviour of the paths, as
well as information about the autocovariance function of $X$ in the
case that the solution is recurrent on $\mathbb{R}$.

\subsection{Organisation of results and methods of proof}
This model was studied in \cite{App_Dan}, where the authors
considered the case $a>0$. Under this condition the solution $X$ was
shown to grow at a well--defined exponential rate, with a polynomial
correction. Specifically, the rate of growth is given by
\begin{equation}\label{eq:avXexp}
    \lim_{t\to\infty} \frac{X(t)}{e^{at}t^{b/a}}=C, \quad\text{a.s.}
\end{equation}
where $C$ is an almost surely finite and Gaussian distributed random
variable. The results in \cite{App_Dan} rely on the theory of
admissibility of linear deterministic Volterra operators.

While the present article establishes some new results concerning
the case when $a>0$, for the most part it is concerned with the case
when $a\leq0$, where the solution need not have a well-defined
growth rate but rather may fluctuate. This behaviour is not wholly
unexpected; in the case when $a<0$ and $b=0$, for example, the
solution of \eqref{eq.stochintro} is an asymptotically stationary
Ornstein--Uhlenbeck process, while when $a=0$ and $b=0$, it is a
scaled standard Brownian motion.

A complete asymptotic dynamical picture of the solution $X$ is
determined for all real values of $a$ and $b$ in the paper. Our
analysis shows that there are only three principal regions in the
`$a-b$' parameter space, within which the process $X$ undergoes
different pathwise asymptotic behaviour. For clarity we provide a
bifurcation diagram of the parameter space:
\begin{center}
    \includegraphics[height = 6.5cm]{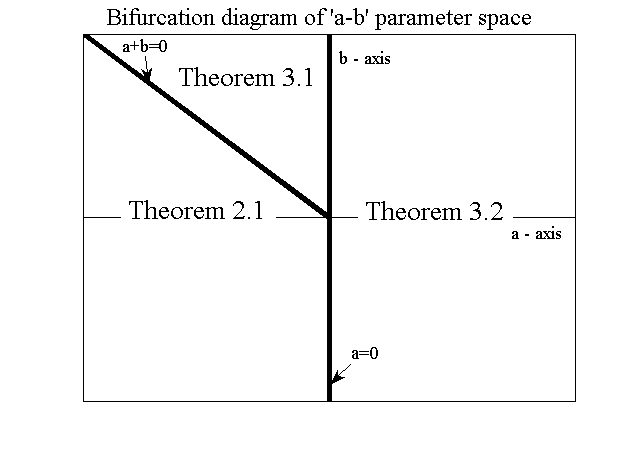} \hspace{.1cm}
\end{center}
\vspace{-0.5cm}
\begin{itemize}
\item In Theorem~\ref{thm.averagegaussian}, corresponding to $a<0$ and $a+b\leq0$, the solution $X$ is asymptotically equal to an
Ornstein--Uhlenbeck process and has oscillations of magnitude
described by
\begin{equation*} 
\limsup_{t\to\infty} \frac{X(t)}{\sqrt{2\log
t}}=\frac{\sigma}{\sqrt{2|a|}}, \quad \liminf_{t\to\infty}
\frac{X(t)}{\sqrt{2\log t}}=-\frac{\sigma}{\sqrt{2|a|}},
\quad\text{a.s.}
\end{equation*}
\item In Theorem~\ref{cor:stochexp}, corresponding to $a<0$ and $a+b>0$, the solution $X$ tends to plus or minus infinity at a polynomial rate
\begin{equation} \label{eq:stochpolygrowthintro}
    \lim_{t\to\infty}\frac{X(t)}{t^{-(1+\ba)}}=C, \quad \text{a.s.}
\end{equation}
    where $C$ is an almost surely finite proper random variable.
\item In Theorem~\ref{thm.bubbleaveexp}, corresponding to $a>0$, the solution $X$ is shown to obey \eqref{eq:avXexp}
\item In Theorem~\ref{thm:modBes}, corresponding to $a=0$ and $b>0$, the solution grows at a rate which is faster than the polynomial
growth of \eqref{eq:stochpolygrowthintro} yet slower than the
exponential growth given by \eqref{eq:avXexp}.
\item In Theorem~\ref{thm:Bes}, corresponding to $a=0$ and $b<0$, the solution $X$ is recurrent on $\mathbb{R}$ and
its largest fluctuations are described by a result reminiscent of
the Law of the Iterated Logarithm.
\end{itemize}
In analysing the solution of the stochastic equation it is helpful
first to ask how the underlying deterministic equation behaves
asymptotically. This deterministic equation is attained from
\eqref{eq.stochintro} by letting $\sigma=0$. The solution of this
underlying equation (which corresponds to the mean of $X$) may be
expressed in terms of confluent hypergeometric, modified Bessel and
Bessel functions. Properties of these special functions are
well--documented, c.f. e.g. \cite{abramstegun, olver:asy, NIST}. An
associated differential resolvent may also be decomposed in terms of
these special functions. In
Theorems~\ref{cor:stochexp},~\ref{thm.bubbleaveexp} and
\ref{thm:modBes} the asymptotic behaviour of the solution $X$ may
then be shown to mirror that of the deterministic equations, i.e.
the asymptotic rates of growth or decay of the solutions of the
deterministic equations are preserved under the addition of a
stochastic perturbation. However, as
Theorem~\ref{thm.averagegaussian} demonstrates the stochastic
perturbation can for particular values of the parameters produce
asymptotic behaviour which is distinct from that of the solution of
the associated deterministic equation.
The analysis is achieved via this decomposition of the resolvent and
a variation of parameters formula.

Many of the asymptotic results concern pathwise behaviour. However,
many of the growth results also hold true in mean or in mean square.
Furthermore, in the main case where there are fluctuations (i.e.,
when $a<0$ and $a+b\leq 0$), we show that the autocovariance
function of the process $X$ decays at a polynomial rate in time,
i.e. for any fixed $t>0$,
\[
    \lim_{\Delta\to\infty}\frac{\gamma_t(\Delta)}{\Delta^{-1-\ba}} =c_t\in(0,\infty),
\]
where $\gamma_t(\cdot)=\Cov[X(t),X(t+\Delta)]$. Thus $X$ may be
viewed as possessing {\it long memory}, in the sense that for any
fixed $t$,
\begin{equation*}
    \int_0^\infty \gamma_{t,X}(\Delta)\,d\Delta=+\infty, \quad a<0, \quad b>0, \quad a+b<0.
\end{equation*}
This result is all the more striking as Theorem
\ref{thm.averagegaussian} proves that $X$ is asymptotically equal to
a process whose autocovariance function decays exponentially
quickly, i.e. a ``short memory`` process. Moreover, it can be shown
that $X$ is transiently non--stationary, and has limiting
autocovariance function equal to that of the stationary Ornstein
Uhlenbeck process to which it converges pathwise. We comment more on
the this result in the next section.
%

\subsection{Motivation for the work}
One of the motivations of this work is to develop a parameterised
stochastic functional differential equation whose asymptotic
behaviour is completely characterised, as such an equation can act
as a test equation for simulation methods for SFDEs. Another
mathematical motivation is to demonstrate that the general approach
of admissibility theory developed in \cite{App_Dan} can generate the
same results as the special function theory outlined here (at least
in some cases), thus supporting the conjecture that it can prove a
sharp tool in studying the asymptotic behaviour of linear,
quasilinear or affine stochastic functional differential equation.

However, one of the main interests in examining this equation is to
gain insight into some features of price dynamics in inefficient
financial markets. First, we argue that \eqref{eq.stochintro} may be
considered as a simple model of such a market.
Suppose that there is a class of technical analysts who compare the
current returns of a risky asset with the average of historical
returns. This leads to an instantaneous excess demand of
\[
\alpha\left(X_1(t)-\frac{1}{1+t}\int_{-1}^t X_1(s)\,\rd s\right)
\]
per unit time at time $t$. A class of feedback traders compare the
returns to a reference level $\bar{X}$, leading to an instantaneous
excess demand of
\[
\beta(X_1(t)-\bar{X})
\]
per unit time at time $t$. Unplanned demand by the traders arises
from "news", where the news in each period is independent of that in
previous periods. The contribution of this news to overall excess
demand is $\sigma(B(t_2)-B(t_1))$ over the time interval
$[t_1,t_2]$, where $B$ is a standard one--dimensional Brownian
motion. If we presume that returns respond linearly to the excess
demand of the market, then $X_1$ obeys the stochastic functional
differential equation
\[
\rd X_1(t)=\left( \alpha\left(X_1(t)-\frac{1}{1+t}\int_{-1}^t
X_1(s)\,\rd s\right) +\beta(X_1(t)-\bar{X})  \right)\,\rd t +
\sigma\, \rd B(t),
\]
for $t\geq 0$, where $X_1(t)=\psi_1(t)$ for $t\in [-1,0]$. The price
of the risky asset at time $t\geq 0$ is denoted by $S(t)$ and
defined by
\[
\rd S(t)=\mu S(t)\,\rd t + S(t)\,\rd X_1(t), \quad t\geq 0
\]
with $S(0)=s_0$. Now define $X(t)=X_1(t)-\bar{X}$ for $t\geq 0$ and
$\psi(t)=\psi_1(t)-\bar{X}$ for $t\in [-1,0]$. Then $X$ obeys
\eqref{eq.stochintro} with $a=\alpha+\beta$ and $b=-\alpha$.

Motivation and literature for such models, as well as alternative
inefficient market models may be found in~\cite{jamrcs10}, in which
a market with finite memory is considered. In common
with~\cite{jamrcs10}, in this work $X_1$ can grow to plus or minus
infinity, with both events being possible. In terms of the
mathematics, this happens if and only if
\begin{itemize}
\item  $a>0$;
\item $a<0$ and $a+b>0$;
\item $a=0$ and $b>0$.
\end{itemize}
From an economic perspective, the first case corresponds to the
situation where the feedback traders chase trends ($\alpha>0$) and
either dominate the fundamental investors, who have mean--reverting
expectations about price movements ($\alpha+\beta>0$, $\beta<0$) or
both classes of agents have trend chasing type expectations
($\alpha>0$, $\beta>0$). The other two cases, while interesting
mathematically, are less likely within the scope of the model: the
second case requires $\beta>0$, which implies that fundamental
investors are bullish about higher than average returns, but
$\alpha+\beta<0$, which indicates these investors dominate the
technical traders, who now have mean reverting expectations about
returns. Nonetheless, this case serves to demonstrate that if at
least one of the investor classes believes that high and rising
returns are a signal of higher returns in the future, and that that
class of agent dominates, then bubbles are likely outcomes. The
third case occurs if the two classes of traders have equal strength,
 ($\beta+\alpha=0$), with the technical traders having mean reverting expectations, and the fundamental investors being bullish about
 higher than average returns ($\alpha<0$, $\beta>0$).

In all these cases, the limiting random variable is path dependent,
so it follows that the initial behaviour of the market determines
whether there is a bubble or a crash. This picture is consistent
with the mechanism proposed for the formation of bubbles with those
formed in models of mimetic contagion, first introduced by
Orl\'ean~\cite{orlean}.

From a modelling and time series perspective, the behaviour in the
''non--bubble'' case when $a<0$ and $a+b\leq 0$, or $a=0$ and $b<0$
is also of interest. The former corresponds to the situation where
$\alpha+\beta<0$, $\beta\leq 0$, in which the fundamental investors
have mean reverting expectations, and either dominate the technical
investors (if they have trend chasing expectations) or the technical
traders also have mean reverting expectations themselves.
In this case as we observed the size of the largest fluctuations of
the process is given by $\sigma/\sqrt{2|\alpha+\beta|}$. Thus as the
process is actually mean reverting in this scenario it is in the
interests of the trend chasing traders to ensure that $\alpha+\beta$
is as close to zero as possible so that the the process undergoes as
large fluctuations as possible. This phenomenon is observed in
financial markets, i.e. when there is a large proportion of
uninformed investors in a market then the volatility of the market
tends to be higher than in their absence c.f. e.g. De Long et al. \cite{delong}.  
If however the uninformed investors where to force $\alpha+\beta>0$ then this,
as already observed, will result in the formation of an
uncontrollable bubble.

The case when $a=0$ and $b<0$ is consistent with solutions obeying
the law of the iterated logarithm, and so may be roughly associated
with Gaussian processes that are non--stationary, but possess
stationary increments. However, in the former case, not only (as we
have already pointed out) is $X$ is asymptotically indistinguishable
from an asymptotically stationary process, it can be shown that $X$
itself is asymptotically stationary (or {\it transiently
non--stationary}) ,i.e.
\[
    \lim_{t\to\infty}\text{Cov}(X(t),X(t+\Delta)) = \gamma(\Delta),
\]
for some function $\gamma:\R\to\R$. Moreover this limiting
autocovariance, as a function of $\Delta$, decays exponentially and
so is indicative of a short memory process. At the same time, we
have already seen that when $t$ is fixed and $\Delta\to \infty$,
then $\Delta \to \text{Cov}(X(t),X(t+\Delta))$ tends to zero at a
polynomial rate, and is indeed non--integrable when $b>0$. In a
sense therefore, the process exhibits ``long--memory'' and
``short--memory'' characteristics. Of course, it is not unheard of
that reversing the order of these limits leads to different answers,
and while this is an interesting mathematical example of this
phenomenon, it is otherwise not noteworthy. However, given that
there is considerable debate among empiricists in finance concerning
the presence or absence of long memory in certain financial time
series, it is interesting to note that this paper presents an
asymptotically stationary process in a (highly simplified, indeed
unrealistic) market model, which also possesses somewhat ambiguous
memory properties.

\subsection{Organisation of the paper and mathematical preliminaries}\label{sb:avprelim}
The article is organised as follows. In this section
(Section~\ref{sb:avprelim}), we formally introduce the equation
under scrutiny and define some notation. Section~\ref{sec.specialfn}
gives a detailed description of the decomposition of the solution of
the deterministic equation into special functions, and in particular
details the differing functions which are used depending on the
values of $a$ and $b$. In order to make our presentation
self--contained, various properties of these functions which are
needed in the analysis of the asymptotic behaviour, are listed.
Section~\ref{sect:rec} deals with recurrent dynamics of $X$, with
Subsection~\ref{sect:recpath} giving results on the almost sure
pathwise asymptotic behaviour of the process, while
Subsection~\ref{sect:recauto} discusses the memory properties when
$X$ has these recurrent dynamics. Section~\ref{sect:trans} gives
results concerning transient dynamical behaviour of the process.
Proofs of the results are deferred to
Section~\ref{sect:avcntsproof}.

Let $\mathbb{R}$ denote the real numbers. If $x\in\mathbb{R}$, then
$\lceil x\rceil$, or the \emph{ceiling} of $x$ is the smallest
integer greater than or equal to $x\in\mathbb{R}$. The Wronskian for
two any functions $x_1$ and $x_2$, is defined as $\mathcal{W}(t) =
x_1(t)x_2'(t)-x_1'(t)x_2(t)$. We define the Gamma function
$\Gamma:\mathbb{C}\to\mathbb{C}$ according to
$\Gamma(z)=\int_{0}^{\infty}s^{z-1}\e^{-s}\rd s$  for $\Re(z)>0$ .
When $\Re(z)\leq0$, $\Gamma(z)$ is defined by analytic continuation.
We employ the standard Landau notation: if
$f:\mathbb{C}\to\mathbb{C}$ and $g:\mathbb{C}\to\mathbb{R}$, we
write $f=O(g)$ as $|z|\to\infty$ if there exist $z_0>0$ and $M>0$
such that $|f(z)|\leq M|g(z)|$ for all $|z|>z_0$. while $f\sim g$ as
$z\to z_0$ is equivalent to $\lim_{z\to z_0}f(z)/g(z)=1$.

Let us fix a complete probability space
$(\Omega,\mathcal{F},\mathbb{P})$ with a filtration
$\{\mathcal{F}(t)\}_{t\geq 0}$ satisfying the usual conditions and
let $B=\{B(t):\,t\geq 0\}$ be a one--dimensional Brownian motion
adapted to $\{\mathcal{F}(t)\}_{t\geq 0}$ on this space. The
probability measure induces an expectation $\mathbb{E}$ in the usual
manner, in the sense that if $Y$ is an $\mathcal{F}$--measurable
random variable such that $\int_{\Omega} |Y(\omega)| \,\rd
\mathbb{P}\{\omega\}<+\infty$, then $\mathbb{E}[Y]=\int_{\Omega}
Y(\omega) \,\rd \mathbb{P}\{\omega\}$. In this paper, the
abbreviation \emph{a.s.} stands for ``almost sure'' or ``almost
surely''.

We consider the affine scalar stochastic functional differential
equation with an average functional
\begin{subequations}\label{eq.longmemorysdde}
\begin{gather} \label{eq.longmemorysddeeq}
dX(t)=\left(aX(t)+b\frac{1}{1+t}\int_{-1}^t X(s)\,ds\right)\,dt + \sigma\,dB(t), \quad t\geq0;\\
\label{eq.longmemorysddeic} X(t)=\psi(t), \quad t\in[-1,0],
\end{gather}
\end{subequations}
Here $\sigma>0$, $a$, $b\in\mathbb{R}$ and $\psi\in
C([-1,0],\mathbb{R})$. Then by Berger and Mizel~\cite{BergMiz} or
Mao~\cite[Theorem~2.3.1]{Mao2} there is a unique continuous adapted
process which obeys  \eqref{eq.longmemorysdde}, hereinafter referred
to as the \emph{solution} of \eqref{eq.longmemorysdde} and denoted
$X$. There is also a unique continuous solution of
    \begin{subequations} \label{eq.detaverage}
    \begin{gather} \label{eq.detaveragedde}
        x'(t) = a x(t) + b\frac{1}{1+t} \int_{-1}^t x(s)\,ds,\quad t\geq 0, \\
        \label{eq.detaverageinitcondn}
        x(t)  =  \psi(t),\quad t \in [-1,0].
    \end{gather}
    \end{subequations}
The differential resolvent $r$ associated with \eqref{eq.detaverage}
is defined according to
\begin{subequations} \label{eq.resolv}
\begin{gather} \label{eq.resolveq}
\frac{\partial r}{\partial t}(t,s)=a\,r(t,s)+b\frac{1}{1+t}\int_{s}^t r(u,s)\,du, \quad t>s; \\
\label{eq.resolvic} r(t,s)=0, \quad t<s; \quad r(s,s)=1.
\end{gather}
\end{subequations}
Then with $x$ being the solution of \eqref{eq.detaverage}, the
solution of \eqref{eq.longmemorysdde} has a variation of parameters
representation.
\begin{lemma}\label{lm:avXsol}
Suppose that  $\psi\in C([-1,0];\mathbb{R})$. Let $X$ be the unique
solution of \eqref{eq.longmemorysdde}, $x$ the unique solution of
\eqref{eq.detaverage} and $r$ the unique solution of
\eqref{eq.resolv}. Then $X$ is a Gaussian process and obeys
\begin{equation} \label{eq.xrep}
X(t)=x(t)+\sigma \int_0^t r(t,s)\,dB(s), \quad t\geq 0.
\end{equation}
\end{lemma}
A proof of the validity of this representation is provided in
Section~\ref{sect:avproof}.

Using the representation \eqref{eq.xrep} for $X$, we deduce formulae
for the mean and autocovariance of $X$. By considering for $t\geq 0$
fixed and $\tau\geq 0$ the process
\[
M(\tau)=\int_0^\tau r(t,s)\,dB(s), \quad \tau \geq 0,
\]
we can see that $M$ is a martingale and moreover a Gaussian process,
so therefore $X(t)=x(t)+M(t)$ is Gaussian distributed. Since
$\mathbb{E}[M(\tau)^2]<+\infty$ for all $\tau\geq 0$, we have that
$\mathbb{E}[M(\tau)]=0$ for all $\tau\geq 0$, and hence
$\mathbb{E}[M(t)]=0$. Hence
\begin{equation}\label{eq.xmean}
\mathbb{E}[X(t)]=x(t), \quad t\geq 0.
\end{equation}
Since $\mathbb{E}[X(t)^2]$ is finite for all $t\geq 0$, it follows
that $\text{Cov}(X(t),X(t+\Delta))$ is well--defined for all $t\geq
0$ and $\Delta\geq 0$. We also see that
\begin{align*}
\text{Cov}(X(t),X(t+\Delta))&=
\sigma^2\mathbb{E}[M(t)M(t+\Delta)]\\
&=\sigma^2\mathbb{E}[\int_0^{t+\Delta} r(t,s)\chi_{[0,t]}(s)\,dB(s)
\int_0^{t+\Delta} r(t+\Delta,s)\,dB(s)].
\end{align*}
Considering $t$ and $\Delta$ as fixed, we may apply It\^{o}'s
isometry to obtain the variance of
\begin{multline*}
V_1:=\int_0^{t+\Delta} r(t+\Delta,s)\,dB(s), \quad V_2:=\int_0^{t+\Delta} r(t,s)\chi_{[0,t]}(s)\,dB(s) \quad \text{and}\\
\int_0^{t+\Delta}
\{r(t,s)\chi_{[0,t]}+r(t+\Delta,s)\}\,dB(s)=V_1+V_2,
\end{multline*}
and using the fact that
$2\text{Cov}(V_1,V_2)=\text{Var}[V_1+V_2]-\text{Var}[V_1]-\text{Var}[V_2]$,
we obtain
\begin{equation} \label{eq.autocovarianceaverage}
\Cov(X(t),X(t+\Delta))=\sigma^2 \int_0^t r(t,s)r(t+\Delta,s)\,ds,
\quad t\geq 0, \quad \Delta\geq 0.
\end{equation}

We have already seen that mean and resolvent obey functional
differential equations involving an average functional. This also
holds true for the autocovariance function, and the result is
recorded below.
\begin{proposition}\label{lm:acf}
Suppose that  $\psi\in C([-1,0];\mathbb{R})$. Let $X$ be the unique
solution of \eqref{eq.longmemorysdde} and $r$ the unique solution of
\eqref{eq.resolv}. Fix $t\geq0$ and define
    \begin{equation}\label{eq.acfdef}
        \g_{t}(\dl):=\sigma^2\int_{0}^{t}r(t,s)r(t+\dl,s) \, ds, \quad \dl\geq-t.
    \end{equation}
    If $\dl\geq0$, then $\g_{t}(\dl)=\Cov(X(t),X(t+\dl))$
    and, 
    \begin{align}\label{eq.acfdde}
        \g_{t}'(\dl)&=a\g_{t}(\dl) + \frac{b}{1+t+\dl}\int_{-t}^{\dl}\g_{t}(w)\,dw, \quad \dl\geq0, \\
      \g_{t}'(\dl)&=a\g_{t}(\dl) + \frac{b}{1+t+\dl}\int_{-t}^{\dl}\g_{t}(w)\,dw + \sigma^2r(t,t+\Delta), \quad -t\leq\dl<0.
      \label{eq.acfdde2}
    \end{align}
\end{proposition}
This result is proven in Section~\ref{sect:avcntsproof}. The
differential equation \eqref{eq.acfdde} may be thought of as a
Yule--Walker--type representation of the autocovariance function.

In this work, we could equally have studied the equation
\[
dX(t)=\left(aX(t)+b\frac{1}{t}\int_0^t X(s)\,ds\right)\,dt +
\sigma\,dB(t), \quad t\geq 0; \quad X(0)=\xi.
\]
However, this equation is more delicate to analyse, on account of
the potential singularity in the average functional at $t=0$. We
obviate such complications by considering an equation with an
initial history on a non--trivial compact interval. Taking this to
be  $[-1,0]$ leads to \eqref{eq.longmemorysdde}.


\section{Formulae and Asymptotic Behaviour of  Solutions of \eqref{eq.detaverage} and \eqref{eq.resolv}} \label{sec.specialfn}
The solution of \eqref{eq.detaverage} can be rewritten as the
solution of an initial value problems for a second--order
differential equation. The equation is
\begin{subequations} \label{eq:secondorderplusic}
\begin{gather}\label{eq:secondorder}
x''(t)+\left(\frac{1}{1+t}-a\right)x'(t)-\frac{a+b}{1+t}x(t)=0, \quad t\geq0; \\
\label{eq:secondorderic} x(0)=\psi(0), \quad
x'(0)=a\psi(0)+b\int_{-1}^0\psi(s)\,\rd s.
\end{gather}
\end{subequations}
There are three cases to consider: $a<0$, $a>0$ and $a=0$. We
discuss each case and their subcases, conditioned by $b$, in turn.
In the case when $b=0$, the stochastic differential equation
\eqref{eq.longmemorysdde} reduces to an Ornstein--Uhlenbeck SDE, and
so the behaviour of $x$, $r$, and indeed $X$, are well--understood.
Therefore, we exclude the case $b=0$ from our analysis.
In the exposition below the asymptotic behaviour of the solution of (6.2.1) is deduced from
the known asymptotic behaviour of certain functions. It is here observed however that a general
theory concerning the asymptotic behaviour of linear second order equations with analytic
coefficients may be found in e.g. \cite[Ch. 7.1 and 7.2]{olver:asy}.


\subsection{$a<0$}
When $a<0$, the solution of \eqref{eq:secondorderplusic} can be
expressed in terms of two linearly independent confluent
hypergeometric functions, according to:
\begin{equation}\label{eq:detspec1}
        x(t) = c_{1}r_{1}(t) + c_{2}r_{2}(t) \quad \text{ for $a<0$ and $b/a\not\in\{1,2,...\}$} \\
\end{equation}
where
\begin{align*}
    r_{1}(t) = \e^{at}U(-b/a,1,-a(1+t)), \quad r_{2}(t) = \e^{at}M(-b/a,1,-a(1+t)).
\end{align*}
Here $U(\alpha,\beta,\cdot)$ and $M(\alpha,\beta,\cdot)$ are two
linearly independent solutions of \emph{Kummer's differential
equation} which is given by
\[
    z w''(z) + (\beta-z)w'(z) - \alpha w(z)=0,
\]
where $\alpha$ and $\beta$ are real and $z$ a complex number. $M$ is
sometimes referred to as Kummer's function (of the first kind) or a
\emph{confluent hypergeometric function}, while $U$ is sometimes
called the Tricomi confluent hypergeometric function. See
\cite[Chapter 13.2.1]{NIST} and following sections.

To see that $r_1$ and $r_2$ are solutions of \eqref{eq:secondorder},
observe that as $z\mapsto U(\alpha,\beta,z)$ is a solution of
Kummer's equation then $t\mapsto U(-b/a,1,-a(1+t))$ satisfies
\begin{align*}
    &-a(1+t)U''(-b/a,1,-a(1+t)) + (1+a(1+t))U'(-b/a,1,-a(1+t)) \\
    &\qquad+ \ba U(-b/a,1,-a(1+t)) =0.
\end{align*}
Therefore
\begin{align*}
r_1''&(t) + \left(\frac{1}{1+t}-a\right)r'(t) -\frac{a+b}{1+t}r_1(t) \\
&= a^2\e^{at}U''(-\ba,1,-a(1+t)) - 2a^2\e^{at}U'(-\ba,1,-a(1+t)) \\
&\qquad + a^2\e^{at}U(-\ba,1,-a(1+t)) - \frac{a+b}{1+t}U(-\ba,1,-a(1+t)) \\
&\qquad+ \left(\frac{1}{1+t}-a\right)\left(-a\e^{at}U'(-\ba,1,-a(1+t))+a\e^{at}U(-\ba,1,-a(1+t))\right) \\
&=\frac{-a\,\e^{at}}{1+t}\biggl[ -a(1+t)U''(-\ba,1,-a(1+t)) + (1+a(1+t))U'(-\ba,1,-a(1+t)) \\
&\qquad + \ba U(-\ba,1,-a(1+t)) \biggr] = 0,
\end{align*}
as required. A similar calculation shows that $r_2$ is a solution of
\eqref{eq:secondorder}.

As we are chiefly interested in the long--run behaviour of $X$ it is
necessary to have information on the asymptotic behaviour of both
$U$ and $M$. This is given by \cite[13.1.4 \& 13.1.8]{abramstegun},
or
\begin{subequations}\label{eq:chg3}
\begin{align}\label{eq:chg3a}
    M(\alpha,\beta,t) &= \frac{\Gamma(\beta)}{\Gamma(\alpha)}\e^{t}t^{\alpha-\beta}[1+O(t^{-1})], \quad \text{as $t\to\infty$,} \\
    U(\alpha,\beta,t) &= t^{-\alpha}[1+O(t^{-1})], \quad \text{as $t\to\infty$}. \label{eq:chg3b}
\end{align}
\end{subequations}
This immediately gives asymptotic information about $r_1$ and $r_2$:
\begin{align} \label{eq.r1asy}
r_1(t)&\sim  e^{at} |a|^{b/a}t^{b/a}, \text{ as $t\to\infty$},
\\
r_2(t)& \sim \frac{1}{\Gamma(-b/a)} e^{-a} |a|^{-b/a-1} t^{-b/a-1},
\text{ as $t\to\infty$} \label{eq.r2asy}
\end{align}

To determine the asymptotic behaviour of $x$, we need values for
$c_1$ and $c_2$ in \eqref{eq:detspec1} in terms of the initial
conditions of \eqref{eq:secondorder}. As usual, by using
\eqref{eq:secondorderic}, these values are obtained by solving
\begin{equation} \label{eq.c1c2eqn}
    c_1 r_1(0) + c_2 r_2(0) = \psi(0), \quad c_1 r_1'(0) + c_2 r_2'(0)= a\psi(0)+b\int_{-1}^0\psi(s)\,\rd s.
\end{equation}
Clearly, these values can be expressed in terms of the Wronskian of
$r_1$ and $r_2$, evaluated at $t=0$, as well as the derivatives of
$r_1$ and $r_2$. Since $r_1$ and $r_2$ depend on $M$ and $U$, it is
of value to have a general formula for the Wronskian and the
derivatives of $U$ and $M$. A formula for the Wronskian, $W$, of $M$
and $U$ is given by \cite[13.2.34]{NIST}:
\begin{equation}\label{eq:chg4}
    W\{M(\alpha,\beta,z),U(\alpha,\beta,z)\}= -\Gamma(\beta)z^{-\beta}\e^{z}/\Gamma(\alpha).
\end{equation}
Expressions for the derivatives of $U$ and $M$ are given by
\cite[13.3.15 \& 13.3.22]{NIST}:
\begin{align}\label{eq:dMU}
    M'(\alpha,\beta,z) = \frac{\alpha}{\beta}M(\alpha+1,\beta+1,z), \quad   U'(\alpha,\beta,z) = -\alpha U(\alpha+1,\beta+1,z).
\end{align}
Using these results, we obtain the following formulae for $c_1$ and
$c_2$:
\begin{align}
    c_{1}&= \Gamma(-\ba)\e^{a} b\left( \psi(0)M(1-\ba,2,-a) - \int_{-1}^{0}\psi(s) \rd s \, M(-\ba,1,-a) \right), \nonumber\\
\label{eq.c2}
    c_{2}&= \Gamma(-\ba)\e^{a} b\left( \psi(0)U(1-\ba,2,-a) + \int_{-1}^{0}\psi(s) \rd s \, U(-\ba,1,-a) \right).
\end{align}


We now consider the case when $b/a\in\{1,2,\dots\}$. As alluded to
earlier, in this case
    $t\mapsto M(-b/a,1,-a(1+t))$ and $t\mapsto U(-b/a,1,-a(1+t))$ are linearly dependent, and therefore the representation \eqref{eq:detspec1} for $x$ is not valid. It is however known that $t\mapsto U(-b/a,1,-a(1+t))$ is a polynomial in $|a|(1+t)$ of degree $b/a$.
We even have an explicit formula for this polynomial. Indeed, for
$n\in\{0,1,2,...\}$, we have from \cite[13.2.7]{NIST} that
\begin{equation}\label{eq:chg5}
    U(-n,1,z) = (-1)^{n}\sum_{j=0}^{n}\frac{(n!)^2}{(n-j)!(j!)^2}(-z)^j.
\end{equation}
Note that $z\mapsto U(-n,1,z)$ is analytic, and so its  (at most
$b/a$) zeros are isolated. Therefore, the zeros of the real--valued
polynomial $t\mapsto U(-b/a,1,-a(1+t))$ are also isolated.

Suppose now we take $r_1(t)=e^{at}U(-b/a,1,-a(1+t))$ for $t\geq 0$.
 We know from standard theory (cf., e.g.~\cite{BoyceDiPrima})
that there exists a second solution, $\tilde{r}_2$, of
\eqref{eq:secondorder} which is linearly independent of $r_1$.
Next, 
 by Abel's Theorem (cf., e.g. \cite[Ch.3.3.2]{BoyceDiPrima}), the Wronskian of $r_1$ and $\tilde{r}_2$, which is  associated with
 \eqref{eq:secondorder} obeys
 \[
    \mathcal{W}(a,b,t) = 
    \mathcal{W}(a,b,0)\e^{at}(1+t)^{-1}, \quad t\geq 0,
 \]
 where $\mathcal{W}(a,b,0)=r_1(0)\tilde{r}_2'(0)-r_1'(0)\tilde{r}_2'(0)\neq 0$.
%

This expression is equivalent to
\[
    r_1(t)\tilde{r}'_2(t) - r'_1(t)\tilde{r}_2(t) = \mathcal{W}(a,b,0)\e^{at}(1+t)^{-1}, \quad t\geq0.
\]
We now wish to find a representation for $\tilde{r}_2$ which allows
us to deduce its asymptotic properties.

Notice that because $r_1$ has finitely many zeros, it must have a
maximal real zero. Let
$t_1=1+\max(0,\sup\{t\in\mathbb{R}:r_1(t)=0\})$, where we define
$\sup\{t\in\mathbb{R}:r_1(t)=0\}=-\infty$ if $r_1(t)\neq 0$ for all
$t\geq 0$. Then for $t\geq t_1$ we have
\begin{equation} \label{eq.r2firstorder}
    \tilde{r}'_2(t) - \frac{r'_1(t)}{r_1(t)}\tilde{r}_2(t) = \mathcal{W}(a,b,0)\frac{\e^{at}(1+t)^{-1}}{r_1(t)}, \quad t\geq t_1.
\end{equation}
Since $r_1(t)\neq 0$ for all $t\geq t_1$, we have that $t\mapsto
r_1'(t)/r_1(t)$ and $t\mapsto e^{at}(1+t)^{-1}/r_1(t)$ are
continuous on $[t_1,\infty)$, and therefore we may solve
\eqref{eq.r2firstorder} for $\tilde{r}_2$ to obtain the following
representation for $\tilde{r}_2$ on $[t_1,\infty)$:
 \begin{equation}\label{eq:linr2}
    \tilde{r}_{2}(t) = r_{1}(t) \frac{\tilde{r}_2(t_1)}{r_1(t_1)}
    + \mathcal{W}(a,b,0)r_{1}(t)\int_{t_1}^{t}\frac{\e^{as}(1+s)^{-1}}{r_{1}^{2}(s)} \rd s,
    \quad t\geq t_1.
\end{equation}
Since $t_1$ exceeds the maximal zero of $r_1$, the integral on the
right hand side of \eqref{eq:linr2} is well--defined
 for $t\geq t_1$.
Moreover, using l'H\^{o}pital's rule together with \eqref{eq:chg3b}
or \eqref{eq:chg5}, one may show that
\begin{equation}\label{eq:secr2}
    \lim_{t\to\infty} t^{1+\ba}\tilde{r}_2(t) = \mathcal{W}(a,b,0) |a|^{-1-\ba}, \quad a<0, \quad -\ba\in \{1,2,\ldots\}.
\end{equation}
Note that this recovers the asymptotic behaviour of $r_2$ above in
\eqref{eq.r2asy} the case $a<0$ and $b/a\not\in \{1,2,\ldots\}$.

It is also useful to determine some asymptotic information about
$\tilde{r}_2'$. Notice that $r_1(t)\sim e^{at}t^{b/a}|a|^{b/a}$ as
$t\to\infty$. Also we have
\[
\frac{r_1'(t)}{r_1(t)}-a =
\frac{r_1'(t)-ar_1(t)}{r_1(t)}=\frac{-aU'(-b/a,1,-a(1+t))}{U(-b/a,1,-a(1+t))},
\quad t\geq t_1,
\]
so using the fact that $t\mapsto U(-b/a,1,-a(1+t))$ is a polynomial
of degree $b/a\in \mathbb{N}$, we have that $\lim_{t\to\infty}
r_1'(t)/r_1(t)=a$. By \eqref{eq:secr2}, it follows that there is
$t_1'>0$ such that $\tilde{r}_2(t)\neq 0$ for all $t\geq t_1'$. Let
$t_1''=\max(t_1',t_1)$. Then we may rewrite \eqref{eq.r2firstorder}
for $t\geq t_1''$ to get
\begin{equation*}
    \frac{\tilde{r}'_2(t)}{\tilde{r}_2(t)} = \frac{r'_1(t)}{r_1(t)}
    + \mathcal{W}(a,b,0)\frac{\e^{at}(1+t)^{-1}}{r_1(t)\tilde{r}_2(t)}.
\end{equation*}
Using the fact that $r_1(t)\sim e^{at}t^{b/a}|a|^{b/a}$ as
$t\to\infty$ together with \eqref{eq:secr2} shows that the second
term has limit $|a|=-a$, and therefore
\begin{equation} \label{eq.r2tilasy}
\lim_{t\to\infty} \frac{\tilde{r}'_2(t)}{\tilde{r}_2(t)}=0.
\end{equation}
Finally, we see that the solution of \eqref{eq:secondorderplusic} is
given by
\begin{equation}\label{eq:detspec2}
    x(t)=\tilde{c}_{1} r_{1}(t) + \tilde{c}_{2} \tilde{r}_{2}(t), \quad t\geq0, \quad \text{ for $a<0$ and $b/a\in\{1,2,...\}$}
\end{equation}
where $\tilde{c}_1$ and $\tilde{c}_2$ are found using
\eqref{eq:secondorderic}.
Note that $\tilde{c}_2$ is known entirely in terms of $r_1$ and its
dependence on $\tilde{r}_2$ is solely through the value of the
Wronskian, because
\[
    \tilde{c}_2 =
    \frac{1}{\mathcal{W}(a,b,0)} \left( b\psi(0)\, U(1-\ba,2,|a|) + b\int_{-1}^{0}\psi(s)\,ds\, U(-\ba,1,|a|) \right).
\]
Note also that for $b=0$, \eqref{eq:detspec2} reduces to $x(t) =
\psi(0)\e^{at}$.

We now turn our attention to the representation of the resolvent $r$
defined by \eqref{eq.resolv}. In a manner similar to the treatment
of the solution $x$ of \eqref{eq.detaverage}, it can be shown for
every fixed $s\geq 0$, the solution $t\mapsto r(t,s)=:r_s(t)$ of the
resolvent equation \eqref{eq.resolv} is also the solution of the
second order differential equation
\begin{equation} \label{eq.rs2ode}
r_s''(t)+\left(\frac{1}{1+t}-a\right)r_s'(t)-\frac{a+b}{1+t}r_s(t)=0,
\quad t\geq s,
\end{equation}
with initial conditions $r_s(s)=1$ and $r_s'(s)=a$. It is to be
noted that \eqref{eq.rs2ode} is the same differential equation as
\eqref{eq:secondorder} apart from the fact that the argument of the
solution is restricted to the interval $[s,\infty)$, a subinterval
of the interval of existence of the equation \eqref{eq:secondorder}.
Therefore, $r(t,s)=r_s(t)$ can be represented as a linear
combination of the linearly independent solutions of
\eqref{eq:secondorder} according to
\begin{equation} \label{eq:resdecomp}
    r(t,s) =
    \begin{cases}
        d_1(s)r_{1}(t) + d_2(s)r_{2}(t), &\quad t\geq s\geq 0, \quad a<0, \quad b/a\not\in\{1,2...\}, \\
        \tilde{d}_1(s)r_{1}(t) + \tilde{d}_2(s)\tilde{r}_{2}(t), &\quad t\geq s\geq 0, \quad a<0, \quad b/a\in\{1,2...\}.
    \end{cases}
\end{equation}
The multipliers $d_1$, $d_2$ etc are $s$--dependent, because initial
data for the problem \eqref{eq.rs2ode} is specified at $s$.
Considering first the non--degenerate case when
$b/a\not\in\{1,2...\}$, it can be seen that expressions for the
coefficients $d_1$ and $d_2$ are obtained from the initial
conditions \eqref{eq.resolvic} and \eqref{eq:chg4}, i.e.
\begin{equation}   \label{eq.d1d2ic}
    d_1(s)r_1(s) + d_2(s)r_2(s)= 1, \quad d_1(s) r_1'(s) + d_2(s) r_2'(s) = a.
\end{equation}
From these equations, and using \eqref{eq:chg4} and \eqref{eq:dMU},
we obtain the formulae
\begin{align}
    d_{1}(s) &= \Gamma(-\ba)(1+s)\e^{a} b \, M(1-\ba,2,-a(1+s)), \label{eq:d1} \\
    d_{2}(s) &= \Gamma(-\ba)(1+s)\e^{a} b \, U(1-\ba,2,-a(1+s)) \label{eq:d2}.
\end{align}
Using the fact that $\Gamma(1-b/a)=-b/a \Gamma(-b/a)$ and employing
\eqref{eq:chg3}, we get
\begin{align}
 d_{1}(s) &\sim  b \, \frac{\Gamma(-\ba)}{\Gamma(1-\ba)} |a|^{-1-\ba}  e^{-as} s^{-\ba}
 =  
 |a|^{-\ba}  e^{-as} s^{-\ba}, \text{ as $s\to\infty$},
 \label{eq:d1asy} \\
 \label{eq:d2asy}
    d_{2}(s) &\sim \Gamma(-\ba)\e^{a} b |a|^{\ba-1}\, s^{\ba} , \text{ as $s\to\infty$}. 
\end{align}
In the degenerate case when $b/a\in\{1,2,...\}$, we have
\begin{align*}
    \tilde{d}_{1}(s) &= \frac{\tilde{r}_2'(s)-a\tilde{r}_2(s)}{\mathcal{W}(a,b,0)\e^{as}(1+s)^{-1}}
    =    -\frac{a}{\mathcal{W}(a,b,0)}\tilde{r}_2(s)(1+s)\e^{-as}\left(1+\frac{1}{-a}\frac{\tilde{r}_2'(s)}{\tilde{r}_2(s)}\right) , \\
    \tilde{d}_{2}(s) &= -\frac{r_1'(s)-ar_1(s)}{\mathcal{W}(a,b,0)\e^{as}(1+s)^{-1}}
    =\frac{1}{\mathcal{W}(a,b,0)}b(1+s)U(1-\ba,2,-a(1+s)). 
\end{align*}
We notice by \eqref{eq:secr2} and \eqref{eq.r2tilasy} that
\begin{equation}\label{eq.d1tilasy}
\tilde{d}_{1}(s) \sim |a|^{-\ba}
 s^{-b/a}\e^{-as}, \quad \text{as $s\to\infty$,}
\end{equation}
which mirrors the asymptotic behaviour for $d_1$ in \eqref{eq:d1asy}
in the non--degenerate case.
As to the asymptotic behaviour of $\tilde{d}_2$, we may use
\eqref{eq:chg3b} to obtain
\begin{equation}\label{eq.d2tilasy}
\tilde{d}_{2}(s) \sim  \frac{1}{\mathcal{W}(a,b,0)}b
|a|^{\ba-1}s^{\ba}
\quad \text{as $s\to\infty$,}
\end{equation}
and so $\tilde{d}_2$ has the same asymptotic behaviour as $d_2$
given in \eqref{eq:d2asy} in the non--degenerate case.
%

Using the fact that 
$\text{Cov}(X(t),X(t+\Delta))$ obeys
\eqref{eq.autocovarianceaverage} for $t\geq 0$ and $\Delta\geq 0$,
and $r(t,s)$ is given by \eqref{eq:resdecomp}, we have
\begin{equation} \label{eq:acvfspec1}
 \text{Cov}(X(t),X(t+\Delta))=
 \begin{cases}
        c_{1,t}r_{1}(t+\Delta) + c_{2,t}r_{2}(t+\Delta), &\quad a<0, \quad b/a\not\in\{1,2...\}, \\
        \tilde{c}_{1,t}r_{1}(t+\Delta) + \tilde{c}_{2,t}\tilde{r}_{2}(t+\Delta), &\quad a<0, \quad b/a\in\{1,2...\},
    \end{cases}
\end{equation}
for $t\geq 0$ and $\Delta\geq 0$, where
\begin{equation} \label{eq:c1c2t}
c_{1,t}=\sigma^2\int_0^t r(t,s)d_1(s)\,ds, \quad c_{2,t}=\sigma^2
\int_0^t r(t,s)d_2(s)\,ds,
\end{equation}
and
\begin{equation}  \label{eq.c1c2tildet}
\tilde{c}_{1,t}=\sigma^2\int_0^t r(t,s)\tilde{d}_1(s)\,ds, \quad
\tilde{c}_{2,t}=\sigma^2 \int_0^t r(t,s)\tilde{d}_2(s)\,ds.
\end{equation}

In order that certain limiting constants in our analysis are
non--zero, we find it useful to employ the following integral
representation of $U$:
\begin{equation}\label{eq:Uint}
    U(\alpha,\beta,t) = \frac{1}{\Gamma(\alpha)}\int_{0}^{\infty}\e^{-tu} u^{\alpha-1} (1+u)^{\beta-\alpha-1} du, \quad \alpha>0.
\end{equation}
 It appears as
\cite[13.4.4]{NIST}.

\subsection{$a>0$}
When $a>0$, the solution of \eqref{eq:secondorder} can be expressed
in terms of confluent hypergeometric functions, according to:
\begin{equation}\label{eq:detspec3}
        x(t) = c_{3}r_{3}(t) + c_{4}r_{4}(t) \quad \text{ for $a>0$ and $b/a\not\in\{-1,-2,...\}$} \\
\end{equation}
where
\begin{align}\label{eq:r3r4}
    r_{3}(t) = U(1+\ba,1,a(1+t)), \quad r_{4}(t) = M(1+\ba,1,a(1+t)).
\end{align}
Using \eqref{eq:chg3b}, we get
\begin{equation} \label{eq.r3asy}
r_3(t)\sim a^{-1-\ba}t^{-1-\ba}, \quad \text{as $t\to\infty$,
$a>0$},
\end{equation}
and using \eqref{eq:chg3a}, we obtain
\begin{equation}  \label{eq.r4asy}
r_4(t)
\sim \frac{1}{\Gamma(1+\ba)} e^a a^{\ba} e^{at} t^{\ba}, \quad
\text{ as $t\to\infty$, $a>0$, $\ba\not\in \{-1,-2,\ldots\}$}
\end{equation}

The initial conditions \eqref{eq:secondorderic} can be used to
determine $c_3$ and $c_4$; the relevant formulae are:
\begin{align}
    c_{3} &= \Gamma(1+\ba)\e^{-a} \left( b\psi(0)M(1+\ba,2,a) - b\int_{-1}^{0}\psi(s) \rd s \, M(1+\ba,1,a) \right), \nonumber\\
\label{eq.c4}
    c_{4} &= \Gamma(1+\ba)\e^{-a} \left( a\psi(0)U(1+\ba,2,a) + b\int_{-1}^{0}\psi(s) \rd s \, U(1+\ba,1,a) \right).
\end{align}
One may verify, as before, that $r_3$ and $r_4$ solve
\eqref{eq:secondorder}. In the determination of these formulae for
$c_3$ and $c_4$, we have used the fact that one may deduce from
Kummer's differential equation
the identities \cite[13.3.13 \& 13.3.14]{NIST}, which are 
\begin{align}
    &(\alpha+1)z M(\alpha+2,\beta+2,z) + (\beta+1)(\beta-z)M(\alpha+1,\beta+1,z) \notag\\
    &\qquad -\beta(\beta+1)M(\alpha,\beta,z) = 0 \label{eq:KM} \\
    &(\alpha+1)z U(\alpha+2,\beta+2,z) + (z-\beta)U(\alpha+1,\beta+1,z) - U(\alpha,\beta,z) = 0. \label{eq:KU}
\end{align}
Moreover, letting $\beta\to 0$ in \eqref{eq:KM} and \eqref{eq:KU}
gives
\begin{align}
    &(\alpha+1)z M(\alpha+2,2,z) - z M(\alpha+1,1,z)- \alpha z M(\alpha+1, 2,z) = 0, \label{eq:KM1}\\
    &(\alpha+1)z U(\alpha+2,2,z) + z U(\alpha+1,1,z) - z U(\alpha+1,2,z) = 0. \label{eq:KU1}
\end{align}
as \cite[13.2.5]{NIST} in conjunction with \cite[5.2.1]{NIST} gives
$\lim_{\beta\to0} \beta M(\alpha,\beta,z)= \alpha z M(\alpha+1,2,z)$
and \cite[13.2.11]{NIST} gives $U(\alpha,0,z) = z U(\alpha+1,2,z)$.

Again, for certain values of $a$ and $b$ (i.e., if $-b/a\in
\{1,2,3\ldots\}$), the two functions on the right--hand side of
\eqref{eq:detspec1} are no longer linearly independent. Nevertheless
the second--order equation \eqref{eq:secondorder} has two linearly
independent solutions $r_3$ (still given by \eqref{eq:r3r4}) and
$\tilde{r}_4$, and so the solution of \eqref{eq.detaverage} obeys
\begin{equation}\label{eq:detspec4}
    x(t)=\tilde{c}_{3} r_{3}(t) + \tilde{c}_{4} \tilde{r}_{4}(t), \quad t\geq0, \quad \text{ for $a>0$ and $b/a\in\{-1,-2,...\}$}.
\end{equation}
By \eqref{eq.r3asy}, $r_3(t)>0$ for all $t$ sufficiently large.
Therefore we may define
$t_2=1+\max(0,\sup\{t\in\mathbb{R}:r_3(t)=0\})$, where
$\sup\{t\in\mathbb{R}:r_3(t)=0\}:=-\infty$ if $r_3(t)\neq 0$ for all
$t\geq 0$. By considering the Wronskian of $r_3$ and $\tilde{r}_4$
for $t\geq t_2$ we have
\begin{equation} \label{eq.r4firstorder}
    \tilde{r}'_4(t) - \frac{r'_3(t)}{r_3(t)}\tilde{r}_4(t) = \mathcal{W}(a,b,0)\frac{\e^{at}(1+t)^{-1}}{r_3(t)}, \quad t\geq t_2,
\end{equation}
where $\mathcal{W}(a,b,0)\neq 0$ is the Wronskian of $r_3$ and
$\tilde{r}_4$ at $t=0$. \eqref{eq.r4firstorder}
yields the representation
\[
    \tilde{r}_{4}(t) = r_3(t)\frac{\tilde{r}_4(t_2)}{r_3(t_2)}
    + \mathcal{W}(a,b,0)r_{3}(t)\int_{t_2}^{t}\frac{\e^{as}(1+s)^{-1}}{r_{3}^{2}(s)} \rd s, \quad t\geq t_2
\]
for $\tilde{r}_{4}$.
By means of l'H\^{o}pital's rule and \eqref{eq.r3asy} we can deduce
from this representation for $\tilde{r}_4$ that
\begin{equation}\label{eq:secr4}
    \lim_{t\to\infty}\e^{-at}t^{-\ba}\tilde{r}_{4}(t) = \mathcal{W}(a,b,0) a^{\ba}.
\end{equation}
This is consistent with the asymptotic behaviour we established for
$r_4$ in \eqref{eq.r4asy}.

It is also useful to determine some asymptotic information about
$\tilde{r}_4'$. Notice that
$t\mapsto U(1+\ba,1,a(1+t))$ is a polynomial of degree $-1-b/a\in
\mathbb{N}$, and so $\lim_{t\to\infty} r_3'(t)/r_3(t)=0$. By
\eqref{eq:secr4}, it follows that $\tilde{r}_4(t)\neq 0$ for all
$t\geq t_3$. Letting $t_4=\max(t_2,t_3)$, we rewrite
\eqref{eq.r4firstorder} for $t\geq t_4$ to get
\begin{equation*}
    \frac{\tilde{r}'_4(t)}{\tilde{r}_4(t)} = \frac{r'_3(t)}{r_3(t)}
    + \mathcal{W}(a,b,0)\frac{\e^{at}(1+t)^{-1}}{r_3(t)\tilde{r}_4(t)}.
\end{equation*}
Using the fact that $r_3(t)\sim a^{-1-\ba}t^{-1-b/a}$ as
$t\to\infty$ together with \eqref{eq:secr4} shows that the second
term has limit $a$, and therefore
\begin{equation} \label{eq.r4tilasy}
\lim_{t\to\infty} \frac{\tilde{r}'_4(t)}{\tilde{r}_4(t)}=a.
\end{equation}

Since $r_3$ and $\tilde{r}_4$ are linearly independent, we can use
the representation \eqref{eq:detspec4} for $x$ to find $\tilde{c}_3$
and $\tilde{c}_4$ such that the initial conditions of
\eqref{eq:secondorderic} (or \eqref{eq.detaverage}) are satisfied.
In particular, $\tilde{c}_4$ can be expressed according to
\[
    \tilde{c}_4 = \frac{1}{\mathcal{W}(a,b,0)} \left( a\psi(0)\, U(1+\ba,2,a) + b\int_{-1}^{0}\psi(s)ds\, U(1+\ba,1,a) \right).
\]
An argument, which is identical in all respects to that used to
deduce the representation \eqref{eq:resdecomp} of the solution $r$
of the resolvent equation \eqref{eq.resolv} in the case when $a<0$,
can be used to justify the formulae
\begin{equation} \label{eq:resdecomp2}
    r(t,s) =
    \begin{cases}
        d_3(s)r_{3}(t) + d_4(s)r_{4}(t), &\quad a>0, \quad b/a\not\in\{-1,-2...\}, \\
        \tilde{d}_3(s)r_{3}(t) + \tilde{d}_4(s)\tilde{r}_{4}(t), &\quad a>0, \quad b/a\in\{-1,-2...\}.
    \end{cases}
\end{equation}
Conditions for $d_3$ and $d_4$, and for $\tilde{d}_3$ and
$\tilde{d}_4$, are obtained from the initial conditions
\eqref{eq.resolvic} and \eqref{eq:chg4}, just as was done to obtain
the equations \eqref{eq.d1d2ic} for $d_1$ and $d_2$ in the case when
$a<0$. Solving the corresponding equations to \eqref{eq.d1d2ic}, we
obtain
\begin{align}
    d_{3}(s) &= \Gamma(1+\ba) \e^{-a(1+s)}(1+s) b M(1+\ba,2,a(1+s)) , \notag\\
    d_{4}(s) &= \Gamma(1+\ba) \e^{-a(1+s)}(1+s) a U(1+\ba,2,a(1+s)). \label{eq:d4}
\end{align}
Proceeding in the same manner in the degenerate case when
$b/a\in\{-1,-2,...\}$ yields the expressions
\begin{align*}
    \tilde{d}_{3}(s) &= \frac{\tilde{r}_4'(s)-a\tilde{r}_4(s)}{\mathcal{W}(a,b,0)\e^{as}(1+s)^{-1}} =-\frac{a}{\mathcal{W}(a,b,0)}\tilde{r}_4(s)(1+s)\e^{-as}\left(1+\frac{1}{-a}\frac{\tilde{r}_4'(s)}{\tilde{r}_4(s)}\right) , \\
    \tilde{d}_{4}(s) &= -\frac{r_3'(s)-ar_3(s)}{\mathcal{W}(a,b,0)\e^{as}(1+s)^{-1}}
    =\frac{1}{\mathcal{W}(a,b,0)}\e^{-as}(1+s)aU(1+\ba,2,a(1+s)). 
\end{align*}
We now turn our attention to the asymptotic behaviour of $d_3$,
$d_4$ etc. Using \eqref{eq:chg3}, we can show that
\begin{align} \label{eq.d3asy}
d_3(s)&\sim ba^{b/a-1}s^{b/a}, \quad\text{ as $s\to\infty$},
\\
\label{eq.d4asy} d_4(s)&\sim \Gamma(1+b/a) e^{-a} a^{-b/a} s^{-b/a}
\e^{-as}, \quad \text{ as $s\to\infty$}.
\end{align}
In the degenerate case when $b/a\in\{-1,-2,...\}$, 
we may use \eqref{eq:secr4} and \eqref{eq.r4tilasy} to establish
that
\begin{equation}\label{eq.d3tilasy}
\tilde{d}_{3}(s) = o(s^{\ba+1}), \quad \text{as $s\to\infty$.}
\end{equation}
\eqref{eq.d3tilasy} is consistent with, but weaker than, the
asymptotic estimate obtained for $d_3$ in \eqref{eq.d3asy} in the
non--degenerate case.
As to the asymptotic behaviour of $\tilde{d}_4$, we may use
\eqref{eq:chg3b} to give
\begin{equation}\label{eq.d4tilasy}
\tilde{d}_{4}(s) \sim  \frac{1}{\mathcal{W}(a,b,0)}a^{-\ba}s^{-\ba}
\e^{-as}
\quad \text{as $s\to\infty$,}
\end{equation}
which is consistent with the asymptotic behaviour in
\eqref{eq.d4asy} in the non--degenerate case.

\subsection{$a=0$}
When $a=0$ and $b>0$, it transpires that the solution of
\eqref{eq:secondorder} can be expressed in terms of {\it modified
Bessel functions}. To be more precise, we have
\begin{equation}\label{eq:detspec5}
        x(t) = c_{5}r_{5}(t) + c_{6}r_{6}(t), \quad \text{for $t\geq 0$, when $a=0$ and $b>0$} \\
\end{equation}
where
\begin{align}\label{eq:r5r6}
    r_{5}(t) = I_0(2\sqrt{b(t+1)}), \quad r_{6}(t) = K_0(2\sqrt{b(t+1)})
\end{align}
and $I_\nu$ and $K_\nu$ are two linearly independent solutions of
{\it modified Bessel's equation}
\[
    z^2 w''(z) + z w'(z) - (z^2+\nu^2) w(z) =0,
\]
with $\nu$ a real parameter. See e.g.~\cite[Chapter~10.25.1]{NIST}
for details. $I_\nu$ and $K_\nu$ are referred to as {\it modified
Bessel functions of the first kind} and {\it second kind}
respectively. One may verify that $r_5$ and $r_6$ are linearly
independent solutions of \eqref{eq:secondorder} by a direct
calculation.

The constants $c_5$ and $c_6$ in \eqref{eq:detspec5} can be found
using the initial conditions \eqref{eq:secondorderic} or
\eqref{eq.detaverageinitcondn}. Doing this yields the formulae
\begin{align}\label{eq:cntsc5}
    c_{5} &= 2 \left( \psi(0)\sqrt{b}K_1(2\sqrt{b}) + b\int_{-1}^{0}\psi(s)ds K_0(2\sqrt{b}) \right), \\
    c_{6} &= 2 \left( \psi(0)\sqrt{b}I_1(2\sqrt{b}) - b\int_{-1}^{0}\psi(s)ds I_0(2\sqrt{b}) \right). \notag
\end{align}
In finding these expressions for $c_5$ and $c_6$, we have exploited
the fact that the Wronskian of $I_\nu$ and $K_\nu$ obeys the
identity
\begin{equation} \label{eq.Wmodbessel}
    W\{K_\nu(z),I_\nu(z)\} = 1/z
\end{equation}
(which appears as~\cite[10.28.2]{NIST}, for example) and the
derivatives of $I_0$ and $K_0$ obey
\begin{equation} \label{eq.derivmodbessel}
    I_0'(z)=I_1(z) , \quad K_0'(z)=-K_1(z).
\end{equation}
(cf., e.g.~\cite[10.29.3]{NIST}). We will also employ in the sequel
the asymptotic behaviour of $I_\nu$ and $K_\nu$. The relevant
results are  
\begin{align}\label{eq:modBesasy}
    I_\nu(t) = \frac{\e^{t}}{\sqrt{2\pi t}} \{1+O(t^{-1})\},
    \quad
    K_\nu(t) = \sqrt{\frac{\pi}{2t}} \e^{-t} \{1+O(t^{-1})\}, \quad \text{ as $t\to\infty$},
\end{align}
which appear as ~\cite[9.7.1 \& 9.7.2]{abramstegun}, for example.

As in the cases when $a<0$ or $a>0$, the solution to the resolvent
equation \eqref{eq.resolv} can be represented as the sum of products
of functions in $t$ and $s$. Indeed, $r(t,s)$ can be written in the
form
\begin{equation} \label{eq:resdecomp6}
    r(t,s) = d_5(s)r_{5}(t) + d_6(s)r_{6}(t), \quad \text{$t\geq s\geq 0$, for $a=0$ and $b>0$}.
\end{equation}
As in e.g., \eqref{eq.d1d2ic}, $d_5$ and $d_6$ may be found by
solving a pair of linear simultaneous equations formulated from
\eqref{eq.resolvic}. This leads to the formulae
\begin{align}\label{eq:cntsd5}
    d_5(s) = 2\sqrt{b(s+1)} K_1(2\sqrt{b(s+1)}), \quad
    d_6(s) = 2\sqrt{b(s+1)} I_1(2\sqrt{b(s+1)}),
\end{align}
by making use of the identities \eqref{eq.Wmodbessel} and
\eqref{eq.derivmodbessel}.

In the case when $a=0$ and $b<0$, it turns out that the solution of
\eqref{eq:secondorder} can be expressed in terms of {\it Bessel
functions}. Indeed, we have
\begin{equation}\label{eq:detspec7}
        x(t) = c_{7}r_{7}(t) + c_{8}r_{8}(t) \quad \text{for $t\geq 0$, when  $a=0$ and $b<0$}
\end{equation}
where
\begin{align}\label{eq:r7r8}
    r_{7}(t) = J_0(2\sqrt{-b(t+1)}), \quad r_{8}(t) = Y_0(2\sqrt{-b(t+1)})
\end{align}
and $J_\nu$ and $Y_\nu$ are two linearly independent solutions of
{\it Bessel's Equation}
\[
    z^2 w''(z) + z w'(z) + (z^2-\nu^2) w(z) =0,
\]
where $\nu$ is a real parameter (cf.,
e.g.~\cite[Chapter~10.2.1]{NIST} for details). $J_\nu$ and $Y_\nu$
are referred to as the {\it Bessel functions of the first kind} and
{\it second kind} respectively. We remark that the Bessel functions
are oscillatory, convergent to zero and real--valued for positive
arguments. Moreover as the argument $z\to +\infty$, $Y_\nu(z)$ and
$J_\nu(z)$ share the same amplitude, and are out of phase by
$\frac{1}{2}\pi$, \cite[pp.242, Ch.7.5.1]{olver:asy}. We make this
precise in \eqref{eq:Besasy} below. One may verify by direct
calculation that $r_7$ and $r_8$ are linearly independent solutions
of \eqref{eq:secondorder}.

From \eqref{eq:detspec7} and \eqref{eq:secondorderic}, we can find
expressions for the constants $c_7$ and $c_8$. In fact, one obtains
\begin{align}\label{eq:cntsc7}
    c_{7} &= \pi\left( \psi(0)\sqrt{|b|}Y_1(2\sqrt{|b|}) - b\int_{-1}^{0}\psi(s)ds\, Y_0(2\sqrt{|b|}) \right), \\
    c_{8} &= \pi\left( \psi(0)\sqrt{|b|}J_1(2\sqrt{|b|}) + b\int_{-1}^{0}\psi(s)ds\, J_0(2\sqrt{|b|}) \right). \label{eq:cntsc8}
\end{align}
In deducing these formulae, we have used the fact that the Wronskian
of $J_\nu$ and $Y_\nu$ obeys
\[
    W\{J_\nu(z),Y_\nu(z)\} = 2/(\pi z)
\]
(cf., e.g.,~\cite[10.5.2]{NIST}) and also that the derivatives of
$J_\nu$ and $Y_\nu$ obey
\[
    J_0'(z)=-J_1(z) , \quad Y_0'(z)=Y_1(z)
\]
cf., e.g.~\cite[10.6.3]{NIST}. In asymptotic analysis of the
solution of the stochastic equation, we will need information about
the asymptotic behaviour of $J_\nu(t)$ and $Y_\nu(t)$ as
$t\to\infty$. The required asymptotic information is furnished
by~\cite[9.2.1, 9.2.2, 9.2.5, 9.2.6]{abramstegun}, which we record
now for convenience:
\begin{subequations}\label{eq:Besasy}
\begin{align}
    J_\nu(t) &=\sqrt{2/(\pi t)} \{\cos(t-\frac{1}{2}\nu\pi -\frac{1}{4}\pi) + O(t^{-1}) \}, \quad \text{ as $t\to\infty$},  \\
    Y_\nu(t) &=\sqrt{2/(\pi t)} \{\sin(t-\frac{1}{2}\nu\pi -\frac{1}{4}\pi) + O(t^{-1}) \}, \quad \text{ as $t\to\infty$}.
\end{align}
\end{subequations}
Once again the solution to the resolvent equation \eqref{eq.resolv}
can be written as a sum of products of functions depending on $t$
and $s$. Indeed, $r(t,s)$ can be written in the form
\begin{equation} \label{eq:resdecomp8}
    r(t,s) = d_7(s)r_{7}(t) + d_8(s)r_{8}(t), \quad t\geq s\geq 0,\quad a=0, \quad b<0,
\end{equation}
and expressions for $d_7$ and $d_8$ may be obtained from this
representation and \eqref{eq.resolvic}.
This yields
\begin{align}\label{eq:cntsd78}
    d_7(s)
    = \pi \sqrt{|b|(1+s)} Y_{1}(2\sqrt{|b|(s+1)}), \quad
    d_8(s)
    =\pi\sqrt{|b|(s+1)}\, J_{1}(2\sqrt{|b|(s+1)}),
\end{align}
upon use of the identities for the Wronskian of $J_0$ and $Y_0$ and
formulae for the derivatives of $J_0$ and $Y_0$.

\section{Recurrent Asymptotic Behaviour}\label{sect:rec}
\subsection{Pathwise asymptotic stationary behaviour}\label{sect:recpath}
The asymptotic behaviour of \eqref{eq.longmemorysdde} in the case
when $a<0$ and $a+b<0$ is very similar to the Ornstein--Uhlenbeck
process $U$ given by
\begin{equation} \label{eq.OU}
dU(t)=aU(t)\,dt + \sigma\,dB(t), \quad t\geq 0; \quad U(0)=0.
\end{equation}
There is a unique continuous adapted process which obeys
\eqref{eq.OU} and it is given by
\begin{equation}\label{eq:urep}
    U(t)=e^{at}\int_0^t \sigma e^{-as}\,dB(s),\quad t\geq 0.
\end{equation}
\begin{theorem}\label{thm.averagegaussian}
Let $a<0$ and $a+b\leq 0$. Suppose that $\psi\in
C([-1,0];\mathbb{R})$. Let $X$ be the unique continuous adapted
process which obeys \eqref{eq.longmemorysdde} and let $U$ be the
unique continuous adapted process which obeys \eqref{eq.OU}. Then:
\begin{itemize}
\item[(i)] $X$ obeys
\begin{equation} \label{eq.ldXave}
\limsup_{t\to\infty} \frac{X(t)}{\sqrt{2\log
t}}=\frac{\sigma}{\sqrt{2|a|}}, \quad \liminf_{t\to\infty}
\frac{X(t)}{\sqrt{2\log t}}=-\frac{\sigma}{\sqrt{2|a|}},
\quad\text{a.s.}
\end{equation}
\item[(ii)] In the case that $a+b<0$, we have
\begin{align} \label{eq.ouaveasyequal1}
    \lim_{t\to\infty} \{ X(t)-U(t) \} =0, \quad\text{a.s.}
\end{align}
and 
\begin{equation} \label{eq.ave0}
\lim_{t\to\infty} \frac{1}{1+t}\int_{-1}^t X(s)\,ds = 0, \quad \text{a.s.}
\end{equation}
\item[(iii)] In the case that $a+b=0$, we have
\begin{align} \label{eq.ouaveasyequal2}
    \lim_{t\to\infty} \{ X(t)-U(t) \} = L, \quad\text{a.s.}
\end{align}
where $L$ is a proper Gaussian random variable with mean and
variance given by
\begin{align*}
    \mathbb{E}[L] &= b^2 \Gamma(-\ba)\left(\int_{-1}^{0}\psi(u)du\right)\int_{0}^{\infty}
     \, U(1-\ba,2,-a(1+s))\,ds\\
    &\qquad+ b^2  \Gamma(-\ba) \psi(0)\int_{u=0}^{\infty}\e^{au}
    \int_{s=u}^{\infty}  \, U(1-\ba,2,-a(1+s))\,ds\,du, \\
    \Var[L]&= \sigma^2\int_{0}^{\infty} \e^{-2au}
    \left(\int_{u}^{\infty}\e^{aw}\int_{w}^{\infty} b^2  \Gamma(-\ba)  \, U(1-\ba,2,-a(1+s))\,ds\,dw
     \right)^2\,du,
\end{align*}
and 
\begin{equation} \label{eq.aveL}
\lim_{t\to\infty} \frac{1}{1+t}\int_{-1}^t X(s)\,ds = L, \quad \text{a.s.}
\end{equation}
\end{itemize}
\end{theorem}
The result \eqref{eq.ave0} shows that, when $a < 0$ and $a+b < 0$, the sample mean of the process
$X$ tends to zero, i.e. the fluctuations of $X$, which are of order
$\sqrt{\log t}$, occur symmetrically
about zero. The result \eqref{eq.aveL} however shows that, when $a < 0$ and $a + b = 0$, the
fluctuations of $X$ occur about the level $L$ (which is random and so will be different
for each sample path).

It is of interest to ask if we provide an upper bound on the a.s.
rate of convergence of $X-U$ to zero when $a+b<0$. Of course the
case when $a+b=0$ is excluded, because in that case $X-U$ tends to a
non--trivial limit. We show that in all cases, the bound on the
closeness decays polynomially.
\begin{theorem}\label{thm:XUrate}
Let $a<0$ and $a+b< 0$. Suppose that $\psi\in C([-1,0];\mathbb{R})$.
Let $X$ be the unique continuous adapted process which obeys
\eqref{eq.longmemorysdde} and let $U$ be the unique continuous
adapted process which obeys \eqref{eq.OU}. Then:
\begin{itemize}
    \item[(i)] If $a+b<0$ and $2b+a>0$, then
    \[
        \limsup_{t\to\infty}\frac{|X(t)-U(t)|}{t^{-1-\ba}} \in[0,\infty), \quad\text{a.s.}
    \]
    \item[(ii)] If $2b + a < 0$
    \[
     \limsup_{t\to\infty} \frac{|X(t)-U(t)|}{t^{-1/2}\sqrt{\log\log t}} \in[0,\infty), \quad\text{a.s.}
    \]
    \item[(iii)] If $2b+a=0$, then
    \[
        \limsup_{t\to\infty} \frac{|X(t)-U(t)|}{t^{-1/2}\log t \sqrt{\log\log t}} \in[0,\infty), \quad\text{a.s.}
    \]
\end{itemize}
\end{theorem}
While we conjecture that these estimates are sharp, i.e. the limits
superior in Theorem~\ref{thm:XUrate} are positive, such an analysis
would involve, amongst other things, a sharper analysis of the
leading order terms in the expansions in \eqref{eq:chg3}, as well as
lower estimates of certain integrals in the proof. Such analysis
goes beyond the scope of the present work.
\subsection{Asymptotic behaviour of the autocovariance function}\label{sect:recauto}
Theorem~\ref{thm.averagegaussian} shows that $X$ is a Gaussian
process which is asymptotically close to the asymptotically
stationary Gaussian process $U$ (for $b=0$, $X$ is itself an
Ornstein-Uhlenbeck process). Since $U$ is given by \eqref{eq:urep},
its autocovariance function may be shown to obey
\[
\Cov(U(t),U(t+\Delta))=\sigma^2 e^{a\Delta} e^{2at} \int_0^t
e^{-2as} \,ds= e^{a\Delta} \sigma^2
\frac{1}{2|a|}\left(1-e^{2at}\right).
\]
Therefore, for each fixed $t>0$ we have $\Delta\mapsto
\Cov(U(t),U(t+\Delta))$ decays \emph{exponentially} to zero as
$\Delta\to\infty$. It is therefore reasonable to expect that the
autocovariance function of $X$ defined by
\eqref{eq.autocovarianceaverage} to behave according to
$\lim_{\Delta\to\infty} \text{Cov}(X(t),X(t+\Delta))=0$ for every
$t\geq 0$. However, as is shown below, although $X(t)-U(t)\to 0$ as
$t\to\infty$ a.s., for each fixed $t>0$, the autocovariance
$\Delta\mapsto \text{Cov}(X(t),X(t+\Delta))$ decays
\emph{polynomially} to zero as $\Delta\to\infty$.

We have already seen in \eqref{eq:acvfspec1} that it is possible to
represent the autocovariance function in terms of $r_1, r_2, d_1,
d_2$ etc. Using the information about the asymptotic behaviour of
these functions, we can readily how rapidly the autocovariance
function decays in the time lag $\Delta$.
\begin{theorem}\label{thm:acfmem}
    Suppose that $a<0$ and $a+b\leq 0$. Suppose that $\psi\in C([-1,0];\mathbb{R})$. Let $X$ be the unique continuous
adapted process which obeys \eqref{eq.longmemorysdde}.  Let $t\geq0$
be fixed.
    Then
    \begin{equation}
\label{eq.polydecayacvf}
\lim_{\dl\to\infty}\frac{\Cov(X(t),X(t+\Delta))}{\dl^{-(1+\ba)}}=c_{t}(a,b),
    \end{equation}
    where $c_{t}=c_t(a,b)$ is given by
    \begin{align}\label{eq.limacvf2}
         c_t(a,b)=\sigma^2 b |a|^{-1-b/a} \int_0^t r(t,s)(1+s)U(1-b/a,2,-a(1+s))\,ds.
    \end{align}
\end{theorem}
Hence the process $X$ defined by \eqref{eq.longmemorysdde} is a long
memory process (i.e., for each fixed $t$,
$\int_{0}^{\infty}\text{Cov}(X(t),X(t+\Delta)\,d\dl = +\infty$) when
$a<0$, $b>0$ and $a+b<0$.

In the case when $a+b=0$, the covariance does not tend to zero as
$\Delta\to\infty$; instead
\begin{equation} \label{eq.covdellimconstaplusb0}
\lim_{\Delta\to\infty}\text{Cov}(X(t),X(t+\Delta))=c_t(a,b).
\end{equation}

In the special case $a<0$ and $b=0$, equation
\eqref{eq.longmemorysdde} reduces to an Ornstein-Uhlenbeck equation
and hence its autocovariance function is decays exponentially. This
is consistent with the result of Theorem \ref{thm:acfmem}, because
the value of $c_{t}$ is zero in \eqref{eq.limacvf2}. This leads us
to question under what conditions will the limit obtained in Theorem
\ref{thm:acfmem} be nonzero.

\begin{proposition}\label{prop:posacf}
    Let $b>0$. Then $\text{Cov}(X(t),X(t+\Delta))>0$ for all $\Delta>0$.
\end{proposition}

\begin{proposition}\label{thm:specfnexp2}
    If $a<0$, $b>0$ and $a+b<0$,
    then the limiting constant in \eqref{eq.limacvf2} obeys $c_{t}(a,b)>0$.
\end{proposition}
The case when $b<0$ is more delicate to analyse. However, it can be
shown that if $t$ is sufficiently large, then $c_t(a,b)$ is
negative. We can also show that $c_t(a,b)\to 0$ a $t\to\infty$ in
the case when $b>0$ and that $c_t(a,b)\to-\infty$ as $t\to\infty$ in
the case that $b<0$. We also see that $\lim_{t\to\infty} c_t(a,b)$
is nontrivial in the case when $a+b=0$, and its limit will be of
interest later in this section. Accordingly, the asymptotic
behaviour of $c_t$ is recorded in the next result.
\begin{proposition} \label{prop.shortlong}
Suppose that $a<0$ and $a+b\leq 0$ and let $c_t(a,b)$ be defined by
\eqref{eq.limacvf2}.
\begin{itemize}
\item[(a)] If $b<0$ and $a+b<0$, then
\begin{equation}  \label{eq.ctasyblt0}
\lim_{t\to\infty} \frac{c_t(a,b)}{t^{b/a}} =\sigma^2   b |a|^{-3}
\frac{|b|+|a|}{2|b|+|a|}<0,
\end{equation}
and so $c_t\to-\infty$ as $t\to\infty$.
\item[(b)] If $b>0$ and $a+b<0$, then $c_t\to 0$ as $t\to\infty$. Furthermore
\begin{enumerate}
\item[(i)] If $2b+a>0$, then
\begin{equation} \label{eq.ctasybgt01}
\lim_{t\to\infty} \frac{c_t(a,b)}{t^{-b/a-1}} =
 \frac{\sigma^2b^2}{|a|^{2+2b/a}}  \int_0^\infty   (1+s)^2U^2\left(1-\frac{b}{a},2,-a(1+s)\right)\,ds>0;
\end{equation}
\item[(ii)] If $2b+a=0$, then
\begin{equation} \label{eq.ctasybgt02}
\lim_{t\to\infty} \frac{c_t(a,b)}{t^{-1/2} \log t} =  \sigma^2
\frac{1}{4}|a|^{-2}>0;
\end{equation}
\item[(iii)] If $2b+a<0$, then $c_t$ obeys \eqref{eq.ctasyblt0} with the limit on the righthand side being positive.
\end{enumerate}
\item[(c)] If $a+b=0$, then
\begin{equation} \label{eq.ctasyaplusb0}
\lim_{t\to\infty}
c_t(a,b)=\sigma^2\frac{b^2}{|a|^{2+2b/a}}\int_0^\infty (1+s)^2
U(1-b/a,2,|a|(1+s))^2\,ds.
\end{equation}
\end{itemize}
\end{proposition}

%

In Theorem~\ref{thm:acfmem} we held the starting time, $t$, fixed
and observed the behaviour of the auto--covariance function as the
time lag, $\Delta$ tended to infinity. However it is perhaps more
typical, when testing for long memory (c.f. e.g. \cite{adk:2012}),
to fix the time lag and let the starting time tend to infinity. It
is then observed that this limiting auto--covariance function
depends only on the time lag $\Delta$ (so that the process is
transiently non--stationary) and the limiting autocovariance
function is integrable over $\Delta$, so that $X$ does not have long
memory.
\begin{theorem}\label{thm:acfshortmem}
Suppose that $a<0$ and $a+b\leq 0$. Suppose that $\psi\in
C([-1,0];\mathbb{R})$. Let $X$ be the unique continuous adapted
process which obeys \eqref{eq.longmemorysdde}.
Then, for all $\Delta\geq0$,
\begin{itemize}
\item[(a)] If $a+b<0$, then
\begin{equation} \label{eq.covaplusblt0}
    \lim_{t\to\infty}\text{Cov}(X(t),X(t+\Delta)) =
        \frac{\sigma^2}{2|a|}\e^{a\Delta}.
\end{equation}
\item[(b)] If $a+b=0$, then
\begin{multline} \label{eq.covaplusbeq0}
    \lim_{t\to\infty}\text{Cov}(X(t),X(t+\Delta))
    \\= \frac{\sigma^2}{2|a|}\e^{a\Delta}
    + \sigma^{2} \frac{b^2}{|a|^{2+2\ba}} \int_{0}^{\infty} (1+s)^2 U(1-\ba,2,|a|(1+s))^2 \,ds.
\end{multline}
\end{itemize}
\end{theorem}
It is interesting to remark that the differing rates of decay of the
autocovariance function recorded for the solution of
\eqref{eq.longmemorysdde} when $a<0$ and $a+b<0$ in the limits
\eqref{eq.covaplusblt0} and \eqref{eq.polydecayacvf} are \emph{not}
generally seen in autonomous affine differential equations. We show
below for asymptotically stationary scalar affine SFDEs which are
either finite delay or of Volterra type, that one is in a position
to characterise short or long memory by means of a single limiting
autocovariance function. Therefore, in the case of autonomous affine
equations, it does not matter whether one takes $\Delta\to\infty$ or
$t\to\infty$: as both limits leads to the same function, both give
the same classification to the process as being short or long
memory.

To make this claim more precise, and to find notation to connect
the behaviour of the autocovariance function of the solution of
\eqref{eq.longmemorysdde} with autocovariance functions of solutions
of such autonomous affine SFDEs, and to also contrast these
behaviours, we start by examining, for example, the solution $X$ of
an affine SFDE with finite delay. Such a process $X$ would be the
solution of
\begin{equation} \label{eq.affsde}
dX(t)=L(X_t)\,dt + \sigma\,dB(t), \quad t\geq 0; \quad X(t)=\psi(t),
\quad \text{$t\in [-\tau,0]$},
\end{equation}
where $L:C([-\tau,0];\mathbb{R})\to\mathbb{R}$ is a linear
functional and $\psi\in C([-\tau,0];\mathbb{R})$. Suppose that $r$
is the differential resolvent is given by
\[
r'(t)=L(r_t),\quad t>0; \quad r(0)=1; \quad r(t)=0 \text{ for $t\in
[-\tau,0)$}.
\]
We now summarise the situation in the following claim.
\begin{remark} \label{rem.acfSFDE}
If $X$ is the solution of \eqref{eq.affsde}, and the differential
resolvent $r$ associated with the drift of \eqref{eq.affsde} obeys
$r(t)\to 0$ as $t\to\infty$ and $r(t)$ is of one sign for all $t$
sufficiently large, then there are functions $\gamma$ and $c$ such
that
\begin{subequations} \label{eq.equallimacvfaffine}
\begin{gather} \label{eq.equallimacvfaffinet}
\lim_{t\to\infty} \frac{\text{Cov}(X(t),X(t+\Delta))}{\gamma(\Delta)}=1,\\
\label{eq.equallimacvfaffinedel}
\lim_{\Delta\to\infty} \frac{\text{Cov}(X(t),X(t+\Delta))}{\gamma(\Delta)}=c_t, \\
\label{eq.ctaffine} \lim_{t\to\infty} c_t=1.
\end{gather}
\end{subequations}
\end{remark}
A similar result pertains to Volterra equations with slowly decaying
autocovariance function. For instance, if $X$ is the solution of
\begin{equation} \label{eq.affineVol}
dX(t)=\left(-aX(t)+\int_0^t k(t-s)X(s)\,ds \right)\,dt
+\sigma\,dB(t), \quad t\geq 0; \quad X(0)=\xi,
\end{equation}
and we suppose that $k$ is a continuous, positive and integrable
function. Let the differential resolvent $r$ be the solution of
\[
r'(t)=-ar(t)+\int_0^t k(t-s)r(s)\,ds, \quad t\geq 0; \quad r(0)=1.
\]
\begin{remark} \label{rem.acfvol}
Suppose that $k$ is a positive, continuous and integrable function
which is subexponential and asymptotic to a decreasing function, and
moreover obeys $a>\int_0^\infty k(s)\,ds$. Then the autocovariance
function of the solution $X$ of \eqref{eq.affineVol} obeys
\eqref{eq.equallimacvfaffine}.
\end{remark}
We are now in a position to compare and contrast the situation with
\eqref{eq.equallimacvfaffine}, which pertains for solutions of
affine autonomous equations. For the average equation the
autocovariance function obeys
\begin{subequations} \label{eq.equallimacvfave}
\begin{gather} \label{eq.equallimacvavet}
\lim_{t\to\infty} \frac{\text{Cov}(X(t),X(t+\Delta))}{\gamma_1(\Delta)}=1,\\
\label{eq.equallimacvavedel}
\lim_{\Delta\to\infty} \frac{\text{Cov}(X(t),X(t+\Delta))}{\gamma_2(\Delta)}=c_t, \\
\label{eq.ctave}
 \lim_{t\to\infty} c_t=
 \left\{\begin{array}{cc}
 0, & b>0, \\
 -\infty, & b<0
 \end{array}
 \right.
\end{gather}
\end{subequations}
where $\gamma_1(\Delta)=\sigma^2/2|a| \cdot e^{a\Delta}$ and
$\gamma_2(\Delta)=\Delta^{-(1+b/a)}$. Therefore, the situation in
\eqref{eq.equallimacvfave} differs from the case in
\eqref{eq.equallimacvfaffine}, because there are two different rates
of decay in $\Delta$ in \eqref{eq.equallimacvavet} and
\eqref{eq.equallimacvavedel} and the function $c_t$ in
\eqref{eq.ctave} does not tend to a non--trivial finite limit as
$t\to\infty$.

Theorem~\ref{thm:acfshortmem} part (a) is consistent with
Theorem~\ref{thm.averagegaussian} part (b), because in the case when
$a+b<0$, the latter result shows that $X$ is pathwise asymptotic to
a process whose limiting autocovariance function is given in part
(a). The result of part (b) is also consistent with
Theorem~\ref{thm.averagegaussian}, because when $a+b=0$, we know
from part (c) of Theorem~\ref{thm.averagegaussian} that the solution
is asymptotic to $U$ plus a non--trivial limiting random variable,
whose presence is suggested by the form of the limiting
autocovariance function in part (b).

It is tempting to remark that when $b>0$,
Proposition~\ref{prop.shortlong} part (a) may be thought of as
partly reconciling the differing asymptotic behaviour of
$\text{Cov}(X(t),X(t+\Delta))$ recorded in Theorem~\ref{thm:acfmem}
and \ref{thm:acfshortmem} according as to whether $\Delta\to\infty$
or $t\to\infty$. This is because $c_t(a,b)\to 0$ as $t\to\infty$, so
that the ``long memory'' recorded in \eqref{eq.polydecayacvf}
becomes ever weaker as the start time $t$ becomes greater, and
therefore becomes closer to the ``short memory'' or exponential
decay in $\Delta$ in the limiting autocovariance function determined
in part (a) of Theorem~\ref{thm:acfshortmem}.

This heuristic explanation of the reconciliation of the asymptotic
behaviour of the autocovariance must however be taken with caution.
In particular, in the case when $b<0$, it is harder to forward with
equal confidence the same explanation as to the differing asymptotic
behaviour recorded in Theorem~\ref{thm:acfmem} and
\ref{thm:acfshortmem}. In this case,
Proposition~\ref{prop.shortlong} part (b) shows that $c_t(a,b)\to
-\infty$ as $t\to\infty$, suggesting that the polynomial decay in
the autocovariance function given in \eqref{eq.polydecayacvf} tends
to become \emph{stronger} as the start time is chosen to be very
large. On the other hand, the fact that $|c_t|$ has power law growth
which is less rapid as $t\to\infty$ (at a rate $t^{b/a}$ according
to \eqref{eq.ctasyblt0}) compared to the power law decay of
$\Cov(X(t),X(t+\Delta))$ as $\Delta\to\infty$ (which is at the rate
$\Delta^{-(1+b/a)}$) may point to a weakening overall correlation.

One situation in which it does not seem to matter in which order
limits are taken is when $a+b=0$. Taking the limit as
$\Delta\to\infty$ in \eqref{eq.covaplusbeq0} leads to
\[
\lim_{\Delta\to\infty}\lim_{t\to\infty}\text{Cov}(X(t),X(t+\Delta))
=
    \sigma^{2} \frac{b^2}{|a|^{2+2\ba}} \int_{0}^{\infty} (1+s)^2 U(1-\ba,2,|a|(1+s))^2
    \,ds.
    \]
    On the other hand, by \eqref{eq.covdellimconstaplusb0} and \eqref{eq.ctasyaplusb0}
    we have that
    \[
\lim_{t\to\infty}\lim_{\Delta\to\infty}\text{Cov}(X(t),X(t+\Delta))
=
    \sigma^{2} \frac{b^2}{|a|^{2+2\ba}} \int_{0}^{\infty} (1+s)^2 U(1-\ba,2,|a|(1+s))^2
    \,ds,
    \]
so the limits are equal.

%
\subsection{Non-stationary asymptotic behaviour}
In the case when $a<0$ and $b<0$, we have already seen that the solution of \eqref{eq.longmemorysdde} is asymptotically 
stationary, and when $a>0$ and $b<0$, the solution exhibits a.s. exponential growth. Therefore, we expect to see intermediate 
asymptotic behaviour on the boundary of these two parameter regions, where $a=0$ and $b<0$. In broad terms, we can establish that 
the solution behaves in some ways like a standard Brownian motion, in the sense that the solution is a Gaussian process which has 
asymptotically vanishing mean, variance which grows linearly in time, and experiences a.s. large fluctuations which satisfy the Law of 
the iterated logarithm. 
\begin{theorem}\label{thm.meanvarbessel}
Suppose that $\psi\in C([-1,0];\mathbb{R})$. Let $X$ be the unique
continuous adapted process which obeys \eqref{eq.longmemorysdde}. If
$a=0$ and $b<0$, then $\mathbb{E}[X(t)]\to 0$ as $t\to\infty$ and 
\begin{equation*}
    \lim_{t\to\infty} \frac{\text{Var}[X(t)]}{t} =\frac{1}{3}\sigma^2. 
\end{equation*}
\end{theorem}
We now state the result which deals with the magnitude of the large fluctuations of $X$.
\begin{theorem}\label{thm:Bes}
Suppose that $\psi\in C([-1,0];\mathbb{R})$. Let $X$ be the unique
continuous adapted process which obeys \eqref{eq.longmemorysdde}. If
$a=0$ and $b<0$, then
\begin{align*}
    \limsup_{t\to\infty} \frac{X(t)}{\sqrt{2t\log\log t}}=\frac{1}{\sqrt{3}}\sigma,
    \quad
    \liminf_{t\to\infty} \frac{X(t)}{\sqrt{2t\log\log t}}=-\frac{1}{\sqrt{3}}\sigma, \quad\text{a.s.} 
\end{align*}
\end{theorem}
\begin{remark} Both Theorems~\ref{thm.meanvarbessel} and \ref{thm:Bes} show that, asymptotically, $X$ has behaviour is
somewhat akin to standard Brownian motion. In particular it is observed that the limiting
constant in Theorem~\ref{thm.meanvarbessel} is the square of that in Theorem~\ref{thm:Bes}. We are then drawn to
conjecture that the increments of $X$, under the hypotheses of Theorems~\ref{thm.meanvarbessel} and \ref{thm:Bes},
are asymptotically stationary.
\end{remark}

\section{Transient Asymptotic Behaviour}\label{sect:trans}
From \eqref{eq.xrep} we see that as $X$ depends upon $x$, we then
expect the asymptotic behaviour of $X$ to also depend upon $x$,
especially in the case when $|x(t)|\to\infty$ as $t\to\infty$. This
arises in two main situations: when $a<0$ and $a+b>0$, and when
$a>0$. We deal with the first of these cases first, and establish
that $|X(t)|\to\infty$ as $t\to\infty$ like a power of $t$. In fact,
$X$ can tend to $+\infty$ or to $-\infty$, each with positive
probability. Moreover, the choice of which limit is attained depends
on the path of the Brownian motion driving $X$, with the increments
of $B$ earlier in the path generally proving to be more influential
in deciding which limit is attained. The key to the proof of this
result, and to the others in this Section, hinge on the
representation of the solution $X$ of \eqref{eq.longmemorysdde} in
terms of the resolvent $r$ and mean $x$, as well as the asymptotic
analysis of these functions  given in Section~\ref{sec.specialfn}.
\begin{theorem}\label{cor:stochexp}
 Suppose that $a<0$, $a+b>0$. Suppose also that $\psi\in C([-1,0];\mathbb{R})$. Let $X$ be the unique continuous
adapted process which obeys \eqref{eq.longmemorysdde}. Then
\begin{itemize}
\item[(a)] There exists an $\mathcal{F}^B(\infty)$ measurable normal random variable $C$ such that
    \begin{equation}\label{eq:stochpolygrowth}
    \lim_{t\to\infty}\frac{X(t)}{t^{-(1+\ba)}}=C, \quad\text{a.s.}
    \end{equation}
\item[(b)] $C$ is given by
\begin{multline*}
C=|a|^{-1-\ba}b\left\{\psi(0)\,U\left(1-\ba,2,|a|\right)
      +\int_{-1}^{0}\psi(s)ds\, U\left(-\ba,1,|a|\right)\right\}\\
      +\sigma \int_0^\infty \frac{b}{|a|^{1+\ba}} (1+s) U(1-\ba,2,|a|(1+s))\, dB(s).
\end{multline*}
\item[(c)] The mean and variance of $C$ are given by
\begin{align}\label{eq.limconflu2}
    \mathbb{E}[C]&= |a|^{-1-\ba}b\left\{\psi(0)\,U\left(1-\ba,2,|a|\right)
      +\int_{-1}^{0}\psi(s)ds\, U\left(-\ba,1,|a|\right)\right\},\\
\label{eq:varCt}
    \Var[C] 
    &= \sigma^2 \frac{b^2}{|a|^{2+2\ba}} \int_{0}^{\infty}(1+s)^2 U(1-\ba,2,|a|(1+s))^2 ds>0.
\end{align}
\item[(d)] The mean and variance of $X$ obey
\begin{align*}
\lim_{t\to\infty} \frac{\mathbb{E}[X(t)]}{t^{-1-\ba}}=\mathbb{E}[C],
\quad \lim_{t\to\infty}
\frac{\text{Var}[X(t)]}{t^{-2-2\ba}}=\text{Var}[C].
\end{align*}
\end{itemize}
\end{theorem}
Once the formula \eqref{eq:varCt} is established, it is clear that
$C$ is a proper Gaussian random variable, because  $s\mapsto
U(1-\ba,2,|a|(1+s))^2$ is asymptotic to a positive function and so
is itself eventually positive. Thus we have $C\not=0$ a.s.

In the case when $a>0$, we show that $X$ grows to plus or minus
infinity at an exponential rate, with a power law correction growth
factor. Once again, there is a positive probability of each of the
events $\{\lim_{t\to\infty}X(t)=+\infty\}$ and
 $\{\lim_{t\to\infty}X(t)=-\infty\}$ occurring.
\begin{theorem}\label{thm.bubbleaveexp}
Suppose that $a>0$. Suppose also that $\psi\in
C([-1,0];\mathbb{R})$. Let $X$ be the unique continuous
adapted process which obeys \eqref{eq.longmemorysdde}. 
\begin{itemize}
\item[(a)] There exists a $\mathcal{F}^B(\infty)$ normal random variable $C$ such that
\begin{equation} \label{eq.Xaveexpbubble}
    \lim_{t\to\infty} \frac{X(t)}{e^{at}t^{b/a}}=C,
    \quad \text{a.s.}
\end{equation}
\item[(b)] $C$ is given by
\begin{multline*}
C=
a^{\ba}\left\{a\psi(0)U\left(1+\ba,2,a\right)+b\int_{-1}^{0}\psi(s)\rd
s \, U\left(1+\ba,1,a\right) \right\}
\\+\sigma a^{1+\ba}\int_0^\infty e^{-as}(1+s) U(1+\ba,2,a(1+s))\,dB(s).
\end{multline*}
\item[(c)] 
The mean and variance of $C$ are given by
\begin{equation} \label{eq.cbanotinZmin}
\mathbb{E}[C]=
a^{\ba}\left\{a\psi(0)U\left(1+\ba,2,a\right)+b\int_{-1}^{0}\psi(s)\rd
s \, U\left(1+\ba,1,a\right) \right\}.
\end{equation}
and 
\[
    \Var[C] =\sigma^2 a^{2+2\ba}\int_{0}^{\infty} \e^{-2as}(1+s)^2 U(1+\ba,2,a(1+s))^2 \,ds>0.
\]
\item[(d)] The mean and variance of $X$ obey
\begin{equation*}
\lim_{t\to\infty}
\frac{\mathbb{E}[X(t)]}{e^{at}t^{b/a}}=\mathbb{E}[C], \quad
\lim_{t\to\infty}
\frac{\text{Var}[X(t)]}{e^{2at}t^{2b/a}}=\text{Var}[C].
\end{equation*}
\end{itemize}
\end{theorem}
It can be seen from part (b) of Theorem~\ref{thm.bubbleaveexp} that
the limiting random variable in \eqref{eq.Xaveexpbubble} is a linear
functional of (the increments of) the Brownian motion $B$.
 The formula for $\mathbb{E}[C]$, given in part (c) of Theorem~\ref{thm.bubbleaveexp} is discussed~\cite{App_Dan}, 
where it is shown that in certain regions of the parameter space $\mathbb{E}[C]$ is non--zero and
hence the continuous random variable $C$ is non-zero almost surely. While part (a) is also
 dealt with in \cite{App_Dan} we present an alternative
method of proof in this paper, with the chief difference being that
a simpler formula for $C$ is attained in this paper from the the
variation of parameters representation (rather using an
admissibility approach as in \cite{App_Dan}).

In the $ab$-parameter space the line $a=0$ and $b>0$ is bordered by
a region wherein $X$ undergoes polynomial growth (covered by
Theorem~\ref{cor:stochexp}) and a region of exponential growth
(which is described by Theorem~\ref{thm.bubbleaveexp}).
As neither the representation \eqref{eq:detspec1} nor
\eqref{eq:detspec3} of the resolvent $r$ are valid on this line, it
therefore seems somewhat apt that $X$ should have a rate of faster
then polynomial yet slower than exponential growth on this line. A
precise asymptotic result is recorded in the next theorem.
\begin{theorem}\label{thm:modBes}
Suppose that $a=0$ and $b>0$. Suppose also that $\psi\in
C([-1,0];\mathbb{R})$. Let $X$ be the unique continuous adapted
process which obeys \eqref{eq.longmemorysdde}. Then
\begin{itemize}
\item[(a)] There exists an $\mathcal{F}^B(\infty)$ measurable normal random variable $C$ such that
\begin{equation*}
    \lim_{t\to\infty} \frac{X(t)}{t^{-1/4}\e^{2\sqrt{bt}}}=C, \quad \text{a.s.}
\end{equation*}
\item[(b)] $C$ is given by
\begin{multline*}
C=\frac{1}{b^{1/4}\sqrt{\pi}} \left( \psi(0)\sqrt{b}K_1(2\sqrt{b}) +
b\int_{-1}^{0}\psi(s)ds K_0(2\sqrt{b}) \right)
\\+\frac{\sigma b^{1/4}}{\sqrt{\pi}} \int_{0}^{\infty}  \sqrt{s+1} K_1(2\sqrt{b(s+1)})\,dB(s) ,
\end{multline*}
where $K_0$ and $K_1$ are modified Bessel functions of the second
kind.
\item[(c)] The mean and variance of $C$ are given by
\begin{align*}
\mathbb{E}[C]&= \frac{1}{\sqrt{\pi}}\left( \psi(0) b^{1/4}K_1(2\sqrt{b}) + b^{3/4}\int_{-1}^{0}\psi(s)ds K_0(2\sqrt{b}) \right),\\
\text{Var}[C]&=\frac{\sigma^2 b^{1/2}}{\pi} \int_{0}^{\infty} (s+1)
K_1^2(2\sqrt{b(s+1)})\,ds>0.
\end{align*}
\item[(d)] The mean and variance of $X$ obey
\begin{equation*}
\lim_{t\to\infty}
\frac{\mathbb{E}[X(t)]}{t^{-1/4}\e^{2\sqrt{bt}}}=\mathbb{E}[C],
\quad \lim_{t\to\infty}
\frac{\text{Var}[X(t)]}{t^{-1/2}\e^{4\sqrt{bt}}}=\text{Var}[C].
\end{equation*}
\end{itemize}
\end{theorem}
We see from part (c) that $C$ has positive variance, so we have that
$C\neq 0$ a.s. Therefore the limit in part (a) is nontrivial a.s.

\begin{remark}
    If one scales \eqref{eq.xrep} by $r_2$ then we have
    \[
        X(t)/r_2(t) = x(t)/r_2(t) + \sigma\int_{0}^{t}H(t,s)\,dB(s)
    \]
    where $H(t,s)=r(t,s)/r_2(t)$. Under the hypothesis of Theorem~\ref{cor:stochexp}, it is immediate from
    Theorem 4 in~\cite{appdanlimop} 
     that as the stochastic integral $\int_{0}^{t}H(t,s)dB(s)$ converges to $C$ almost surely then the convergence must take place in mean square also.
    Similarly each of the results of Theorems~\ref{thm.bubbleaveexp},\ref{thm:modBes},\ref{thm:Bes} for almost sure convergence hold true for mean square convergence also.
\end{remark}

\section{Proofs from Section~\ref{sb:avprelim} and \ref{sect:recauto}}  \label{sect:avproof}
\subsection{Proof of Lemma~\ref{lm:avXsol}} \label{sect:avcntsproof}
    Existence and uniqueness  of the solution of \eqref{eq.longmemorysdde} is known from general theory of SFDEs, c.f. e.g. \cite{BergMiz, Mao2}. Thus we need only demonstrate that the representation \eqref{eq.xrep} satisfies the SFDE \eqref{eq.longmemorysdde}.

    Firstly observe that the resolvent equation, \eqref{eq.resolv}, may be re--expressed as
    \[
        r(t,s) = 1 + a\int_{s}^{t}r(u,s)\,du + \int_{s}^{t}\frac{b}{1+u}\int_{s}^{u}r(w,s)\,dw \,du, \quad t\geq s.
    \]
    Defining $Z=X-x$, we have that $Z$ obeys
    \begin{subequations}\label{eq:Zlongmem}
        \begin{align}
            Z(t) &= a\int_{0}^{t}Z(s)ds + \int_{0}^{t}\frac{b}{1+s}\int_{0}^{s}Z(u)\,du\, ds + \sigma B(t), \quad t\geq0, \\
            Z(t) &= 0, \quad t\in[-1,0].
        \end{align}
    \end{subequations}
    From the definition of $Z$ it is apparent that demonstrating the validity of \eqref{eq.xrep} is equivalent to showing
    that $Z$ obeys
    \begin{equation}\label{eq.zrep}
        Z(t) = \sigma\int_{0}^{t}r(t,s) dB(s), \quad t\geq0.
    \end{equation}
    Let $Z^*(t)=\sigma\int_{0}^{t}r(t,s) dB(s)$, $t\geq0$ and so $Z^*(0)=0$ as required.
    Now using the stochastic Fubini theorem
    \begin{align*}
        &a\int_{0}^{t}Z^*(s)ds + \int_{0}^{t}\frac{b}{1+s}\int_{0}^{s}Z^*(u)\,du\, ds + \sigma B(t) \\
        &\,=a\sigma\int_{0}^{t}\int_{0}^{s}r(s,w)\,dB(w)\,ds \\
        &\qquad+ \int_{0}^{t}\frac{b}{1+s}\int_{0}^{s}\sigma\int_{w=0}^{u}r(u,w) dB(w)\,du\,ds + \sigma B(t) \\
        &= \sigma\int_{0}^{t} \left( a\int_{w}^{t}r(s,w)\,ds + \int_{w}^{t}\frac{b}{1+s}\int_{w}^{s}r(u,w)\,du\,ds \right)dB(w) + \sigma B(t) \\
        &= \sigma\int_{0}^{t}\left( r(t,w)-1 \right)\,dB(w) + \sigma B(t) = \sigma\int_{0}^{t}r(t,w)\,dB(w) = Z^{*}(t).
    \end{align*}
    As $Z$ is the unique solution of \eqref{eq:Zlongmem} we have $Z=Z^*$ and hence $X$ has the representation \eqref{eq.xrep}.
\subsection{Proof of Proposition~\ref{lm:acf}}
    Let $t\geq0$ and $\dl\geq0$.
    Differentiating \eqref{eq.acfdef} with respect to $\dl$, using \eqref{eq.resolveq}, and by exchanging the order of integration and decomposing the integral, we get
    \begin{align*}
        \g_t'(\dl)&=\sigma^2\int_{0}^{t}r(t,s)\frac{\partial}{\partial\dl}r(t+\dl,s)\,ds\\
        &=\sigma^2a\int_{0}^{t}r(t,s)r(t+\dl,s)\,ds + \sigma^2\frac{b}{1+t+\dl}\int_{0}^{t}\int_{s}^{t+\dl}r(t,s)r(u,s)\,du\,ds \\
        &= a\g_{t}(\dl) + \frac{b\sigma^2}{1+t+\dl}\int_{0}^{t}\int_{0}^{u}r(t,s)r(u,s)ds\,du \\
        &\quad + \frac{b\sigma^2}{1+t+\dl}\int_{t}^{t+\dl}\int_{0}^{t}r(t,s)r(u,s)ds\,du.
      \end{align*}
      Next, because $r(w,s)=0$ for $0\leq w<s$, we see that
      $\int_0^u r(t,s)r(u,s)\,ds=\int_0^t r(t,s)r(u,s)\,ds$ for $u\in [0,t]$. Hence the two integrals on the right hand side can be combined. By making the substitution $w=u-t$,
      and then splitting the integral, we get
        \begin{align*}
      \g_t'(\dl)  &=a\g_{t}(\dl)+\frac{b\sigma^2}{1+t+\dl}\int_{0}^{t+\dl}\int_{0}^{u}r(t,s)r(u,s)ds\,du \\
        &=a\g_{t}(\dl)+\frac{b\sigma^2}{1+t+\dl}\int_{-t}^{\dl}\int_{0}^{t}r(t,s)r(t+w,s)ds\,dw \\
        &\quad +\frac{b\sigma^2}{1+t+\dl}\int_{-t}^{\dl}\int_{t}^{w+t}r(t,s)r(t+w,s)ds\,dw\\
        &=a\g_{t}(\dl)+\frac{b\sigma^2}{1+t+\dl}\int_{-t}^{\dl}\gamma_t(w)\,dw \\
        &\quad +\frac{b\sigma^2}{1+t+\dl}\int_{-t}^{\dl}\int_{t}^{w+t}r(t,s)r(t+w,s)ds\,dw,
    \end{align*}
   where we have used the definition of $\gamma_t(w)$ at the last step. It now suffices to show that the last integral is zero.
   We first decompose it according to
    \begin{align*}
        \int_{-t}^{\dl}&\int_{t}^{w+t}r(t,s)r(t+w,s)ds\,dw \\
        &= \int_{-t}^{0}\int_{t}^{w+t}r(t,s)r(t+w,s)ds\,dw + \int_{0}^{\dl}\int_{t}^{w+t}r(t,s)r(t+w,s)ds\,dw \\
        &= \int_{-t}^{0}\int_{t}^{w+t}r(t,s)r(t+w,s)ds\,dw,
    \end{align*}
    where the last integral is zero as when $w>0$, $r(t,s)=0$ for $s\in(t,t+w]$. Since $t\geq0$ and $w\in[-t,0]$, we have that $s\in(t+w,t]$ in the remaining integral and therefore $r(t+w,s)=0$. Thus,
    \[
        \int_{-t}^{\dl}\int_{t}^{w+t}r(t,s)r(t+w,s)ds\,dw=0,
    \]
    which proves \eqref{eq.acfdde}.

For $t\geq0$ and $-t\leq\dl\leq0$, we prove \eqref{eq.acfdde2} in a
similar manner to \eqref{eq.acfdde}. However, since $\Delta\in
[-t,0]$, we can show that $\gamma_t$ can be written in the form
\[
    \gamma_t(\Delta) = \sigma^2\int_{0}^{t+\Delta}r(t,s)r(t+\Delta,s)\,ds, \quad \Delta\in [-t,0].
\]
The function on the righthand side is differentiable with respect to
$\Delta$ on $(-t,0)$, because $\Delta\mapsto r(t+\Delta,s)$ is
differentiable on $(-t,0)$. Now, differentiating with respect to
$\Delta$, we get
\[
\gamma_t'(\Delta) =
\sigma^2\int_{0}^{t+\Delta}r(t,s)\frac{\partial}{\partial\Delta}r(t+\Delta,s)\,ds
+\sigma^2 r(t,t+\Delta)r(t+\Delta,t+\Delta) , \quad \Delta\in
(-t,0),
\]
and proceeding in a manner similar to the proof of \eqref{eq.acfdde}
above, we establish \eqref{eq.acfdde2}.

\subsection{Proof of Theorem~\ref{thm:acfmem}}
In the case when $b/a\not\in \{1,2,\ldots\}$, from
\eqref{eq:acvfspec1}, we have
\[
\frac{\text{Cov}(X(t),X(t+\Delta))}{\Delta^{-(1+b/a)}}=c_{1,t}\frac{r_1(t+\Delta)}{\Delta^{-(1+b/a)}}
+c_{2,t}\frac{r_2(t+\Delta)}{(t+\Delta)^{-(1+b/a)}}\cdot\left(\frac{t+\Delta}{\Delta}\right)^{-(1+b/a)}.
\]
By \eqref{eq.r1asy} and \eqref{eq.r2asy} we have that
\[
\lim_{\Delta\to\infty}
\frac{\text{Cov}(X(t),X(t+\Delta))}{\Delta^{-(1+b/a)}}
=c_{2,t}\frac{1}{\Gamma(-b/a)}e^{-a}|a|^{-1-b/a}.
\]
Since $c_{2,t}$ is given by \eqref{eq:c1c2t} and $d_2$ by
\eqref{eq:d2}, we obtain
 \[
 \lim_{\dl\to\infty}\frac{\Cov(X(t),X(t+\Delta))}{\dl^{-(1+\ba)}}=c_{t}(a,b)
 \]
 where $c_t$ is given by \eqref{eq.limacvf2}. The proof in the case when  $b/a\in \{1,2,\ldots\}$ proceeds in the same manner, making use of \eqref{eq.r1asy} and \eqref{eq:secr2} to obtain
 \[
\lim_{\Delta\to\infty}
\frac{\text{Cov}(X(t),X(t+\Delta))}{\Delta^{-(1+b/a)}}
=\tilde{c}_{2,t} \mathcal{W}(a,b,0) |a|^{-1-b/a}.
\]
From this and the formula for $\tilde{c}_{2,t}$ in
\eqref{eq.c1c2tildet} we obtain the desired representation.
%
%

\subsection{Proof of Proposition~\ref{prop:posacf}}
Since $\Cov(X(t),X(t+\Delta)$ obeys \eqref{eq.autocovarianceaverage}
for $t\geq 0$ and $\Delta\geq 0$, we see that it suffices to show
that $r(t,s)>0$ for all $t\geq s>0$.

To this end, fix $s>0$ and write $r_{s}(t) = r(t,s)$ for $t\geq s$.
Then \eqref{eq.resolveq} and \eqref{eq.resolvic} are equivalent to
    \begin{equation*}
            r_{s}'(t) = ar_{s}(t)+b\frac{1}{1+t}\int_{s}^{t}r_{s}(u)\rd u, \quad t\geq s;\quad r_s(s)=1.
        \end{equation*}
    Note that $r_s\in \text{C}^{1}(s,\infty)$. Hence there exists some $\epsilon>0$ such that $r(t,s)>0$ for $t\in(s,s+\epsilon)$.
    Suppose there exists a minimal $t_0>s$ such that $r_{s}(t_0)=0$, but $r_{s}(t)>0$ for $s\leq t\leq t_0$. Then $r_{s}'(t)\leq0$ and
    \begin{equation*}
         0\geq r_{s}'(t_0) = ar_{s}(t_0) + b\frac{1}{1+t}\int_{s}^{t_0}r_{s}(u) \rd u
        = b\frac{1}{1+t}\int_{s}^{t_0}r_{s}(u) \rd u >0,
    \end{equation*}
    a contradiction. Hence $r(t,s)=r_{s}(t)>0$ for all $t\geq s$, and so
  $\text{Cov}(X(t),X(t+\Delta))>0$ for all $t>0$ and $\Delta\geq 0$.

\subsection{Proof of Proposition~\ref{thm:specfnexp2}}
Since $a<0$, by Theorem~\ref{thm:acfmem} we have that $c_t(a,b)$
obeys \eqref{eq.limacvf2}.
In the proof of Proposition~\ref{prop:posacf} we showed that
$r(t,s)>0$ for all $t\geq s>0$. Therefore, to show that $c_t(a,b)>0$
for all $t>0$, by examining the integral in \eqref{eq.limacvf2}, it
suffices to show that $U(1-b/a,2,|a|(1+t))>0$ for $t\geq 0$. Since
$a<0$ and $b>0$, we have $1-b/a>0$, so by the integral
representation \eqref{eq:Uint}, we have
\begin{align*}
    U(1-\ba,2,-a(1+t)) &= \frac{1}{\Gamma(1-\ba)}\int_{0}^{\infty}\e^{a(1+t)s}s^{-\ba}(1+s)^{\ba}\rd s, \quad \text{ for $t\geq 0$.} 
\end{align*}
Thus $ U(1-b/a,2,-a(1+t))>0$ for all $t\geq 0$ and $a<0<b$, and the
claim is proven.

\subsection{Proof of Proposition~\ref{prop.shortlong}}
Suppose that $b/a\not\in\{1,2,\ldots\}$. We estimate the asymptotic
behaviour of $c_t$ in \eqref{eq.limacvf2} by substituting
$r(t,s)=r_1(t)d_1(s)+r_2(t)d_2(s)$ and estimating the asymptotic
behaviour of each resulting integral in
\begin{multline} \label{eq.ctfirst}
 c_t/(\sigma^2   b |a|^{-1-b/a})= r_1(t) \int_0^t d_1(s)(1+s)U(1-b/a,2,-a(1+s))\,ds
 \\+  r_2(t) \int_0^t d_2(s)(1+s)U(1-b/a,2,-a(1+s))\,ds.
\end{multline}

We start with the first integral in \eqref{eq.ctfirst}. By
\eqref{eq:chg3b} we have that
\begin{equation} \label{eq.Ussasy}
(1+s)U(1-b/a,2,-a(1+s))\sim |a|^{b/a-1}s^{b/a} \text{ as
$s\to\infty$}.
\end{equation}
Therefore by  \eqref{eq:d1asy} we have that
\[
d_1(s)(1+s)U(1-b/a,2,-a(1+s))  \sim  \frac{1}{|a|}  e^{-as}   \text{
as $s\to\infty$}.
\]
Using the fact that $a<0$, by \eqref{eq.r1asy} we get 
\begin{equation} \label{eq.term1ctasy}
 r_1(t)\int_0^t d_1(s)(1+s)U(1-b/a,2,-a(1+s))\,ds \sim |a|^{b/a-2}  t^{b/a}, \quad \text{ as $t\to\infty$.}
\end{equation}

In the case when $b/a\in \{1,2,\ldots\}$, $c_t$ is given by
\begin{multline} \label{eq.cttildefirst}
 c_t/(\sigma^2   b |a|^{-1-b/a})= r_1(t) \int_0^t \tilde{d}_1(s)(1+s)U(1-b/a,2,-a(1+s))\,ds
 \\+  \tilde{r}_2(t) \int_0^t \tilde{d}_2(s)(1+s)U(1-b/a,2,-a(1+s))\,ds.
\end{multline}
Again, we estimate the asymptotic behaviour of the first integral.
By \eqref{eq.Ussasy} and \eqref{eq.d1tilasy} we have that
\[
\tilde{d}_1(s)(1+s)U(1-b/a,2,-a(1+s))  \sim |a|^{-1} \e^{-as},
\text{ as $s\to\infty$}.
\]
Using the fact that $a<0$ and that $r_1$ obeys \eqref{eq.r1asy}, we get 
\begin{equation} \label{eq.term1ctasytilde}
 r_1(t)\int_0^t \tilde{d}_1(s)(1+s)U(1-b/a,2,-a(1+s))\,ds \sim |a|^{b/a-2}  t^{b/a}, \quad \text{ as $t\to\infty$.}
\end{equation}

We next prepare estimates of the integrand in the second integral in
\eqref{eq.ctfirst} and \eqref{eq.cttildefirst}. When $b/a\not\in
\{1,2,\ldots\}$, we use \eqref{eq:d2asy} and \eqref{eq.Ussasy} to
obtain
\begin{equation} \label{eq.d2sUsasy}
d_2(s)(1+s)U(1-b/a,2,-a(1+s)) \sim \Gamma(-\ba)\e^{a} b
|a|^{2b/a-2}\, s^{2b/a}   \text{ as $s\to\infty$}.
\end{equation}
When $b/a\in \{1,2,\ldots\}$, we use \eqref{eq.d2tilasy} and
\eqref{eq.Ussasy} to obtain
\begin{equation} \label{eq.d2tilsUsasy}
\tilde{d}_2(s)(1+s)U(1-b/a,2,-a(1+s)) \sim
\frac{1}{\mathcal{W}(a,b,0)}b  |a|^{2b/a-2}s^{2b/a}
 \text{ as $s\to\infty$}.
\end{equation}

We now prove part (a). If $b<0$ and $b/a\not\in\{1,2,\ldots\}$, we
have that $2b/a>0$, so using \eqref{eq.d2sUsasy}
\[
\int_0^t d_2(s)(1+s)U(1-b/a,2,-a(1+s))\,ds \sim \Gamma(-\ba)\e^{a} b
|a|^{2b/a-2}\, t^{2b/a+1}\frac{1}{2b/a+1},
\]
as $t\to\infty$. Therefore by \eqref{eq.r2asy}, as $t\to\infty$, we
have that
\begin{equation} \label{eq.int2ctnondeg}
r_2(t)\int_0^t d_2(s)(1+s)U(1-b/a,2,-a(1+s))\,ds
 \sim b |a|^{b/a-3}\, \frac{1}{2b/a+1}  t^{b/a}. 
\end{equation}
In the case that $b<0$ and $b/a\in\{1,2,\ldots\}$ using
\eqref{eq.d2tilsUsasy} gives
\[
\int_0^t \tilde{d}_2(s)(1+s)U(1-b/a,2,-a(1+s))\,ds \sim
\frac{1}{\mathcal{W}(a,b,0)}b  |a|^{2b/a-2}t^{2b/a+1}\frac{1}{2b/a+1}, 
\]
as $t\to\infty$. Therefore by \eqref{eq:secr2} we have that
\begin{equation} \label{eq.int2ctdeg}
\tilde{r}_2(t)\int_0^t \tilde{d}_2(s)(1+s)U(1-b/a,2,-a(1+s))\,ds
\sim   b  |a|^{b/a-3} t^{b/a}\frac{1}{2b/a+1}, \quad \text{ as
$t\to\infty$}.
\end{equation}
Examining \eqref{eq.int2ctnondeg} and \eqref{eq.int2ctdeg}, we see
that the second integrals on the righthand sides of
\eqref{eq.ctfirst} and \eqref{eq.cttildefirst} have the same
asymptotic behaviour. Similarly, by \eqref{eq.term1ctasy} and
\eqref{eq.term1ctasytilde},  we see that the first integrals on the
righthand sides of \eqref{eq.ctfirst} and \eqref{eq.cttildefirst}
have the same asymptotic behaviour. Hence, if $b<0$, we have that
\[
 \frac{c_t}{\sigma^2   b |a|^{-1-b/a}} \sim |a|^{b/a-2} \left( b |a|^{-1}\, \frac{1}{2b/a+1}  + 1\right) t^{b/a}, \quad \text{ as $t\to\infty$},
\]
which implies \eqref{eq.ctasyblt0}.

We now prove part (b). In this case $b>0$. Therefore, $b/a\not\in
\{1,2\ldots\}$, so we estimate the asymptotic behaviour of each
integral on the right hand side of \eqref{eq.ctfirst}. In
particular, the estimate \eqref{eq.term1ctasy} holds for the fist
integral. To analyse the asymptotic behaviour of the second term, we
must consider three subcases: $2b/a<-1$, $2b/a=-1$ and $2b/a>-1$.

\textbf{Case 1: $2b/a<-1$.} If $2b/a<-1$, by  \eqref{eq.d2sUsasy} we
have
\begin{multline*}
\lim_{t\to\infty}
\int_0^t d_2(s)(1+s)U(1-b/a,2,-a(1+s))\,ds \\
= 
 \Gamma(-\ba)\e^{a} b\int_0^\infty   (1+s)^2U(1-b/a,2,-a(1+s))^2\,ds,
\end{multline*}
where we have used \eqref{eq:d2} to obtain the formula for the
limit.
Hence by \eqref{eq.r2asy} we have
\begin{multline*}
r_2(t) \int_0^t d_2(s)(1+s)U(1-b/a,2,-a(1+s))\,ds \\
\sim
 b |a|^{-b/a-1}  \int_0^\infty   (1+s)^2U(1-b/a,2,-a(1+s))^2\,ds \cdot  t^{-b/a-1}
\text{ as $t\to\infty$}.
\end{multline*}
Since $2b/a<-1$, we have that $b/a<-1-b/a<0$, so using the last
estimate, \eqref{eq.term1ctasy} and \eqref{eq.term1ctasy} we have
\eqref{eq.ctasybgt01}. Notice also that $c_t\to 0$ as $t\to\infty$.

\textbf{Case 2: $2b/a=-1$.} If $2b/a=-1$, by  \eqref{eq.d2sUsasy}
and  \eqref{eq.r2asy} we have
\begin{multline*}
r_2(t) \int_0^t d_2(s)(1+s)U(1-b/a,2,-a(1+s))\,ds \sim
 b|a|^{b/a-3} t^{-1/2} \log t
\\ =  \frac{1}{2}|a|^{-5/2} t^{-1/2} \log t
 , \quad \text{as $t\to\infty$.}
\end{multline*}
Using this estimate, \eqref{eq.ctfirst} and \eqref{eq.term1ctasy},
together with the fact that $b/a=-1/2$, we have
\eqref{eq.ctasybgt02}. Notice also that $c_t\to 0$ as $t\to\infty$.

\textbf{Case 3: $2b/a>-1$.} If $2b/a>-1$, then by
\eqref{eq.d2sUsasy} and \eqref{eq.r2asy} we have
\[
r_2(t)\int_0^t d_2(s)(1+s)U(1-b/a,2,-a(1+s))\,ds
 \sim
  b|a|^{b/a-3}   \, t^{b/a} \frac{1}{2b/a+1}   \text{ as $t\to\infty$}.
\]
Using this estimate, \eqref{eq.ctfirst} and \eqref{eq.term1ctasy},
we have \eqref{eq.ctasyblt0}. Since $b>0$ and $a<0$, we have $c_t\to
0$ as $t\to\infty$.

Finally we prove part (c), or \eqref{eq.ctasyaplusb0}, in the case
that $a+b=0$.
We consider the asymptotic behaviour of the first term on the right
hand side of \eqref{eq.ctfirst}. We can still apply
\eqref{eq.term1ctasy} so that
\begin{multline*}
r_1(t)\int_0^t d_1(s)(1+s) U(1-b/a,2,|a|(1+s)) \,ds\\
\sim |a|^{b/a-2}
t^{b/a}=|a|^{b/a-2} t^{-1}, \quad\text{ as $t\to\infty$}.
\end{multline*}
Therefore
\begin{equation} \label{eq.ctfirstaplusbeq01}
\lim_{t\to\infty} r_1(t)\int_0^t d_1(s)(1+s) U(1-b/a,2,|a|(1+s))
\,ds = 0.
\end{equation}
Since $a+b=0$ and $r_2$ obeys \eqref{eq.r2asy}, we have
$r_2(t)\to\frac{1}{\Gamma(-b/a)}e^{-a}|a|^{-1-b/a}$ as $t\to\infty$.
Since $d_2$ is given by \eqref{eq:d2}, we have that
\begin{multline*}
\int_0^t d_2(s)(1+s)U(1-b/a,2,|a|(1+s))\,ds \\= \e^{a} b
\Gamma(-\ba) \int_0^t (1+s)^2 U^2(1-b/a,2,|a|(1+s))\,ds.
\end{multline*}
By \eqref{eq:chg3b}, we have that $(1+s)^2U^2(1-b/a,2,|a|(1+s))\sim
(|a|s)^{2b/a}=(|a|s)^{-2}$ as $s\to\infty$.
Therefore it follows that the integral tends to a finite limit and
therefore
\begin{multline*}
\lim_{t\to\infty} r_2(t) \int_0^t d_2(s)(1+s)U(1-b/a,2,-a(1+s))\,ds
 \\=
 |a|^{-1-b/a} b  \int_0^\infty (1+s)^2 U^2(1-b/a,2,|a|(1+s))\,ds.
\end{multline*}
Combining this limit with \eqref{eq.ctfirstaplusbeq01} and taking
the limit as $t\to\infty$ in \eqref{eq.ctfirst}, we obtain
\eqref{eq.ctasyaplusb0}.

%

\subsection{Proof of Theorem \ref{thm:acfshortmem}}
Let $t\geq 0$ and $\Delta\geq 0$. Suppose first that $b/a\not\in
\{1,2,\ldots\}$. Using \eqref{eq.autocovarianceaverage} and
\eqref{eq:resdecomp} one obtains
    \begin{multline}  \label{eq.shortmemmaster}
       \text{Cov}(X(t),X(t+\Delta)) \\
       = \sigma^2 r_{1}(t)r_{1}(t+\Delta)\int_{0}^{t}d_{1}^2(s) \rd s
         + \sigma^2 r_{1}(t)r_{2}(t+\Delta)\int_{0}^{t}d_{1}(s)d_2(s) \rd s \\
          + \sigma^2 r_{2}(t)r_{2}(t+\Delta)\int_{0}^{t}d_{2}^2(s) \rd s
         + \sigma^2 r_{1}(t+\Delta)r_{2}(t)\int_{0}^{t}d_{1}(s)d_2(s) \rd s.
    \end{multline}
Our plan is now to determine the exact asymptotic behaviour of each
of the four terms in \eqref{eq.shortmemmaster} as $t\to\infty$ (for
fixed $\Delta\geq 0$).
Since $a<0$ from \eqref{eq:d1asy} we have 
\[
d_1^2(t)\sim |a|^{-2b/a}  e^{-2at} t^{-2b/a}, \quad \text{ as
$t\to\infty$}.
\]
Therefore, one can use the last limit and l'H\^opital's rule to show
that
\[
\int_0^t d_1^2(s)\,ds \sim \frac{1}{2|a|}\cdot |a|^{-2b/a}  e^{-2at}
t^{-2b/a}, \quad \text{ as $t\to\infty$}.
\]
By \eqref{eq.r1asy}, and the above limit, we have
\begin{align}
\lefteqn{\lim_{t\to\infty}
r_1(t)r_1(t+\Delta)\int_0^t d_1(s)^2\,ds}\nonumber\\
&= e^{a\Delta} \lim_{t\to\infty} \biggl\{\frac{r_1(t)}{e^{at}
|a|^{b/a}t^{b/a}}\frac{r_1(t+\Delta)}{e^{a(t+\Delta)}
|a|^{b/a}(t+\Delta)^{b/a}} e^{2at} |a|^{2b/a} t^{b/a}
(t+\Delta)^{b/a}
\nonumber\\
&\quad\times \frac{1}{2|a|}\cdot |a|^{-2b/a}  e^{-2at} t^{-2b/a}
\frac{\int_0^t d_1(s)^2\,ds}{\frac{1}{2|a|}\cdot |a|^{-2b/a}  e^{-2at} t^{-2b/a}}\biggr\}\nonumber\\
\label{eq.term1cov} &= \frac{1}{2|a|}e^{a\Delta} \lim_{t\to\infty}
\biggl\{
(t+\Delta)^{b/a}
\cdot 
t^{-b/a}
\biggr\} = \frac{1}{2|a|}e^{a\Delta}.
\end{align}


For the second and fourth terms in \eqref{eq.shortmemmaster}, we use
\eqref{eq:d1asy} and \eqref{eq:d2asy} to get
\[
    \int_{0}^{t}d_{1}(s)d_{2}(s) \rd s \sim |a|^{-2}\e^{a} b \, \Gamma(-\ba) \e^{-at}, \quad \text{as $t\to\infty$}.
\]
Thus, using \eqref{eq.r1asy} and \eqref{eq.r2asy}, we get
\begin{align}
   \lefteqn{\lim_{t\to\infty} r_{1}(t)r_{2}(t+\Delta)\int_{0}^{t}d_{1}(s)d_2(s) \rd s}\nonumber\\
     &=
     \lim_{t\to\infty}
     \biggl\{
     \frac{r_1(t)}{e^{at} |a|^{b/a}t^{b/a}}\frac{r_2(t+\Delta)}{\frac{1}{\Gamma(-b/a)} e^{-a} |a|^{-b/a-1} (t+\Delta)^{-b/a-1}} \nonumber\\
     &\quad \times e^{at} |a|^{b/a}t^{b/a} \frac{1}{\Gamma(-b/a)} e^{-a} |a|^{-b/a-1} (t+\Delta)^{-b/a-1}\nonumber\\
    &\qquad \times
     |a|^{-2}\e^{a} b \, \Gamma(-\ba) \e^{-at}
    \frac{\int_{0}^{t}d_{1}(s)d_2(s) \rd s}{|a|^{-2}\e^{a} b \, \Gamma(-\ba) \e^{-at}}\biggr\}\nonumber\\
    \label{eq.term2cov}
    &=b|a|^{-3} \lim_{t\to\infty}
     t^{b/a}
     (t+\Delta)^{-b/a-1}
     =0.
    \end{align}
Similarly, we can show that the fourth term on the righthand side of
\eqref{eq.shortmemmaster} obeys
\begin{equation} \label{eq.term4cov}
\lim_{t\to\infty}r_{1}(t+\Delta)r_{2}(t)\int_{0}^{t}d_{1}(s)d_2(s)
\rd s =0.
\end{equation}


Finally, we consider the third term on the righthand side of
\eqref{eq.shortmemmaster}. Using \eqref{eq:d2asy} we have
\[
d_{2}^2(s) \sim \Gamma(-\ba)\e^{2a} b^2 |a|^{2\ba-2}\, s^{2b/a},
\text{ as $s\to\infty$}.
\]
If $2b/a<-1$, we have that $d_2^2\in L^1(0,\infty)$. In the case
that $a+b<0$, we have that $r_2(t)\to 0$ as $t\to\infty$, so
\begin{equation}\label{eq.thirdtermshort}
\lim_{t\to\infty}r_{2}(t)r_{2}(t+\Delta)\int_{0}^{t}d_{2}^2(s) \rd
s=0.
\end{equation}
In the case that $2b/a<-1$ and $a+b=0$, we have from
\eqref{eq.r2asy} that $r_2(t)\to
\frac{1}{\Gamma(-b/a)}|a|^{-b/a-1}e^{-a}$ as $t\to\infty$. Then from
\eqref{eq:d2} we have
\begin{multline} \label{eq.thirdtermaplusbeq0}
\lim_{t\to\infty}r_{2}(t)r_{2}(t+\Delta)\int_{0}^{t}d_{2}^2(s) \rd s
= \frac{1}{\Gamma(-b/a)^2}|a|^{-2b/a-2}e^{-2a}\int_0^\infty d_{2}^2(s) \rd s\\
= b^2|a|^{-2b/a-2}\int_0^\infty (1+s)^2 U^2(1-\ba,2,|a|(1+s)) \rd s.
\end{multline}

If $2b/a=-1$, we have that
\[
\int_0^t d_2^2(s)\,ds \sim \Gamma(-\ba)\e^{2a} b^2 |a|^{2\ba-2}\,
\log t, \quad \text{as $t\to\infty$.}
\]
Since $b/a=-1/2$, we have that $r_2(t)\sim kt^{-3/2}$ as
$t\to\infty$ for some $k\neq 0$, and therefore
\eqref{eq.thirdtermshort} holds. If $2b/a>-1$, then
\[
\int_0^t d_{2}^2(s)\,ds \sim \Gamma(-\ba)\e^{2a} b^2 |a|^{2\ba-2}\,
t^{2b/a+1}\frac{1}{2b+a}, \quad\text{ as $t\to\infty$}.
\]
Using \eqref{eq.r2asy} we have
\[
r_2(t)r_2(t+\Delta) \int_0^t d_{2}^2(s)\,ds \sim
\frac{1}{\Gamma(-b/a)}  |a|^{-4} b^2  \frac{1}{2b+a} t^{-1},
\]
as $t\to\infty$. Hence \eqref{eq.thirdtermshort} holds.

Next, in the case when $b/a\not\in \{1,2,\ldots\}$ and $a+b<0$, by
taking the limit as $t\to\infty$
 on both sides of \eqref{eq.shortmemmaster}, using the limits  \eqref{eq.term1cov}, \eqref{eq.term2cov} and \eqref{eq.term4cov}
 on the first, second and fourth terms, and \eqref{eq.thirdtermshort} on the third term on the righthand side of
  \eqref{eq.shortmemmaster}, we obtain \eqref{eq.covaplusblt0}.

On the other hand, when $a+b=0$, by taking the limit as $t\to\infty$
 on both sides of \eqref{eq.shortmemmaster}, using the limits  \eqref{eq.term1cov}, \eqref{eq.term2cov} and \eqref{eq.term4cov}
 on the first, second and fourth terms, and  \eqref{eq.thirdtermaplusbeq0} on the third term on the righthand side of
  \eqref{eq.shortmemmaster}, we obtain \eqref{eq.covaplusbeq0}.

%

%
For the case when $b/a\in\{1,2,...\}$, then one decomposes
$\Cov(X(t),X(t+\Delta))$ as in \eqref{eq.shortmemmaster} above but
where $\tilde{r}_2$, $\tilde{d}_1$ and $\tilde{d}_2$ play the role
of $r_2$, $d_1$ and $d_2$. Moreover as can be seen from
\eqref{eq:secr2}, \eqref{eq.d1tilasy} and \eqref{eq.d2tilasy},
$\tilde{r}_2,\tilde{d}_1$ and $\tilde{d}_2$ have the same asymptotic
behaviour as $r_2$, $d_1$ and $d_2$ (to within a multiplicative
constant) and so one can deduce the limits \eqref{eq.covaplusblt0}
and \eqref{eq.covaplusbeq0} as before.

\subsection{Proof of Remark~\ref{rem.acfvol}}
Since $a>\int_0^\infty k(s)\,ds$, we have that $r$ is in
$L^1(0,\infty)$, and moreover that $\int_0^\infty r(s)\,ds
=1/(a-\int_0^\infty k(s)\,ds)$. Therefore, we have that
\[
\lim_{t\to\infty} \text{Cov}(X(t),X(t+\Delta))=\sigma^2\int_0^\infty
r(s)r(s+\Delta)\,ds=:\gamma(\Delta).
\]
Next, suppose that $k$ is a subexponential function. Then
\[
\lim_{t\to\infty} \frac{r(t)}{k(t)}=\frac{1}{(a-\int_0^\infty
k(s)\,ds)^2}.
\]
See e.g.,\cite{apprey:2002}. We determine the asymptotic behaviour of $\gamma(\Delta)$ as
$\Delta\to\infty$ under the additional assumption that $k$ is
asymptotic to a decreasing function. We then have
\begin{multline*}
\frac{\gamma(\Delta)}{k(\Delta)}-\sigma^2\int_0^\infty r(s)\,ds\cdot
\frac{1}{(a-\int_0^\infty k(s)\,ds)^2}
\\=\sigma^2\int_0^\infty r(s)\left(\frac{r(s+\Delta)}{k(s+\Delta)}- \frac{1}{(a-\int_0^\infty k(s)\,ds)^2}\right)\cdot \frac{k(s+\Delta)}{k(\Delta)}\,ds\\
+\sigma^2\int_0^\infty r(s)
\left(\frac{k(s+\Delta)}{k(\Delta)}-1\right)\,ds \cdot
\frac{1}{(a-\int_0^\infty k(s)\,ds)^2}.
\end{multline*}
The first term has zero limit as $\Delta\to\infty$. The second term
can be shown to have a zero limit as $\Delta\to\infty$ by splitting
the integral into over the intervals $[0,T)$ and $[T,\infty)$ for
$T>0$ so large that $\int_T^\infty |r(s)|\,ds<\epsilon
(a-\int_0^\infty k(s)\,ds)^2$, where $\epsilon>0$ is taken
arbitrarily small. Then, letting $\Delta\to\infty$, we see that the
first of these two integrals tends to zero, while for the second
using the monotonicity of $k$, the limit superior of the absolute
value is less than $2\sigma^2 \epsilon$. Letting $\epsilon\to 0$
confirms that
 \begin{equation*}
\lim_{\Delta\to\infty} \frac{\gamma(\Delta)}{k(\Delta)}=\sigma^2
\frac{1}{(a-\int_0^\infty k(s)\,ds)^3}.
\end{equation*}
Now we fix $t$ and compute the autocovariance function. We have
\begin{multline*}
\lim_{\Delta\to\infty}\frac{\text{Cov}(X(t),X(t+\Delta))}{\gamma(\Delta)}
=\lim_{\Delta\to\infty} \frac{k(\Delta)}{\gamma(\Delta)}\cdot
\sigma^2\int_0^t r(s)\frac{r(s+\Delta)}{k(s+\Delta)}\cdot
\frac{k(s+\Delta)}{k(\Delta)}\,ds
\\=\frac{(a-\int_0^\infty k(s)\,ds)^3}{\sigma^2} \sigma^2\frac{1}{(a-\int_0^\infty k(s)\,ds)^2}
\int_0^t r(s)\,ds.
\end{multline*}
Therefore, we have
\[
\lim_{\Delta\to\infty}\frac{\text{Cov}(X(t),X(t+\Delta))}{\gamma(\Delta)}
=\frac{\int_0^t r(s)\,ds}{\int_0^\infty r(s)\,ds}=:c_t,
\]
so clearly $c_t\to 1$ as $t\to\infty$. Therefore the autocovariance
function obeys \eqref{eq.equallimacvfaffine}.

\subsection{Proof of Remark~\ref{rem.acfSFDE}}
If $r(t)\to 0$ as $t\to\infty$, it is known that $r\in
L^1(0,\infty)$ and that $r$ decays to zero exponentially. As a
consequence
\[
\lim_{t\to\infty} \text{Cov}(X(t),X(t+\Delta))=\sigma^2\int_0^\infty
r(s)r(s+\Delta)\,ds=:\gamma(\Delta).
\]
Let us further suppose, for example, that $r$ is asymptotic to a
function of one sign. Then there exists $n\in\mathbb{Z}^+$ and
$\alpha>0$ such that $r(t)/(t^{n-1} e^{-\alpha t})\to C\neq 0$ as
$t\to\infty$. We now determine the asymptotic behaviour of
$\gamma(\Delta)$ as $\Delta\to\infty$. We start by writing
\begin{multline*}
\frac{\gamma(\Delta)}{\Delta^{n-1} e^{-\alpha \Delta}}-\sigma^2
C\int_0^\infty e^{-\alpha s} r(s)\,ds
\\=
\sigma^2\int_0^\infty e^{-\alpha s}
r(s)\left(\frac{r(s+\Delta)}{(s+\Delta)^{n-1} e^{-\alpha
(\Delta+s)}}-C\right) \cdot \frac{(s+\Delta)^{n-1} }{\Delta^{n-1}
}\,ds
\\+
\left\{ \sigma^2 C\int_0^\infty e^{-\alpha s} r(s) \cdot
\frac{(s+\Delta)^{n-1} }{\Delta^{n-1} }\,ds-\sigma^2 C\int_0^\infty
e^{-\alpha s} r(s)\,ds\right\}.
\end{multline*}
It can then be shown that the limits as $\Delta\to\infty$ of the two
terms on the righthand side is zero, so that
\[
\lim_{\Delta\to\infty} \frac{\gamma(\Delta)}{\Delta^{n-1} e^{-\alpha
\Delta}} =\sigma^2 C\int_0^\infty e^{-\alpha s} r(s)\,ds=:c^\ast.
\]
Considering now the limit when $\Delta\to\infty$ for $t$ fixed, we
have
\begin{align*}
\lefteqn{\frac{\text{Cov}(X(t),X(t+\Delta))}{\gamma(\Delta)}}\\
&= \frac{\text{Cov}(X(t),X(t+\Delta))}{\Delta^{n-1} e^{-\alpha
\Delta}}
\cdot\frac{\Delta^{n-1} e^{-\alpha \Delta}}{\gamma(\Delta)}\\
&=\sigma^2\int_0^t  r(s)e^{-\alpha s}
\frac{r(s+\Delta)}{(s+\Delta)^{n-1} e^{-\alpha (\Delta+s)}} \cdot
\frac{(s+\Delta)^{n-1} }{\Delta^{n-1} }\,ds \cdot \frac{\Delta^{n-1}
e^{-\alpha \Delta}}{\gamma(\Delta)}.
\end{align*}
Therefore we have
\[
\lim_{\Delta\to\infty}
\frac{\text{Cov}(X(t),X(t+\Delta))}{\gamma(\Delta)} =
\frac{1}{c^\ast}C\sigma^2\int_0^t  r(s)e^{-\alpha s} \,ds=:c_t.
\]
We see that $c_t\to 1$ as $t\to\infty$. Therefore
\eqref{eq.equallimacvfaffine} holds.
\section{Proof of Results in Section~\ref{sect:trans}} \label{sect:transproofs}
In this section, we give the proofs of the growth rates of $X$
stated in Section~\ref{sect:trans}.
\subsection{Proof of Theorem \ref{thm.bubbleaveexp}}
For $b/a\not\in\{-1,-2,...\}$, from \eqref{eq:detspec3},
\eqref{eq:resdecomp2} and \eqref{eq.xrep}, we can write $X$
according to 
\begin{equation}\label{eq:stochdecomp}
    X(t) = r_{3}(t)c_{3} + r_{4}(t)c_{4} + \sigma r_{3}(t)\int_{0}^{t}d_{3}(s)\,dB(s) + \sigma r_{4}(t)\int_{0}^{t}d_{4}(s)\,dB(s).
\end{equation}
We have already deduced the asymptotic behaviour of $r_3$, $r_4$,
$d_3$ and $d_4$ in \eqref{eq.r3asy}, \eqref{eq.r4asy},
\eqref{eq.d3asy} and \eqref{eq.d4asy}. We recapitulate their
limiting behaviour now:
\begin{align*}
    r_{3}(t) \sim a^{-1-\ba}t^{-1-\ba}, &\quad r_{4}(t)\sim \frac{1}{\Gamma(1+\ba)}\e^{a(1+t)}a^{\ba}t^{\ba},
     \quad \text{ as $t\to\infty$},\\
    d_{3}(s) \sim b \, a^{-1+\ba}s^{\ba},  &\quad d_{4}(s) \sim \Gamma(1+\ba)a^{-\ba}\e^{-a(1+s)}s^{-\ba}, \quad \text{ as $s\to\infty$}.
\end{align*}
Dividing across \eqref{eq:stochdecomp} by $r_4(t)$ yields
\begin{equation} \label{eq:stochdecomp2}
    \frac{X(t)}{r_4(t)} = \frac{r_3(t)}{r_4(t)}c_3 + c_4 + \sigma \frac{r_3(t)}{r_4(t)}\int_0^t d_3(s)\,dB(s) + \sigma\int_0^t d_4(s)\,dB(s).
\end{equation}
The asymptotic behaviour of the first and last terms is readily
estimated. Since $a>0$, $r_{3}(t)/r_4(t)\to 0$ as $t\to\infty$.
$a>0$ also implies $d_{4}\in L^2(0,\infty)$. Therefore by the
Martingale Convergence Theorem for continuous martingales (cf.,
e.g.,~\cite[Thm. V.1.8]{RevYor}) we have that
\[
\lim_{t\to\infty} \sigma\int_{0}^{t}d_4(s) \rd
B(s)=\sigma\int_{0}^{\infty}d_4(s) \rd B(s),\quad\text{a.s.}
\]
If
\begin{equation}\label{eq:r4domr3}
    \lim_{t\to\infty}\frac{r_3(t)}{r_4(t)}\int_0^t d_3(s) \rd B(s) =0, \quad \text{a.s.}
\end{equation}
then we obtain
\[
    \lim_{t\to\infty}\frac{X(t)}{r_{4}(t)} = c_{4} + \sigma\int_{0}^{\infty}d_4(s) \rd B(s)=:C_4, \quad \text{a.s.}
\]
By \eqref{eq.r4asy}
we therefore have
\[
    \lim_{t\to\infty}\frac{X(t)}{\e^{at}t^{\ba}}
    = \Gamma(1+\ba)e^a a^{\ba}c_{4} + \Gamma(1+\ba)e^a a^{\ba}\sigma\int_{0}^{\infty}d_4(s) \rd B(s)=C, \quad \text{a.s.}
\]
which implies \eqref{eq.Xaveexpbubble} and also part (b), due to the
definitions of $c_4$ and $d_4$ in \eqref{eq.c4} and \eqref{eq:d4}
and of $C$ in part (b).

Moreover, it follows from \cite[Ch. 2.13.5, p.304-305]{Shir} that
\[
    \mathbb{E}[C]= \lim_{t\to\infty}\frac{\mathbb{E}\left[X(t)\right]}{\e^{at}t^{\ba}}=\lim_{t\to\infty}\frac{x(t)}{\e^{at}t^{\ba}}
    =c_4\Gamma(1+\ba)e^a a^{\ba},
\]
and that
\[
\Var[C] = \lim_{t\to\infty}\frac{\text{Var}[X(t)]}{\e^{2at}t^{2\ba}}
= \sigma^2 \Gamma(1+\ba)^2 e^{2a}
a^{2\ba}\int_{0}^{\infty}d_{4}^2(s) \rd s>0.
\]
These results and  \eqref{eq.c4} and \eqref{eq:d4} establish the
validity of parts (c) and (d).

All that remains to show is that \eqref{eq:r4domr3} is indeed true.
%
%
If $\bba<-1$, then $d_3\in L^2(0,\infty)$, and the stochastic
integral tends to a finite limit by the Martingale Convergence
Theorem. Since $r_3(t)/r_4(t)\to 0$ as $t\to\infty$, we obtain
\[
    \lim_{t\to\infty}\frac{r_{3}(t)}{r_{4}(t)}\int_{0}^{t}d_{3}(s) \rd B(s) =0, \quad \text{a.s.}
\]
If $\bba>-1$, then $d_3\not\in L^2(0,\infty)$. Indeed, the quadratic
variation of $\int_{0}^t d_{3}(s)dB(s)$ is given by
\[
    v(t) := \int_{0}^{t}d_{3}^2(s) \rd s \sim b^2a^{\bba}(\bba+1)^{-1} t^{\bba+1}, \quad\text{as $t\to\infty$},
\]
and hence $ \log\log v(t) \sim \log\log t$ as $t\to\infty$.
Therefore the stochastic integral  $\int_{0}^t d_{3}(s)dB(s)$ obeys
the Law of the Iterated Logarithm for continuous martingales (cf.,
e.g.,~\cite[Exercise V.1.15]{RevYor}), so
\[
    \limsup_{t\to\infty}\frac{\int_{0}^{t}d_{3}(s)dB(s)}{\sqrt{2v(t)\log\log v(t)}}
    = -\liminf_{t\to\infty}\frac{\int_{0}^{t}d_{3}(s)dB(s)}{\sqrt{2v(t)\log\log v(t)}} =1, \quad\text{a.s.}
\]
These asymptotic estimates for the stochastic integral and $v$,
together with \eqref{eq.r3asy} and \eqref{eq.r4asy} yield
\[
    \lim_{t\to\infty} \frac{r_{3}(t)}{r_4(t)}\int_{0}^{t}d_{3}(s)dB(s) =0, \quad \text{a.s.}
\]
as required. The above argument holds similarly for the case when
$\bba=-1$.


The case $b/a\in\{-1,-2,-3,...\}$ can be dealt with similarly. While
we only have the crude estimate \eqref{eq.d3tilasy} for the
asymptotic behaviour of $\tilde{d}_3$, it is nevertheless the case
that the quadratic variation of $\int_{0}^{t}\tilde{d}_3(s)dB(s)$
can grow no faster than a power of $t$ as $t\to\infty$ (or indeed
may converge as $t\to\infty$). Thus we obtain
\begin{equation*}
    \lim_{t\to\infty}\frac{r_{3}(t)}{\tilde{r}_{4}(t)}\int_{0}^{t}\tilde{d}_{3}(s) \rd B(s) =0, \quad \text{a.s.}
\end{equation*}
as before.


\subsection{Proof of Theorem \ref{cor:stochexp}}
Since $a<0$ and $a+b>0$, we have $b/a\not\in \{1,2,\ldots\}$.
Therefore, from \eqref{eq:detspec1}, \eqref{eq:resdecomp} and
\eqref{eq.xrep} one has,
\begin{equation}\label{eq:stochdecomptrans}
    X(t) = r_{1}(t)c_{1} + r_{2}(t)c_{2} + \sigma r_{1}(t)\int_{0}^{t}d_{1}(s)dB(s) + \sigma r_{2}(t)\int_{0}^{t}d_{2}(s)dB(s).
\end{equation}
We have already deduced the asymptotic behaviour of $r_1$, $r_2$,
$d_1$ and $d_2$ in \eqref{eq.r1asy}, \eqref{eq.r2asy},
\eqref{eq:d1asy} and \eqref{eq:d2asy}. We recapitulate their
limiting behaviour now:
\begin{align*} 
    r_{1}(t) \sim \e^{at} |a|^{\ba} t^{\ba}, &\quad r_{2}(t)\sim \frac{\e^{-a}}{\Gamma(-\ba)} |a|^{-1-\ba}t^{-1-\ba},
     \quad \text{ as $t\to\infty$}, \\
    d_{1}(s) \sim |a|^{-\ba} \e^{-as} s^{-\ba},  &\quad d_{2}(s) \sim \Gamma(-\ba) b \, \e^{a} |a|^{-1+\ba} s^{\ba}, \quad \text{ as $s\to\infty$}. 
\end{align*}
Dividing across \eqref{eq:stochdecomptrans} by $r_2(t)$ yields
\begin{equation} \label{eq:stochdecomp3}
    \frac{X(t)}{r_2(t)} = \frac{r_1(t)}{r_2(t)}c_1 + c_2 + \sigma \frac{r_1(t)}{r_2(t)}\int_0^t d_1(s)\,dB(s) + \sigma\int_0^t d_2(s)\,dB(s).
\end{equation}
The asymptotic behaviour of the first and last terms is readily
estimated. Since $a<0$, we have from \eqref{eq.r1asy} and
\eqref{eq.r2asy} that $r_1(t)/r_2(t)\to 0$ as $t\to\infty$. Also,
since $a<0$ and $a+b>0$, we have $2b/a<-2$. Hence $d_2\in
L^{2}(0,\infty)$ and therefore by the martingale convergence theorem
for continuous martingales (cf., e.g., \cite[Thm. V.1.8]{RevYor}) we
have
\begin{equation}\label{eq:coeffl2}
    \lim_{t\to\infty}\int_{0}^{t}d_{2}(s) dB(s) = \int_{0}^{\infty}d_{2}(s) dB(s),\quad  \text{a.s}.
\end{equation}
We now examine the asymptotic behaviour of the third term on the
righthand side of \eqref{eq:stochdecomp3}. Firstly observe that
$\int_{0}^{t}d_{1}(s) dB(s)$ is normally distributed with mean zero
and variance given by
\[
    v_{1}(t) = \int_{0}^{t}d_{1}^{2}(s) ds.
\]
By l'H\^{o}pital's rule we have
\[
    v_{1}(t) \sim \frac{1}{2}|a|^{-1-\bba}\e^{-2at} (1+t)^{-\bba}, \quad \log\log v_{1}(t) \sim \log t,\text{ as $t\to\infty$},
\]
and so we have by the Law of the Iterated Logarithm for continuous
martingales (cf., e.g.,~\cite[Exercise V.1.15]{RevYor}) that
\[
    \limsup_{t\to\infty}\frac{\int_{0}^{t}d_{1}(s)dB(s)}{\sqrt{2v_1(t)\log\log v_1(t)}}
    = -\liminf_{t\to\infty}\frac{\int_{0}^{t}d_{1}(s)dB(s)}{\sqrt{2v_1(t)\log\log v_1(t)}} =1, \quad\text{a.s.}
\]
Thus we have
\begin{equation}\label{eq:1ststochterm}
    \limsup_{t\to\infty} \sigma\frac{r_{1}(t)\int_{0}^{t}d_{1}(s)\rd B(s)}{\sqrt{\log t}}
    = -\liminf_{t\to\infty} \sigma\frac{r_{1}(t)\int_{0}^{t}d_{1}(s)\rd B(s)}{\sqrt{\log t}}
    = \frac{\sigma}{\sqrt{|a|}}.
\end{equation}
Using \eqref{eq:1ststochterm}, the fact that $\log t / r_{2}(t) \to
0$ as $t\to\infty$, together with \eqref{eq:coeffl2}, we arrive at
\begin{equation*} 
    \lim_{t\to\infty} \frac{X(t)}{r_{2}(t)} = c_{2} + \sigma\int_{0}^{\infty}d_{2}(s) \rd
    B(s), \quad\text{a.s.}
\end{equation*}
By \eqref{eq.r2asy} we therefore obtain
\begin{multline}\label{eq:polylim}
    \lim_{t\to\infty} \frac{X(t)}{t^{-1-\ba}}
    \\= \frac{\e^{-a}}{\Gamma(-\ba)} |a|^{-1-\ba}c_{2} + \sigma \frac{\e^{-a}}{\Gamma(-\ba)} |a|^{-1-\ba}\int_{0}^{\infty}d_{2}(s) \rd
    B(s)=C, \quad\text{a.s.}
\end{multline}
which implies part (a) and also part (b), due to the definitions of
$c_2$ and $d_2$ in \eqref{eq.c2} and \eqref{eq:d2} and of $C$ in
part (b).

Moreover, it follows from \cite[Ch. 2.13.5, p.304-305]{Shir} that
\[
    \mathbb{E}[C]= \lim_{t\to\infty}\frac{\mathbb{E}\left[X(t)\right]}{t^{-1-\ba}}=\lim_{t\to\infty}\frac{x(t)}{t^{-1-\ba}}
    =c_2 \frac{\e^{-a}}{\Gamma(-\ba)} |a|^{-1-\ba},
\]
and that
\[
\Var[C] = \lim_{t\to\infty}\frac{\text{Var}[X(t)]}{t^{-2-2\ba}} =
\sigma^2 \frac{\e^{-2a}}{\Gamma^2(-\ba)}
|a|^{-2-2\ba}\int_{0}^{\infty}d_{2}^2(s)\,ds>0
\]
These results and  \eqref{eq.c2} and \eqref{eq:d2} establish the
validity of parts (c) and (d).
%
%

\subsection{
Proof of Theorem \ref{thm:modBes}}
From \eqref{eq:detspec5}, \eqref{eq:resdecomp6} 
 and \eqref{eq.xrep}, we can write $X$
according to 
\begin{equation}\label{eq:stochdecomp4}
    X(t) = r_{5}(t)c_{5} + r_{6}(t)c_{6} + \sigma r_{5}(t)\int_{0}^{t}d_{5}(s)\,dB(s) + \sigma r_{6}(t)\int_{0}^{t}d_{6}(s)\,dB(s).
\end{equation}
We can deduce the asymptotic behaviour of $r_5$, $r_6$, $d_5$ and
$d_6$ using \eqref{eq:modBesasy} and \eqref{eq:r5r6}.  Hence
\begin{gather*}
    r_5(t) \sim \frac{1}{2 b^{1/4}\sqrt{\pi}}\e^{2\sqrt{bt}} t^{-1/4}, \text{ as $t\to\infty$}, \quad 
    r_6(t) \sim \frac{\sqrt{\pi}}{2 b^{1/4}} \e^{-2\sqrt{bt}} t^{-1/4}, \text{ as $t\to\infty$}, \\
    d_5(s)
    \sim \sqrt{\pi} b^{1/4} s^{1/4}\e^{-2\sqrt{bs}}, 
    \text{ as $s\to\infty$}, \quad
    d_6(s)
    \sim \frac{1}{\sqrt{\pi}} b^{1/4} s^{1/4}\e^{2\sqrt{bs}}, \text{ as $s\to\infty$}.
\end{gather*}
Dividing across \eqref{eq:stochdecomp4} by $r_5(t)$ yields
\begin{equation} \label{eq:stochdecomp5}
    \frac{X(t)}{r_5(t)} = c_5 + \frac{r_6(t)}{r_5(t)}c_6 + \sigma\int_0^t d_5(s)\,dB(s) + \sigma \frac{r_6(t)}{r_5(t)}\int_0^t d_6(s)\,dB(s).
\end{equation}
The asymptotic behaviour of the second and third terms is readily
estimated. First as $b>0$, $r_6(t)/r_5(t)\to 0$ as $t\to\infty$.
Also, $d_5\in L^{2}(0,\infty)$ and therefore by the martingale
convergence theorem for continuous martingales (cf., e.g.,
\cite[Thm. V.1.8]{RevYor}) we have
\begin{equation}\label{eq:coeffd5l2}
    \lim_{t\to\infty}\int_{0}^{t}d_{5}(s) dB(s) = \int_{0}^{\infty}d_{5}(s) dB(s),\quad  \text{a.s}.
\end{equation}
We now examine the asymptotic behaviour of the fourth term on the
righthand side of \eqref{eq:stochdecomp5}. Firstly observe that
$\int_{0}^{t}d_{6}(s) dB(s)$ is normally distributed with mean zero
and variance given by
\[
    v_3(t) := \int_{0}^{t}d_6(s)^2 \,ds.
\]
By using l'H\^opital's rule, the asymptotic behaviour of $v_3(t)$ as
$t\to\infty$ can be found:
\[
    \lim_{t\to\infty}\frac{v_3(t)}{t\e^{4\sqrt{bt}}} = \frac{1}{2\pi}, 
\quad
    \lim_{t\to\infty}\frac{\log \log v_3(t)}{\log t}=\frac{1}{2}.
\]
Thus by the Law of the Iterated Logarithm for continuous martingales
(cf., e.g.,~\cite[Exercise V.1.15]{RevYor}) we have that
\[
    \limsup_{t\to\infty}\frac{\int_{0}^{t}d_{6}(s)dB(s)}{\sqrt{2v_3(t)\log\log v_3(t)}}
    = -\liminf_{t\to\infty}\frac{\int_{0}^{t}d_{6}(s)dB(s)}{\sqrt{2v_3(t)\log\log v_3(t)}} =1, \quad\text{a.s.}
\]
Thus we have
\[
\int_{0}^{t}d_6(s)\,dB(s) = O\left(t^{1/2} \e^{2\sqrt{bt}} \sqrt{\log
t}\right), \quad \text{ as $t\to\infty$}.
\]
Therefore
\begin{equation*}
\frac{r_6(t)}{r_5(t)}\int_{0}^{t}d_6(s)\,dB(s) =O\left(t^{1/2}
\e^{-2\sqrt{bt}} \sqrt{\log t}\right), \quad \text{ as
$t\to\infty$},
\end{equation*}
and so
\begin{equation}\label{eq:1ststochterma0}
    \lim_{t\to\infty}\frac{r_6(t)}{r_5(t)} \int_{0}^{t}d_6(s)\,dB(s) =0 \quad \text{a.s.}
\end{equation}
Taking the limit as $t\to\infty$ in \eqref{eq:stochdecomp5} and
using \eqref{eq:1ststochterma0} together with \eqref{eq:coeffd5l2},
we arrive at
\begin{equation*} 
    \lim_{t\to\infty} \frac{X(t)}{r_5(t)} = c_5 + \sigma\int_{0}^{\infty}d_5(s) \,dB(s), \quad\text{a.s.}
\end{equation*}
Using the asymptotic behaviour of $r_5$ we therefore obtain
\begin{multline}\label{eq:exppolylim}
    \lim_{t\to\infty} \frac{X(t)}{\e^{2\sqrt{bt}} t^{-1/4}}
    =\lim_{t\to\infty}\frac{X(t)}{r_{5}(t)}\cdot \frac{r_5(t)}{\e^{2\sqrt{bt}} t^{-1/4}}
    \\=
      \frac{1}{2 b^{1/4}\sqrt{\pi}} \left( c_5 + \sigma\int_{0}^{\infty}d_5(s)\,dB(s)\right)
    =C, \quad\text{a.s.}
\end{multline}
which implies part (a) and also part (b), due to the definitions of
$c_5$ and $d_5$ in \eqref{eq:cntsc5} and \eqref{eq:cntsd5} and of
$C$ in part (b).

Moreover, it follows from \cite[Ch. 2.13.5, p.304-305]{Shir} that
\[
    \mathbb{E}[C]= \lim_{t\to\infty}\frac{\mathbb{E}\left[X(t)\right]}{\e^{2\sqrt{bt}} t^{-1/4}}
    =\lim_{t\to\infty}\frac{x(t)}{\e^{2\sqrt{bt}} t^{-1/4}}
    = \frac{1}{2 b^{1/4}\sqrt{\pi}}  c_5,
\]
and that
\[
\Var[C] = \lim_{t\to\infty}\frac{\text{Var}[X(t)]}{\e^{4\sqrt{bt}}
t^{-1/2}} =\frac{1}{4 b^{1/2} \pi}
\sigma^2\int_{0}^{\infty}d_5^2(s)\,ds>0.
\]
These results and \eqref{eq:cntsc5} and \eqref{eq:cntsd5} establish
the validity of parts (c) and (d).

\section{Proof of Theorem~\ref{thm.averagegaussian} and Theorem~\ref{thm:XUrate}} \label{sect:recproofs}
We note that similar asymptotic analysis as that above would give
us, for $a+b\leq0$,
\[
    \limsup_{t\to\infty}\frac{X(t)}{\sqrt{2\log t}} = \frac{\sigma}{\sqrt{2|a|}}.
\]
We choose however to prove this result via Theorem
\ref{thm.averagegaussian}, as it provides an interesting result
regarding the asymptotic behaviour of the process.

\subsection{A preliminary lemma}
\begin{lemma}\label{lm:Lvar}
    Let $a<0$ and $a+b=0$. Define $H$ by
    \begin{equation} \label{eq:Haplusb0}
        H(t,u)=\int_{u}^{t}d_2(s)\frac{b}{1+s}\e^{-au}\int_{u}^{s}\sigma\e^{aw} \,dw\,ds, \quad 0\leq u\leq t,
    \end{equation}
    where $d_2$ is as given by \eqref{eq:d2}. Define $H_\infty$
\begin{equation} \label{eq:Hinftyaplusb0}
    H_\infty(u) = \frac{\sigma}{|a|}\int_{u}^{\infty}\frac{d_2(s)}{1+s} ds
    - \frac{\sigma}{|a|}\e^{-au}\int_{u}^{\infty}\e^{as}\frac{d_2(s)}{1+s} ds, \quad u\geq 0.
\end{equation}
    Then
    \[
        \lim_{t\to\infty}\int_{0}^{t}H(t,u)dB(u) =
        \int_0^\infty H_\infty(u)\,dB(u), \quad\text{a.s.}
    \]
\end{lemma}
\begin{proof}
The proof of this almost sure convergence result is an application of Theorem 7 in~\cite{appdanlimop}. 
$H$ simplifies to
\[
    H(t,u) = \frac{\sigma}{|a|}\int_{u}^{t}\frac{d_2(s)}{1+s}ds - \frac{\sigma}{|a|}\e^{-au}\int_{u}^{t}\e^{as}\frac{d_2(s)}{1+s} ds.
\]
$H_\infty$ given by \eqref{eq:Hinftyaplusb0} is well--defined by
virtue of \eqref{eq:d2asy}. To estimate the rate of decay of
$H_\infty$ to  zero, we use \eqref{eq:d2asy} to get
\begin{subequations}\label{eq:asyfg}
\begin{align}
    \int_{u}^{\infty}\frac{d_2(s)}{1+s}ds &\sim |a|^{-1}\e^{a}u^{-1}, \quad \quad \text{ as $u\to\infty$},  \\
    \e^{-au}\int_{u}^{\infty}\e^{as}\frac{d_2(s)}{1+s} ds &\sim |a|^{-2}\e^{a}u^{-2}, \quad \text{ as $u\to\infty$}.
\end{align}
\end{subequations}
Thus $H_\infty(u)\sim\sigma|a|^{-2}\e^{a}u^{-1}$ as $u\to\infty$ and
so $H_\infty\in L^2(0,\infty)$.

We now wish to show that
\begin{equation}\label{eq:HmHi}
    \lim_{t\to\infty}\int_{0}^{t}\left(H(t,u)-H_\infty(u) \right)^2 du\cdot\log t =0.
\end{equation}
Define
\[
    f(t):= \frac{\sigma}{|a|}\int_{t}^{\infty}\frac{d_2(s)}{1+s}ds, \quad g(t):=\frac{\sigma}{|a|}\int_{t}^{\infty}\e^{as}\frac{d_2(s)}{1+s} ds.
\]
Then the Cauchy-Schwarz inequality gives
\begin{align*}
    \int_{0}^{t}\left(H(t,u)-H_\infty(u) \right)^2 du &= \int_{0}^{t}\left(-f(t)+\e^{-au} g(t) \right)^2 du \\
    &\leq \int_{0}^{t}2 f(t)^2 du + \int_{0}^{t} 2\e^{-2au} g(t)^2 du \\
    &= 2t f(t)^2 + \frac{1}{2|a|}2 g(t)^2 \left(\e^{-2at} - 1 \right).
\end{align*}
The asymptotic relations \eqref{eq:asyfg} determine completely the
asymptotic behaviour of $f$ and $g$, and this, together with the
last inequality, gives \eqref{eq:HmHi}.

We show now that there exist $q\geq0$ and $c_q>0$ such that
\begin{equation}\label{eq:delH}
    \int_{0}^{t}\biggl[\frac{\partial}{\partial t} H(t,u) \biggr]^2 du \leq
    c_q(1+t)^{2q}, \quad t\geq 0.
\end{equation}
To do this we estimate according to
\begin{align*}
    \int_{0}^{t}\left[\frac{\partial}{\partial t} H(t,u) \right]^2 du
    &= \frac{\sigma^2}{|a|^2} \frac{d_2^2(t)}{(1+t)^2} \int_{0}^{t}  \left(1 - \e^{-au}\e^{at}\right)^2
    du,
\end{align*}
and using \eqref{eq:d2asy}, we see that $H$ obeys \eqref{eq:delH}
for any $q\geq0$ and $c_q>0$.
Also as $H(t,t)=0$ for all $t\geq0$ then all of the conditions of~\cite[Theorem 7]{appdanlimop}  
are satisfied and so we conclude $\lim_{t\to\infty}\int_0^t
H(t,u)dB(u)=\int_{0}^{\infty}H_{\infty}(u)dB(u)$ a.s. as required.
\end{proof}

\subsection{Proof of Theorem~\ref{thm.averagegaussian}}
We start by defining a process $Y=\{Y(t):t\geq -1\}$, which is
related to $U$ defined by \eqref{eq.OU}. It will be used in proving
Theorems~\ref{thm.averagegaussian} and \ref{cor:stochexp}. $Y$ is
defined by $Y(t)=\psi(t)$ for $t\in[-1,0]$ and it obeys
\begin{equation}\label{eq:Y}
    dY(t)=aY(t)\,dt+\sigma\,dB(t), \quad t\geq 0.
\end{equation}
Note that \eqref{eq.ldXave} is an immediate consequence of
\eqref{eq.ouaveasyequal1} or \eqref{eq.ouaveasyequal2} and the fact
that
\begin{equation} \label{eq.oustaldev}
\limsup_{t\to\infty} \frac{U(t)}{\sqrt{2\log
t}}=\frac{\sigma}{\sqrt{2|a|}}, \quad \liminf_{t\to\infty}
\frac{U(t)}{\sqrt{2\log t}}=-\frac{\sigma}{\sqrt{2|a|}},
\quad\text{a.s.}
\end{equation}
Therefore it remains to prove \eqref{eq.ouaveasyequal1} and
\eqref{eq.ouaveasyequal2}. Firstly extend $U$ to $[-1,0)$ by
$U(t)=0$ for $t\in[-1,0)$. Then for $Y$ defined by \eqref{eq:Y},
for $t\geq 0$ we have $Y(t)-U(t)=\psi(0)e^{at}$. Therefore
$U(t)-Y(t)\to 0$ as $t\to\infty$, a.s. Hence it remains to prove
that $X(t)-Y(t)\to 0$ as $t\to\infty$ a.s. in order to establish
\eqref{eq.ouaveasyequal1} and \eqref{eq.ouaveasyequal2}.

Define $Z(t)=X(t)-Y(t)$ for $t\geq -1$. Then $Z(t)=0$ for
$t\in[-1,0]$ and
\begin{equation} \label{eq.diffave}
Z'(t)=aZ(t)+b\frac{1}{1+t}\int_0^t Z(s)\,ds + f(t), \quad t>0,
\end{equation}
where
\begin{equation} \label{def.f}
f(t):=b\frac{1}{t+1}\int_{-1}^0 \psi(s)\,ds+  b\frac{1}{t+1}\int_0^t
Y(s)\,ds, \quad t\geq 0.
\end{equation}
Next we show that $f(t)\to 0$ as $t\to\infty$, a.s. This clearly
follows if $\int_0^t Y(s)\,ds/t\to 0$ as $t\to\infty$ a.s. To prove
this,  note that
\begin{equation} \label{eq.ouint}
Y(t)=\psi(0)+a\int_0^t Y(s)\,ds +\sigma B(t), \quad t\geq 0.
\end{equation}
Since $U$ obeys \eqref{eq.oustaldev}, $Y(t)-U(t)\to0$ as
$t\to\infty$, $Y$ obeys
\[
\limsup_{t\to\infty} \frac{|Y(t)|}{\sqrt{2\log
t}}=\frac{\sigma}{\sqrt{2|a|}}\quad\text{a.s.}
\]
Therefore by this limit and the strong law of large numbers for
standard Brownian motion \cite[2.9.3]{K&S}, we get from
\eqref{eq.ouint} that $\int_0^t Y(s)\,ds/t\to 0$ as $t\to\infty$
a.s., and therefore that $f(t)\to0$ as $t\to\infty$ a.s. Indeed, by
using the Law of the iterated logarithm for standard Brownian motion
\cite{K&S}, for every $\epsilon>0$, we have
\begin{equation} \label{eq:orderf}
    \limsup_{t\to\infty} \frac{f(t)}{t^{-1/2}\sqrt{2\log\log t}}
    = -\liminf_{t\to\infty} \frac{f(t)}{t^{-1/2}\sqrt{2\log\log t}} = \frac{|b|\sigma}{|a|}, \quad\text{a.s.}
\end{equation}
Recalling that the resolvent $r$ obeys \eqref{eq.resolv}, by
applying the conventional variation of constants formula to
\eqref{eq.diffave}, and using \eqref{eq:resdecomp} in the case that
$b/a\not\in \{1,2,\ldots\}$, we get
\begin{equation}  \label{eq:Zr1r2}
    Z(t) = \int_{0}^{t}r(t,s)f(s) \rd s = r_{1}(t)\int_{0}^{t}d_{1}(s)f(s)ds + r_{2}(t)\int_{0}^{t}d_{2}(s)f(s)ds
\end{equation}
and hence
\begin{equation}\label{eq:modZ}
    |Z(t)| \leq |r_{1}(t)|\int_{0}^{t}|d_{1}(s)||f(s)|ds + |r_{2}(t)|\int_{0}^{t}|d_{2}(s)||f(s)|\,ds.
\end{equation}
The first integral on the righthand side of \eqref{eq:modZ}
converges to zero using \eqref{eq.r1asy}, \eqref{eq:d1asy} and
\eqref{eq:orderf}, on application on l'H\^{o}pital's rule.

It transpires that the limiting behaviour as $t\to\infty$ of the
second integral on the righthand side of \eqref{eq:modZ} differs
according to whether $a+b<0$ or $a+b=0$. We consider first the case
when $a+b<0$. Using \eqref{eq:d2asy} and \eqref{eq:orderf} in the
case that $2b+a>0$, there exists an a.s. finite positive random
variable $M$ such that
\[
    \limsup_{t_\to\infty}\int_{0}^{t}|d_{2}(s)||f(s)|ds
    \leq \limsup_{t\to\infty}M\int_{0}^{\infty} (1+s)^{\ba-1/2}\sqrt{\log\log (e+s)} \,ds<\infty.
\]
Hence
\begin{equation}\label{eq:secint}
    \lim_{t\to\infty}\int_{0}^{t}d_2(s)f(s)\,ds = \int_{0}^{\infty}d_2(s)f(s)ds\in (-\infty,\infty) \quad\text{a.s.}
\end{equation}
Since $r_2$ obeys \eqref{eq.r2asy}, we have
\begin{equation} \label{eq.secintto0}
    \lim_{t\to\infty}|r_{2}(t)|\int_{0}^{t}|d_{2}(s)||f(s)|\,ds =0, \quad\text{a.s.}
\end{equation}
In the case when $2b+a\leq0$, we notice from \eqref{eq:orderf} that
for any $\epsilon<1/2$ that $f(t)/t^{-1/2+\epsilon}\to 0$ as
$t\to\infty$ on the a.s. event $\Omega_1$, say. Therefore, by the
continuity of $f$ and this relation, there is an a.s. finite and
positive random variable $K_\epsilon$ such that $|f(t,\omega)|\leq
K_\epsilon(\omega)(1+t)^{-1/2+\epsilon}$ for all $t\geq 0$.
Therefore, by virtue of the continuity of $r_2$, $d_2$ and
\eqref{eq.r2asy} and \eqref{eq:d2asy}, there exists an a.s. finite
and positive random variable $M_\epsilon$ such that, for all $t\geq
0$, we have
\begin{align*}
    |r_{2}(t)|\int_{0}^{t}|d_{2}(s)||f(s,\omega)|ds &\leq M_\epsilon(\omega) (1+t)^{-1-\ba} \int_0^t (1+s)^{\ba-1/2+\epsilon}\,ds\\
    &\leq M_\epsilon(\omega) (1+t)^{-1-\ba} (1+t)^{\ba+1/2+\epsilon}\frac{1}{b/a+1/2+\epsilon},
\end{align*}
for each $\omega \in \Omega_1$, with the last inequality holding
because $b/a-1/2+\epsilon>-1$. Since $\Omega_1$ is an a.s. event, we again have \eqref{eq.secintto0} and so, using this limit and
\eqref{eq:modZ}, we see that $Z(t)\to 0$ as $t\to\infty$ a.s. in the
case that $b/a\not\in\{1,2,...\}$. We can demonstrate that $Z(t)\to
0$ as $t\to\infty$ a.s. in a similar manner when $b/a\in\{1,2,...\}$
by using the asymptotic behaviour of $r_1$, $\tilde{r}_2$,
$\tilde{d}_1$ and $\tilde{d}_2$. Hence the proof of part (i) is
complete.

For the the proof of part (ii), we consider the case $a+b=0$. Recall
that $Y$ can be written in the form
\[
Y(t) = \psi(0)\e^{at} + \sigma e^{at}\int_{0}^{t}\e^{-as}dB(s),
\quad t\geq0.
\]
In this case, we wish to show that $Z$ tends to a non--trivial
limit. Arguing as above, we have that the first integral on the
right hand side of \eqref{eq:Zr1r2} tends to zero as $t\to\infty$
a.s. As to the second term on the right hand side of
\eqref{eq:Zr1r2}, by using a stochastic Fubini theorem, it is seen
that
\begin{multline*}
\int_{0}^{t}d_2(s)f(s)\,ds =\int_{0}^{t}\frac{b}{1+s}d_2(s)\,ds
\int_{-1}^{0}\psi(u)\,du
\\+ \int_{0}^{t} d_2(s)\frac{b}{1+s}\int_{0}^{s}\psi(0)\e^{au}du\,ds
+ \int_0^t H(t,u)\,dB(u),
\end{multline*}
where $H$ is given by \eqref{eq:Haplusb0}. The two Riemann integrals
on the right--hand side of the above equation converge to finite
limits as $t\to\infty$. Moreover as \eqref{eq:secint} holds
therefore the stochastic integral on the right--hand side above
converges almost surely. Recalling from \eqref{eq.r2asy} that
$\lim_{t\to\infty} r_2(t)=e^{-a}|a|^{-b/a-1}$ in the case when
$a+b=0$, and by applying Lemma~\ref{lm:Lvar}, we have that
\begin{multline*}
\lim_{t\to\infty} Z(t)= \frac{e^{-a}}{|a|^{b/a+1}}
\int_{0}^{\infty}\frac{b}{1+s}d_2(s)\,ds \int_{-1}^{0}\psi(u)\,du
\\+ \frac{e^{-a}}{|a|^{b/a+1}}\int_{0}^{\infty} d_2(s)\frac{b}{1+s}\int_{0}^{s}\psi(0)\e^{au}du\,ds
+ \frac{e^{-a}}{|a|^{b/a+1}}\int_0^\infty H_\infty(u)\,dB(u),
\quad\text{a.s.},
\end{multline*}
where $H_\infty$ is given by \eqref{eq:Hinftyaplusb0}. We call the
limit on the righthand side $L$. Therefore $X(t)-U(t)\to L$ as
$t\to\infty$ a.s. Clearly $L$ is an
$\mathcal{F}^B(\infty)$--measurable normal random variable. In order
to see that $L$ is nontrivial,
we may use It\^o's isometry to show that its mean and variance are
given by the formulae in the statement of part (ii) of the theorem. Since 
\[
\lim_{t\to\infty} \frac{1}{1+t}\int_0^t U(s)\,ds = 0, \quad \text{a.s.}
\]
The proofs of \eqref{eq.ave0} and \eqref{eq.aveL} are simple consequences of the fact that $X(t)-U(t)$ tends to the finite limits $0$ and $L$ as $t\to\infty$
a.s. in case (i) and (ii) respectively.

\subsection{Proof of Theorem~\ref{thm:XUrate}}
    Let $Y$ and $Z$ be as defined in the proof of Theorem~\ref{thm.averagegaussian}.
    To attain a bound on the rate of $X-U$ tending to zero, the integral terms in \eqref{eq:modZ} need to be analysed more carefully.
From \eqref{eq:orderf}, and by using the continuity of $f$, it
follows for every $\omega$ in an almost sure event $\Omega_1$ that
there exists an a.s. finite and positive random variable
$K=K(\omega)>0$ such that
such that
\[
    |f(t,\omega)| \leq K(\omega)(1+t)^{-1/2}\sqrt{\log\log (t+e)}, \quad t\geq 0.
\]
    For the first integral in \eqref{eq:modZ}, we start by  using l'H\^opital's rule to show that
    \[
        \lim_{t\to\infty}\frac{\int_{0}^{t}\e^{-as}(1+s)^{-\ba-1/2}\sqrt{\log\log(e+s)}\,ds}
        {\e^{-at}(1+t)^{-\ba-1/2}\sqrt{\log\log(e+t)}} \in(0,\infty).
    \]
    Therefore, there is $K_3>0$ such that
    \[
    \int_{0}^{t}\e^{-as}(1+s)^{-\ba-1/2}\sqrt{\log\log(e+s)}\,ds\leq K_3
        \e^{-at}(1+t)^{-\ba-1/2}\sqrt{\log\log(e+t)},
    \]
   for all $t\geq 0$. Now, by using 
    \eqref{eq.r1asy} 
    and \eqref{eq:d1asy} 
    and the continuity of $r_1$ and $d_1$, we have
    that there exist $K_1>0$ and $K_2>0$ such that
    \begin{equation*} 
    |r_{1}(t)| \leq K_1 \e^{at} (1+ t)^{\ba}, \quad t\geq 0;  \quad
      |d_{1}(s)| \leq K_2 \e^{-as} (1+s)^{-\ba},  \quad s\geq 0.
\end{equation*}
    Therefore for all $\omega\in \Omega_1$ and $t\geq 0$ we have
    \begin{align*}
    \lefteqn{|r_{1}(t)|\int_{0}^{t}|d_{1}(s)||f(s,\omega)|ds }\\
    &\leq K_4(\omega)  (1+t)^{-1/2} \sqrt{\log\log (t+e)},
    \end{align*}
    where $K_4(\omega)=K_1 K_2K(\omega) K_3$.
    Hence
    \begin{equation}\label{eq:Z1}
        \limsup_{t\to\infty} \frac{|r_{1}(t)|\int_{0}^{t}|d_{1}(s)||f(s)|ds}{(1+t)^{-1/2}\sqrt{\log\log(1+t)}}\in[0,\infty), \quad \text{a.s.}
    \end{equation}

    For the second integral in \eqref{eq:modZ}, we showed in the proof of Theorem~\ref{thm.averagegaussian} that $\limsup_{t\to\infty}\int_{0}^{t}|d_2(s)f(s)|\,ds<+\infty$ a.s. in the case when $2b+a>0$. Hence
    \begin{equation}\label{eq:Z2}
        \limsup_{t\to\infty}\frac{|r_{2}(t)|\int_{0}^{t}|d_{2}(s)||f(s)|ds}{(1+t)^{-1-\ba}}\in[0,\infty).
    \end{equation}
    Moreover in this parameter regime $-1/2<-1-b/a<0$, and so comparing the decay rates in \eqref{eq:Z1} and \eqref{eq:Z2} gives (i).

    When $2b+a<0$, we may use l'H\^opital's rule to get
    \[    \lim_{t\to\infty}\frac{\int_{0}^{t}(1+s)^{\ba-1/2}\sqrt{\log\log(e+s)}\,ds}{(1+t)^{\ba+1/2}\sqrt{\log\log(e+t)}}\in(0,\infty).
    \]
    Hence there exists $K_7>0$ such that
    \[
  \int_{0}^{t}(1+s)^{\ba-1/2}\sqrt{\log\log(e+s)}\,ds\leq K_7(1+t)^{\ba+1/2}\sqrt{\log\log(e+t)}, \quad t\geq 0.
    \]
    Since $r_2$ and $d_2$ obey \eqref{eq.r2asy} and \eqref{eq:d2asy}, we have that there exist $K_5>0$ and $K_6>0$
    such that
     \begin{equation*} 
    |r_{2}(t)|\leq K_5 (1+t)^{-1-\ba}, \quad t\geq 0; \quad
      |d_{2}(s)| \leq K_6 (1+s)^{\ba},\quad s\geq 0.
\end{equation*}
    Therefore for all $\omega\in \Omega_1$ and $t\geq 0$ we have
    \begin{align*}
    \lefteqn{|r_{2}(t)|\int_{0}^{t}|d_{2}(s)||f(s,\omega)|\,ds}\\
    &\leq   K_5 K_6 K(\omega) K_7 (1+t)^{-1/2} \sqrt{\log\log(e+t)},
    \end{align*}
    and so
    \begin{equation}\label{eq:Z3}
        \limsup_{t\to\infty} \frac{|r_{2}(t)|\int_{0}^{t}|d_{2}(s)||f(s)|ds}{(1+t)^{-1/2}\sqrt{\log\log(1+t)}}\in[0,\infty).
    \end{equation}
    Applying \eqref{eq:Z2} and \eqref{eq:Z3} in \eqref{eq:modZ} proves (ii).

In the case $2b + a = 0$, we have the estimate
\[
\lim_{t\to\infty} \frac{\int_0^t (1 + s)^{-1}\sqrt{\log \log(e + s)} \,ds}{\log t \sqrt{\log \log t}}
= 1.
\]
Now following the same procedure as for the proof of part (ii) gives the result.

\section{Proof of Theorem \ref{thm:Bes}} \label{sect:prooflilbessel}
We begin this section with the statement and proof of some
preparatory lemmata. Firstly, we give a discrete version of the Law
of the Iterated Logarithm. The following result is stated as
Theorem~2 of \cite{tomkins:1971} or Exercise~3 in \cite[pp383,
Section~10.2]{chowt}
\begin{lemma}\label{lm:norlil}
    Let $\{X_n\}_{n\in\N}$ be a sequence of independent Gaussian random variables where $X_n$ has mean zero and
    variance $\sigma_n^2$. If $s^2_n=\sum_{i=1}^{n}\sigma^2_i \to\infty$ as $n\to\infty$ and $\sigma_n=o(s_n)$
    as $n\to\infty$, then
    \[
        \limsup_{n\to\infty}\frac{\sum_{j=1}^n X_j}{\sqrt{2s^2_n\log\log s^2_n}}
        = -\liminf_{n\to\infty}\frac{\sum_{j=1}^n X_j}{\sqrt{2s^2_n\log\log s^2_n}}=1, \quad \text{a.s.}
    \]
\end{lemma}
While the above result is sufficient for the analysis of this
article, as Tomkins \cite{tomkins:1971} observes these sufficient
conditions may be sharpened. For instance,  Hartman \cite{hart:1941}
requires only $\limsup_{n\to\infty} \sigma_n/s_n<1$ as opposed to
$\sigma_n/s_n\to0$ as $n\to\infty$ in order to prove a discrete Law
of the Iterated Logarithm in the Gaussian case.
\begin{lemma}\label{lm:intsub}
Let $b<0$. Then the following limits hold:
\begin{equation}\label{eq:quadM2}
    \lim_{t\to\infty}\frac{\pi\sqrt{|b|}\int_{0}^{t} (1+s)^{1/2}\sin^2\left(2\sqrt{|b|(s+1)} -\frac{3}{4}\pi\right)ds}
    {\frac{\pi}{3}|b|^{1/2}(1+t)^{3/2}} =1,
\end{equation}
and
\begin{equation}\label{eq:quadM3}
    \lim_{t\to\infty}\frac{\pi\sqrt{|b|}\int_{0}^{t} (1+s)^{1/2}\cos^2\left(2\sqrt{|b|(s+1)} -\frac{3}{4}\pi\right)ds}
    {\frac{\pi}{3}|b|^{1/2}(1+t)^{3/2}} =1.
\end{equation}
\end{lemma}
While this lemma amounts to little more than integration by parts,
it serves as an asymptotic estimate of the rate of growth of the
quadratic variation of stochastic integrals to be considered later.
\begin{proof}[Proof of Lemma~\ref{lm:intsub}]
Consider first the limit \eqref{eq:quadM2}. Making the substitution
$w=2\sqrt{|b|(s+1)} -\frac{3}{4}\pi$ in the integral, we get
\begin{align*}
    \pi\sqrt{|b|}\int_{0}^{t} & (1+s)^{1/2}\sin^2\left(2\sqrt{|b|(s+1)} -\frac{3}{4}\pi\right)ds \\
    &= \frac{\pi}{4|b|} \int_{2\sqrt{|b|}-3\pi/4}^{2\sqrt{|b|(1+t)}-3\pi/4} \left(w+\frac{3\pi}{4}\right)^2 \sin^2(w)\, dw.
%
%
\end{align*}
Since $\int_0^x (w+3\pi/4)^2 \sin^2(w)\, dw$  can be computed
explicitly for $x\geq 0$, and this leads to
\[
\lim_{x\to\infty} \frac{1}{x^3}\int_0^x (w+\frac{3\pi}{4})^2
\sin^2(w)\, dw=\frac{1}{6},
\]
 \eqref{eq:quadM2} holds. Similar calculations confirm the limit \eqref{eq:quadM3}
\end{proof}
We next introduce functions which correspond to the leading order
asymptotic behaviour of $r_7$, $r_8$, $d_7$ and $d_8$. Define the
functions, for $t\geq0$
\begin{subequations}  \label{eq:g1g2g3g4}
\begin{align}
\label{eq:g1g2g3g4g1}
    g_1(t) &= \frac{1}{\sqrt{\pi}} |b|^{-1/4}(1+t)^{-1/4} \cos(2\sqrt{|b|(1+t)} - \pi/4), \\
    \label{eq:g1g2g3g4g2}
    g_2(t) &= \sqrt{\pi} |b|^{1/4} (1+t)^{1/4} \sin\left(2\sqrt{|b|(1+t)}-\frac{3}{4}\pi\right), \\
    \label{eq:g1g2g3g4g3}
    g_3(t) &= \frac{1}{\sqrt{\pi}} |b|^{-1/4}(1+t)^{-1/4} \sin(2\sqrt{|b|(1+t)} - \pi/4), \\
    \label{eq:g1g2g3g4g4}
    g_4(t) &= \sqrt{\pi} |b|^{1/4} (1+t)^{1/4} \cos\left(2\sqrt{|b|(1+t)}-\frac{3}{4}\pi\right).
\end{align}
\end{subequations}
We aim to show that these leading order terms describe a continuous
time process which obeys the law of the iterated logarithm along
many carefully designed sequences. These sequences will later be
used to extrapolate the asymptotic behaviour of the continuous time
process to the positive real line.
\begin{lemma}\label{lm:seqbes}
Fix $\eta\in[0,\pi/2)$. Define the sequence $\{t_n:n\in\Z^+\}$ such
that
\[
    t_0=0, \quad t_n = |b|^{-1}(n\pi+\pi/8+ \lceil \sqrt{|b|}/\pi -1/8 \rceil\pi +\eta/2)^2 -1, \quad n\geq1.
\]
If $g_1$--$g_4$ are defined by \eqref{eq:g1g2g3g4}, then
\begin{align}\label{eq.limsupgn}
    \limsup_{n\to\infty}\frac{g_1(t_n)\int_{0}^{t_n}g_2(s)dB(s)+g_3(t_n)\int_{0}^{t_n}g_4(s)dB(s)}{\sqrt{2 t_n\log\log t_n}}
    &= \frac{1}{\sqrt{3}}, \quad \text{a.s.}, \\
    \label{eq.liminfgn}
    \liminf_{n\to\infty}\frac{g_1(t_n)\int_{0}^{t_n}g_2(s)dB(s)+g_3(t_n)\int_{0}^{t_n}g_4(s)dB(s)}{\sqrt{2 t_n\log\log t_n}}
    &= -\frac{1}{\sqrt{3}}, \quad \text{a.s.}
\end{align}
\end{lemma}
\begin{proof}[Proof of Lemma~\ref{lm:seqbes}]
We start by noticing that $t_n\geq 0$ for all $n\geq 0$ and
therefore $(t_n)_{n\geq 1}$ is a increasing sequence.
Note also that
\begin{equation} \label{eq.rttnn}
2\sqrt{|b|(t_n+1)}=2n\pi + \frac{\pi}{4}+\eta+ 2\pi L_b,
\end{equation}
where
\begin{equation}\label{def.Lb}
L_b:=\lceil \sqrt{|b|}/\pi -1/8 \rceil \geq \sqrt{|b|}/\pi -1/8 \geq
-1/8.
\end{equation}
Therefore, as $L_b\in\mathbb{Z}$, we see that we must have $L_b$ a
non--negative integer. For all $n\in\Z$ let $\beta=\beta_\eta$ be
the number such that $\cos(2n\pi+\eta)=\beta\in(0,1]$ and it is to
be noted that $\beta$ does not depend upon $n$. Then
\eqref{eq.rttnn} implies
\begin{equation}\label{eq:sigcs}
    \cos(2\sqrt{|b|(1+t_n)}-\pi/4)=\beta, \quad \text{and hence}\quad \sin(2\sqrt{|b|(1+t_n)}-\pi/4)=\sqrt{1-\beta^2}.
\end{equation}
Our plan now is to establish that
\[
\int_{0}^{t_n}[g_2(s)g_1(t_n)+g_4(s)g_3(t_n)]dB(s)
\]
gives rise to a discrete--time Gaussian martingale, to which
Lemma~\ref{lm:norlil} can be applied. To do this, we write
\begin{align}\label{eq:sumdrn2}
    &\frac{\int_{0}^{t_n}[g_2(s)g_1(t_n)+g_4(s)g_3(t_n)]dB(s)}{\sqrt{2t_n\log\log t_n}} \\
    &= (1+t_n)^{-1/4} \biggl(
    \frac{\int_{0}^{t_n}(s+1)^{1/4} \sin\left(2\sqrt{|b|(s+1)}-\frac{3}{4}\pi\right) \beta dB(s)}{\sqrt{2t_n\log\log t_n}} \notag\\
    &\qquad +\frac{\int_{0}^{t_n}(s+1)^{1/4} \cos\left(2\sqrt{|b|(s+1)}-\frac{3}{4}\pi\right) \sqrt{1-\beta^2}dB(s)}{\sqrt{2t_n\log\log t_n}}\biggr) \notag \\
    &= (1+t_n)^{-1/4}
    \frac{\int_{0}^{t_n}(s+1)^{1/4}\sin\left(2\sqrt{|b|(s+1)}-\frac{3}{4}\pi\ + \eta \right)dB(s)}{\sqrt{2t_n\log\log t_n}} \notag,
\end{align}
where we have used \eqref{eq:sigcs} at the last step. As the last
stochastic integral on the right  hand side does not depend upon $n$
in the integrand, we can decompose the integral and apply
Lemma~\ref{lm:norlil} to it. We therefore define for $n\geq 1$
\[
    S_n := \sum_{j=1}^{n}Y_j, \text{ where }
    Y_j=\int_{t_{j-1}}^{t_j}(s+1)^{1/4}\sin\left(2\sqrt{|b|(s+1)}-\frac{3}{4}\pi\ + \eta \right)\,dB(s).
\]
Then $Y_j$ is a Gaussian distributed random variable with mean zero
and variance
\[
    \sigma_j^2:=\int_{t_{j-1}}^{t_j}(s+1)^{1/2}\sin^2\left(2\sqrt{|b|(s+1)}-\frac{3}{4}\pi\ + \eta \right)\,ds
\]
and $S_n$ is a Gaussian distributed random variable with mean zero
and variance
\[
    s_n^2=\sum_{j=0}^{n}\sigma_j^2 = \int_{0}^{t_n}(s+1)^{1/2}\sin^2\left(2\sqrt{|b|(s+1)}-\frac{3}{4}\pi\ + \eta \right)\,ds.
\]
We wish to ascertain the rate of growth of both $\sigma_j^2$ and
$s_n^2$. Define
\[
    M_\eta(t) = \int_{0}^{t}(1+s)^{1/4} \sin(2\sqrt{|b|(1+s)}-3\pi/4 + \eta) dB(s), \quad t\geq 0.
\]
Then $M_\eta$ is a continuous martingale and its quadratic variation
is given by
\[
    \langle M_\eta\rangle(t)
    = \int_{0}^{t} (1+s)^{1/2} \sin^2(2\sqrt{|b|(1+s)}-3\pi/4 + \eta)\,ds, \quad t\geq 0.
\]
Therefore we have that
\[
\langle M_\eta\rangle(t) =
\frac{1}{4|b|^{3/2}}\int_{2\sqrt{|b|}-\frac{3\pi}{4}+\eta}^{2\sqrt{|b|(1+t)}-\frac{3\pi}{4}+\eta}
\left(w+\frac{3\pi}{4}-\eta\right)^2\sin^2(w)\,dw,  \quad t\geq 0.
\]
An explicit calculation following exactly the model of
Lemma~\ref{lm:intsub} shows that
\[
    \langle M_\eta\rangle(t) \sim \frac{1}{3}t^{3/2}, \quad \text{as $t\to\infty$}.
\]
We remark that the asymptotic behaviour of the quadratic variation
is independent of $\eta$. Thus, since $t_n\sim n^2\pi^2/|b|$ as
$n\to\infty$, we have that
\[
    s_n^2 = \langle M_\eta\rangle(t_n) \sim \frac{1}{3}t_n^{3/2} \sim \frac{n^3\pi^3}{3 |b|^{3/2}} \quad \text{ as $n\to\infty$}.
\]
For $n\geq1$, by \eqref{eq.rttnn} we have
\begin{align*}
    \sigma^2_n &= \langle M_\eta\rangle(t_n) - \langle M_\eta\rangle(t_{n-1})\\
    &=\frac{1}{4|b|^{3/2}}\int_{2\sqrt{|b|(1+t_{n-1})}-\frac{3\pi}{4}+\eta}^{2\sqrt{|b|(1+t_n)}-\frac{3\pi}{4}+\eta}
\left(w+\frac{3\pi}{4}-\eta\right)^2\sin^2(w)\,dw\\
&\leq \frac{1}{4|b|^{3/2}}\int_{2(n-1)\pi -\pi/2 +2\eta +2\pi
L_b}^{2n\pi -\pi/2 +2\eta +2\pi L_b}
\left(w+\frac{3\pi}{4}-\eta\right)^2\,dw\\
&=\frac{1}{12|b|^{3/2}}\left( (2n\pi +\pi/4 +\eta +2\pi
L_b)^3-(2(n-1)\pi +\pi/4 +\eta +2\pi L_b)^3\right).
\end{align*}
Therefore we have that $\sigma_n^2=O(n^2)=O(t_n)$ as $n\to\infty$.
Hence $\lim_{n\to\infty} \sigma_n/s_n=0$. Thus all the conditions of
Lemma~\ref{lm:norlil} are satisfied and so the discrete Law of the
Iterated Logarithm may be applied to $S_n$ (or equivalently, to
$M_{\eta}(t_n)$). Therefore by \eqref{eq:sumdrn2}, and by using the
fact that
\begin{align*}
    \lim_{n\to\infty}\frac{t_n^{-1/4}}{\sqrt{2t_n\log\log t_n}} \sqrt{2\langle M_\eta\rangle(t_n) \log\log\langle M_\eta\rangle(t_n)}
    =\frac{1}{\sqrt{3}},
\end{align*}
gives the limit superior in \eqref{eq.limsupgn}. The limit inferior
in \eqref{eq.liminfgn} may be obtained via a symmetry argument.
\end{proof}
\begin{remark}\label{rk:seqbes}
    Although Lemma~\ref{lm:seqbes} fixes $\eta$ in the interval $[0,\pi/2)$, it is apparent from the proof of this
    lemma that one is free to choose $\eta$ in any of the non--overlapping intervals $[\pi/2,\pi)$, $[\pi,3\pi/2)$
    or $[3\pi/2,2\pi)$. The only amendments in the proof that would result from choosing $\eta$ in these other intervals
    would be changes in the signs of the cosine and sine terms in \eqref{eq:sigcs}.
\end{remark}
\begin{lemma}\label{lm:seqden}
Fix $k\in\Z^+$. Define the sequence $\{t_n^{(k)}:n\in\Z^+\}$ by
$t_0^{(k)}=0$ and
\begin{equation}\label{eq:alltn}
    t_j^{(k)} = \frac{1}{|b|}\left(N_j \pi +\left\lceil \frac{\sqrt{|b|}}{\pi} -\frac{1}{8} \right\rceil\pi + \frac{\eta^{(j,k)}}{2}
    + \frac{\pi}{8} \right)^2-1, \quad j\geq1
\end{equation}
where
\[
    N_j = \left\lceil\frac{j}{2^{2+k} } \right\rceil -1, \quad
    i_j = j - 2^{2+k}N_j-1, \quad \eta^{(j,k)} = \frac{i_j}{2^k}\frac{\pi}{2},
\]
so that $i_j\in\{0,1,\dots,2^2 2^{k}-1\}$.
Then
\begin{subequations}\label{eq:lilgg}
\begin{align*}
    \limsup_{n\to\infty}\frac{g_1(t_n^{(k)})\int_{0}^{t_n^{(k)}}g_2(s)dB(s)+g_3(t_n^{(k)})\int_{0}^{t_n^{(k)}}g_4(s)dB(s)}
    {\sqrt{2 t_n^{(k)}\log\log t_n^{(k)}}}
    &= \frac{1}{\sqrt{3}}, \quad a.s., \\
    \liminf_{n\to\infty}\frac{g_1(t_n^{(k)})\int_{0}^{t_n^{(k)}}g_2(s)dB(s)+g_3(t_n^{(k)})\int_{0}^{t_n^{(k)}}g_4(s)dB(s)}
    {\sqrt{2 t_n^{(k)}\log\log t_n^{(k)}}}
    &= -\frac{1}{\sqrt{3}}, \quad a.s.
\end{align*}
\end{subequations}
where $g_1$, $g_2$, $g_3$ and $g_4$ are as defined in
\eqref{eq:g1g2g3g4}.
Also,
\begin{align} \label{eq.Njtjasy}
    &N_j \sim \frac{j}{2^{2}2^{k}}, \quad
    t_j^{(k)}\sim \frac{1}{|b|}N_j^2 \pi^2 \sim \frac{1}{|b|}\frac{1}{2^4 2^{2k}} j^2\pi^2, \quad \text{ as $j\to\infty$},  \\
    &\Delta t_j^{(k)}:= t_{j+1}^{(k)} - t_j^{(k)} \sim \frac{1}{|b|}N_j\frac{1}{2^k}\frac{\pi^2}{2}
    \sim \frac{1}{|b|}\frac{j}{2^{2k}}\frac{\pi^2}{2^3} \quad \text{ as $j\to\infty$}. \label{eq:jall}
\end{align}
\end{lemma}
\begin{proof}[Proof of Lemma~\ref{lm:seqden}]
Define $\beta_{j,k}^{(i)}:=\cos(\eta_{i}^{(j,k)})$, where
\[
    \eta_{i}^{(j,k)} := (i-1)\frac{\pi}{2} + \frac{j}{2^k}\frac{\pi}{2}, \quad i\in\{1,2,3,4\}, \quad j\in\{0,1,...2^{k}-1\}.
\]
Now define the following $4\times 2^k$ sequences. For each $j\in
\{0,1,...2^{k}-1\}$, we define for $n\geq 0$
\begin{align*}
    \tau_{n}^{(j,k)} &= |b|^{-1}(n\pi+\pi/8 
    +\eta_1^{(j,k)}/2)^2 -1, \\
    T_{n}^{(j,k)} &= |b|^{-1}(n\pi+\pi/8 
    + \eta_2^{(j,k)}/2)^2 -1, \\
    \theta_{n}^{(j,k)} &= |b|^{-1}(n\pi+\pi/8 
    +\eta_3^{(j,k)}/2)^2 -1, \\
    \Theta_{n}^{(j,k)} &= |b|^{-1}(n\pi+\pi/8 
    +\eta_4^{(j,k)}/2)^2 -1.
\end{align*}
Notice that each of these sequences is  increasing.
Then the sequence $\{\tau_{n}^{(j,k)}\}_{n\geq0}$ may be expressed
in terms of $\beta_{j,k}^{(1)}$ (which is independent of $n$)
according to
\[
    \beta_{j,k}^{(1)} = \cos(\eta_{i}^{(j,k)}) = \cos\left(2\sqrt{|b|(\tau_{n}^{(j,k)}+1)} -\pi/4\right).
\]
Similarly $T_{n}^{(j,k)},\theta_{n}^{(j,k)},\Theta_{n}^{(j,k)}$ may
be expressed in terms of
$\beta_{j,k}^{(2)},\beta_{j,k}^{(3)},\beta_{j,k}^{(4)}$
respectively.

Define
\[
    \bar{Y}(t) := \frac{g_1(t)\int_{0}^{t}g_2(s)dB(s)+g_3(t)\int_{0}^{t}g_4(s)dB(s)}{\sqrt{2 t\log\log t}}, \quad t \geq e^e.
\]
Then, from Lemma~\ref{lm:seqbes}, for each
$j\in\{0,1,...,2^{k}-1\}$,
\[
    \limsup_{n\to\infty}\bar{Y}(\tau_{n}^{(j,k)}) = -\liminf_{n\to\infty}\bar{Y}(\tau_{n}^{(j,k)}) = \frac{1}{\sqrt{3}},
\]
on an event of probability one, $\Omega_1^{(j,k)}$. Using
Lemma~\ref{lm:seqbes} in conjunction with Remark~\ref{rk:seqbes}
gives
\begin{align*}
    \limsup_{n\to\infty}\bar{Y}(T_{n}^{(j,k)}) &= -\liminf_{n\to\infty}\bar{Y}(T_{n}^{(j,k)}) = \frac{1}{\sqrt{3}}, \\
    \limsup_{n\to\infty}\bar{Y}(\theta_{n}^{(j,k)}) &= -\liminf_{n\to\infty}\bar{Y}(\theta_{n}^{(j,k)}) = \frac{1}{\sqrt{3}}, \\
    \limsup_{n\to\infty}\bar{Y}(\Theta_{n}^{(j,k)}) &= -\liminf_{n\to\infty}\bar{Y}(\Theta_{n}^{(j,k)}) = \frac{1}{\sqrt{3}},
\end{align*}
on almost sure events, $\Omega_2^{(j,k)}$,$\Omega_3^{(j,k)}$ and
$\Omega_4^{(j,k)}$ respectively. Now,
\begin{align*}
&\tau_{n}^{(0,k)}<\tau_n^{(1,k)}<...<\tau_n^{(2^k-1,k)}<T_n^{(0,k)}<...<T_n^{(2^k-1,k)}<\theta_n^{(0,k)}<...<\theta_n^{(2^k-2,k)} \\
&<\theta_n^{(2^k-1,k)}<\Theta_n^{(0,k)}<...<\Theta_n^{(2^k-1,k)}
\end{align*}
and $\Theta_n^{(2^k-1,k)}<\tau_n^{(0,k)}$. Observe that the sequence
$\{t_n^{(k)}\}_{n\geq0}$, defined in the statement of this Lemma,
obeys, for $j\geq1$
\[
t_j^{(k)}=
\begin{cases}
    \tau_{N_j+\lceil \sqrt{|b|}/\pi -1/8 \rceil}^{(i_j,k)}, &\quad i_j\in\{0,...,2^k-1\}, \\
    T_{N_j+\lceil \sqrt{|b|}/\pi -1/8 \rceil}^{(i_j-2^k,k)}, &\quad i_j\in\{2^k,...,2.2^k-1\}, \\
    \theta_{N_j+\lceil \sqrt{|b|}/\pi -1/8 \rceil}^{(i_j-2.2^k,k)}, &\quad i_j\in\{2.2^k,...,3.2^k-1\}, \\
    \Theta_{N_j+\lceil \sqrt{|b|}/\pi -1/8 \rceil}^{(i_j-3.2^k,k)}, &\quad i_j\in\{3.2^k,...,4.2^k-1\},
\end{cases}
\]
Hence, defining $\Omega_5^{(k)}=
\bigcap_{i=1}^{4}\bigcap_{j=0}^{2^k-1}\Omega_i^{(j,k)}$ and noting
that $\Omega_5^{(k)}$ is an almost sure event, we have that
\[
    \limsup_{n\to\infty}\bar{Y}(t_n^{(k)})= -\liminf_{n\to\infty}\bar{Y}(t_n^{(k)}) = \frac{1}{\sqrt{3}},
\]
on the event $\Omega_5^{(k)}$, which is \eqref{eq:lilgg}.

We turn next to determining the asymptotic behaviour of the
sequences $N_j$, $t_j^{(k)}$, $\Delta t_j^{(k)}$ as $j\to\infty$. We
start with $N_j$. By definition,
we have $    j/(2^2.2^k)-1 \leq N_j < j/(2^2.2^k)$, and thus,
$1/(2^2.2^k)-1/j \leq N_j/j < 1/(2^2.2^k)$. Now letting $j$ tend to
infinity and we have $\lim_{j\to\infty}N_j/j = 1/2^{2+k}$. Moreover
as $\eta^{(j,k)}$ is bounded we have
$\lim_{j\to\infty}\eta^{(j,k)}/j=0$. Then from the definition of the
sequence $\{t_n^{(k)}\}_{n\geq0}$ it follows that
\[
    t_j^{(k)}\sim \frac{1}{|b|}N_j^2 \pi^2 \sim \frac{1}{|b|}\frac{1}{2^4 2^{2k}} j^2\pi^2, \quad \text{as $j\to\infty$}.
\]
In determining the asymptotic behaviour of $\Delta t_j^{(k)}$ we
first consider the asymptotic behaviour of
$\Delta\eta^{(j+1,k)}:=\eta^{(j+1,k)}-\eta^{(j,k)}$ for large $j$.
From the definition of $\eta^{(j,k)}$ it is trivially true that
$\Delta \eta^{(j,k)} =\pi/2^{1+k}$ whenever $N_{j+1}=N_j$. Moreover
the only values of $j$ for which $N_{j+1}\not=N_j$ are values of the
type $j=m.2^{2+k}$ for $m\in \{1,2,...\}$. So, if $j\not=m.2^{2+k}$
and $j\geq1$, we get
\begin{align*}
    \Delta t_j^{(k)} &= \frac{1}{|b|}\left(N_j\pi + L_b \pi + \frac{\pi}{8} + \frac{\eta^{(j,k)}}{2} + \frac{\Delta \eta^{(j,k)}}{2} \right)^2 - 1 \\
    &\quad- \frac{1}{|b|}\left(N_j\pi + L_b \pi + \frac{\pi}{8}
    + \frac{\eta^{(j,k)}}{2} \right)^2 + 1 \\
    &= \frac{2}{|b|}\left( N_j\pi + L_b \pi + \frac{\pi}{8}
    + \frac{\eta^{(j,k)}}{2}\right)\frac{\Delta \eta^{(j,k)}}{2} +  \frac{1}{|b|}\frac{(\Delta \eta^{(j,k)})^2}{4}.
\end{align*}
Thus,
\begin{equation}\label{eq:jmid}
    \Delta t_j^{(k)} \sim \frac{2}{|b|}N_j\pi \frac{\Delta \eta^{(j,k)}}{2}
    = \frac{N_j \pi^2}{|b|2.2^k} \sim \frac{j \pi^2}{|b|2^3.2^{2k}}, \quad \text{as $j\to\infty$}.
\end{equation}
If $j=m.2^{2+k}$ for $m\in\{1,2,...\}$, we have $N_{j+1}=N_{j}+1=m$
(as we are interested in the asymptotic behaviour of $\eta^{(j,k)}$
for large $j$ we may exclude $m=0$ from our analysis). In this case,
\[
    \eta^{(j+1,k)} = \frac{(j+1 - 2^{2+k}N_{j+1}-1)}{2^k}\frac{\pi}{2} = \frac{(m.2^{2+k}+1 - m.2^{2+k}-1)}{2^k}\frac{\pi}{2}=0
\]
while
\[
\eta^{(j,k)} = \frac{(j - 2^{2+k}N_{j}-1)}{2^k}\frac{\pi}{2} =
\frac{(m.2^{2+k} -
(m-1)2^{2+k}-1)}{2^k}\frac{\pi}{2}=\frac{(2^{2+k}-1)}{2^k}\frac{\pi}{2}.
\]
This gives
\begin{align*}
    \Delta t_j^{(k)} &= \frac{1}{|b|}\left(N_{j+1}\pi + L_b\pi
    + \frac{\pi}{8} + \frac{\eta^{(j+1,k)}}{2} \right)^2 - 1 \\
    &\quad- \frac{1}{|b|}\left(N_j\pi + L_b\pi + \frac{\pi}{8}
    + \frac{\eta^{(j,k)}}{2} \right)^2 + 1 \\
    &= \frac{1}{|b|}\left(N_{j}\pi + L_b\pi + \frac{\pi}{8}
    + \pi \right)^2 \\
    &\quad- \frac{1}{|b|}\left(N_j\pi + L_b\pi + \frac{\pi}{8}
    + \frac{(2^{2+k}-1)}{2.2^k}\frac{\pi}{2} \right)^2  \\
    &= \frac{2}{|b|}\left( N_j\pi + L_b \pi + \frac{\pi}{8}\right) \frac{\pi}{2^2.2^k} +  \frac{\pi^2}{|b|}\frac{2^{3+2k}-1}{2^{4+2k}}.
\end{align*}
Thus, as $j=m.2^{2+k}$,
\begin{equation}\label{eq:jend}
    \Delta t_{m.2^{2+k}}^{(k)} \sim \frac{2}{|b|} m \frac{\pi^2}{2^2.2^k}, \quad \text{as $m\to\infty$}.
\end{equation}
Therefore \eqref{eq:jend} together with \eqref{eq:jmid} yields
\eqref{eq:jall}.
\end{proof}
\begin{lemma}\label{lm:uplil}
    Let $g_1$,$g_2$,$g_3$ and $g_4$ be as defined in \eqref{eq:g1g2g3g4}. Let
    \[
        Y(t) := g_1(t)\int_{0}^{t}g_2(s)dB(s)+g_3(t)\int_{0}^{t}g_4(s)dB(s), \quad t\geq0.
    \]
    Then,
    \[
        \limsup_{t\to\infty}\frac{Y(t)}{\sqrt{2t\log\log t}} = \frac{1}{\sqrt{3}},  \quad
        \liminf_{t\to\infty}\frac{Y(t)}{\sqrt{2t\log\log t}} = -\frac{1}{\sqrt{3}}, \quad \text{a.s.}
    \]
\end{lemma}
\begin{proof}[Proof of Lemma~\ref{lm:uplil}]
A lower bound on the limit superior may easily be obtained from
Lemma~\ref{lm:seqden}. We have
\begin{equation}\label{eq:Ytlow}
    \limsup_{t\to\infty}\frac{Y(t)}{\sqrt{2t\log\log t}} \geq \limsup_{n\to\infty}\frac{Y(t_n)}{\sqrt{2t_n\log\log t_n}}
    = \frac{1}{\sqrt{3}}, \quad \text{a.s.},
\end{equation}
where the sequence $\{t_n\}_{n\in\Z^+}$ is as defined by
\eqref{eq:alltn} (for ease of notation we omit the $k$-dependence).
We now turn our attention to obtaining an upper bound.

Define $\tilde{Y}(t):=\sqrt{\pi}|b|^{1/4}(1+t)^{1/4}Y(t)$ for $t\geq
0$. Then from Lemma~\ref{eq:lilgg} we have
\begin{equation}\label{eq:YtYn}
    \limsup_{n\to\infty}\frac{|\tilde{Y}(t_n)|}{\sqrt{2}t_n^{3/4}\sqrt{\log\log t_n}}
    = \limsup_{n\to\infty}\frac{\sqrt{\pi}|b|^{1/4} |Y(t_n)|}{\sqrt{2t_n\log\log t_n}} = \frac{\sqrt{\pi}|b|^{1/4}}{\sqrt{3}}, \quad \text{a.s.},
\end{equation}
where the limit superior is taken through the sequence
$\{t_n\}_{n\in\Z^+}$ defined in \eqref{eq:alltn} (again for ease of
notation we omit the $k$-dependence). Now, for $t_n\leq t\leq
t_{n+1}$,
\begin{align*}
    \frac{\tilde{Y}(t)}{\sqrt{2}t^{3/4}\sqrt{\log\log t}}
    &= \frac{\tilde{Y}(t)-\tilde{Y}(t_n)}{\sqrt{2}t_n^{3/4}\sqrt{\log\log t_n}}
    \frac{\sqrt{2}t_n^{3/4}\sqrt{\log\log t_n}}{\sqrt{2}t^{3/4}\sqrt{\log\log t}} \\
    &\quad + \frac{\tilde{Y}(t_n)}{\sqrt{2}t_n^{3/4}\sqrt{\log\log t_n}}
    \frac{\sqrt{2}t_n^{3/4}\sqrt{\log\log t_n}}{\sqrt{2}t^{3/4}\sqrt{\log\log t}},
\end{align*}
and so
\begin{align}\label{eq:Ytlil}
    \frac{\tilde{Y}(t)}{\sqrt{2}t^{\frac{3}{4}}\sqrt{\log\log t}}
    \leq \frac{\sup_{t_n\leq t\leq t_{n+1}}|\tilde{Y}(t)-\tilde{Y}(t_n)|}{\sqrt{2}t_n^{\frac{3}{4}}\sqrt{\log\log t_n}}
    + \frac{|\tilde{Y}(t_n)|}{\sqrt{2}t_n^{\frac{3}{4}}\sqrt{\log\log t_n}}, \, t\in [t_n,t_{n+1}].
\end{align}
We firstly examine the asymptotic behaviour of $\sup_{t_n\leq t\leq
t_{n+1}}|\tilde{Y}(t)-\tilde{Y}(t_n)|$.
Define
\begin{align*}
    \tilde{Y}_1(t) &= \sqrt{\pi}|b|^{1/4}(1+t)^{1/4}g_1(t)\int_{0}^{t}g_2(s)dB(s), \quad t\geq 0,\\
    \tilde{Y}_2(t) &= \sqrt{\pi}|b|^{1/4}(1+t)^{1/4}g_3(t)\int_{0}^{t}g_4(s)dB(s), \quad t\geq 0.
\end{align*}
Then $\tilde{Y}(t) = \tilde{Y}_1(t) + \tilde{Y}_2(t)$ and for $t\in
[t_n,t_{n+1}]$ we have
\begin{align*}
    \lefteqn{|\tilde{Y}_1(t) - \tilde{Y}_1(t_n)|}\\
    & \leq |\cos(2\sqrt{|b|(1+t)}-\pi/4)| \left|\int_{0}^{t}g_2(s)dB(s)-\int_{0}^{t_n}g_2(s)dB(s)\right| \\
    &\quad+ |\cos(2\sqrt{|b|(1+t)}-\pi/4)-\cos(2\sqrt{|b|(1+t_n)}-\pi/4)| \left|\int_{0}^{t_n}g_2(s)dB(s)\right| \\
    &\leq \left|\int_{t_n}^{t}g_2(s)dB(s)\right| + |2\sqrt{|b|(1+t)} -2\sqrt{|b|(1+t_n)}| \left|\int_{0}^{t_n}g_2(s)dB(s)\right| \\
    &=\left|\int_{t_n}^{t}g_2(s)dB(s)\right| + 2\sqrt{|b|}\left(\frac{1+t-(1+t_n)}{\sqrt{1+t}+\sqrt{1+t_n}}\right)
    \left|\int_{0}^{t_n}g_2(s)dB(s)\right|,
\end{align*}
where the Lipschitz continuity of $\cos(2\sqrt{|b|(1+\cdot)}-\pi/4)$
on  $\mathbb{R}$ has been used. A similar inequality can be
developed for $|\tilde{Y}_2(t)-\tilde{Y}_2(t_n)|$ for $t\in
[t_n,t_{n+1}]$.
Using the fact that $|\tilde{Y}(t)-\tilde{Y}(t_n)|\leq
|\tilde{Y}_1(t) - \tilde{Y}_1(t_n)|+
|\tilde{Y}_2(t)-\tilde{Y}_2(t_n)|$, we obtain
\begin{align}
    \sup_{t_n\leq t\leq t_{n+1}}|\tilde{Y}(t)-\tilde{Y}(t_n)|
  \leq \sup_{t_n\leq t\leq t_{n+1}}\left|\int_{t_n}^{t}g_2(s)dB(s)\right|
  + \sup_{t_n\leq t\leq t_{n+1}}\left|\int_{t_n}^{t}g_4(s)dB(s)\right| \notag\\
    + 2\sqrt{|b|}\left(\frac{t_{n+1}-t_n}{2\sqrt{1+t_n}}\right)
    \left\{\left|\int_{0}^{t_n}g_2(s)dB(s)\right| + \left|\int_{0}^{t_n}g_4(s)dB(s)\right| \right\}, \label{eq:YtY}
\end{align}
where we have used the fact that
$1/(\sqrt{1+t}+\sqrt{1+t_n})\leq1/(2\sqrt{1+t_n})$ for $t\geq t_n$.
We now estimate the order of the largest fluctuations of each term
on the right hand side of \eqref{eq:YtY}. We show that, for
$i\in\{2,4\}$
\begin{equation}\label{eq:glim0}
    \limsup_{n\to\infty}\frac{\sup_{t_n\leq t\leq t_{n+1}}\left|\int_{t_n}^{t}g_i(s)dB(s)\right|}{t_n^{3/4}\sqrt{\log\log t_n}} =0, \quad \text{a.s.}
\end{equation}
Now, let $\epsilon_n>0$. By the martingale time change theorem, for
every $n$, there exists a standard Brownian motion $\tilde{B}_{i,n}$
such that
\begin{align*}
    \mathbb{P}\left[\sup_{t_n\leq t\leq t_{n+1}}\left|\int_{t_n}^{t}g_i(s)\,dB(s)\right|\geq\epsilon_n\right]
    &=\mathbb{P}\left[\sup_{t_n\leq t\leq t_{n+1}}\left|\tilde{B}_{i,n}\left( \int_{t_n}^{t}g_i(s)^2\,ds \right)\right|\geq\epsilon_n\right] \\
    &=\mathbb{P}\left[ \sup_{0\leq u\leq \int_{t_n}^{t_{n+1}}g_i(s)^2 ds}\left|\tilde{B}_{i,n}(u)\right| \geq\epsilon_n\right]
    \end{align*}
    Hence there is a Brownian motion $B_{i,n}^\ast$ such that
    \begin{align}
    \lefteqn{\mathbb{P}\left[\sup_{t_n\leq t\leq t_{n+1}}\left|\int_{t_n}^{t}g_i(s)\,dB(s)\right|\geq\epsilon_n\right]}\nonumber\\
    &\leq 2\,\mathbb{P}\left[ \sup_{0\leq u\leq \int_{t_n}^{t_{n+1}}g_i(s)^2 ds} B^*_{i,n}(u) \geq\epsilon_n\right]
    = 2\,\mathbb{P}\left[ \left|B_{i,n}^*\left( \int_{t_n}^{t_{n+1}}g_i(s)^2 ds \right) \right| \geq\epsilon_n\right] \nonumber\\
    \label{eq.estgitntnpl1}
    &= 4\,\mathbb{P}\left[ B_{i,n}^*\left( \int_{t_n}^{t_{n+1}}g_i(s)^2 ds \right) \geq\epsilon_n\right]
    = 4\,\left\{ 1-\Phi\left(\frac{\epsilon_n}{\sqrt{\int_{t_n}^{t_{n+1}}g_i(s)^2 ds}} \right) \right\},
\end{align}
where we have used the fact that $\max_{0\leq s\leq t} W(s)$ has the
same distribution as $|W(t)|$ when $W$ is a standard Brownian
motion, the symmetry of the distribution of a standard Brownian
motion, and $\Phi$ denotes the distribution function of a standard
normal random variable. Now,
\begin{align*}
    g_2(t)^2 &= \pi |b|^{1/2} (1+t)^{1/2} \sin^2\left(2\sqrt{|b|(1+t)}-\frac{3}{4}\pi\right)\leq \pi |b|^{1/2} (1+t)^{1/2}, \\
    g_4(t)^2 &= \pi |b|^{1/2} (1+t)^{1/2} \cos^2\left(2\sqrt{|b|(1+t)}-\frac{3}{4}\pi\right)\leq \pi |b|^{1/2} (1+t)^{1/2}.
\end{align*}
Thus, by \eqref{eq.Njtjasy} we have
\begin{equation}\label{eq:giasy}
    \int_{t_n}^{t_{n+1}}g_i(s)^2 ds \leq \pi |b|^{1/2} (1+t_{n+1})^{1/2} (t_{n+1}-t_n) \sim \pi |b|^{1/2} t_n^{1/2} \Delta t_n , \quad \text{ as $n\to\infty$},
\end{equation}
and therefore by \eqref{eq.Njtjasy} and \eqref{eq:jall}
\[
    \limsup_{n\to\infty}
    \frac{\sqrt{\int_{t_n}^{t_{n+1}}g_i(s)^2 ds}}{ N_n}
    \leq \limsup_{n\to\infty} \frac{\pi^{1/2} |b|^{1/4}t_n^{1/4} (\Delta t_n)^{1/2}}{N_n}
    = \frac{1}{|b|^{1/2}} \pi^{2} \frac{1}{2^{1/2+k/2}}.
\]
So letting $\epsilon_n=t_n^{5/8}\sqrt{\log\log t_n}$, and using the
last relation and \eqref{eq.Njtjasy} gives
\begin{align*}
   \lefteqn{\liminf_{n\to\infty}
    \frac{\epsilon_n}{\sqrt{\int_{t_n}^{t_{n+1}}g_i(s)^2 ds} \, \cdot n^{1/4} \sqrt{\log\log n}}}\\
    &=
 \liminf_{n\to\infty}
    \frac{t_n^{5/8}\sqrt{\log\log t_n}}{N_n n^{1/4} \sqrt{\log\log n}}\frac{N_n}{\sqrt{\int_{t_n}^{t_{n+1}}g_i(s)^2 ds}}\\
     &\geq
     \liminf_{n\to\infty} \frac{(\frac{1}{|b|}\frac{1}{2^4 2^{2k}} n^2\pi^2)^{5/8}}{\frac{n}{2^{2}2^{k}} n^{1/4}}
   \frac{1}{\frac{1}{|b|^{1/2}} \pi^{2} \frac{1}{2^{1/2+k/2}}}=:C_k'>0.
\end{align*}
Therefore there exists a positive constant $C_k$ such that
\[
    \frac{\epsilon_n}{\sqrt{\int_{t_n}^{t_{n+1}}g_i(s)^2 ds}} \geq C_k (1+n)^{1/4} \sqrt{\log\log (n+e^e)}, \quad n\geq 1.
\]
By \eqref{eq.estgitntnpl1}, this implies 
\begin{multline*}
    \mathbb{P}\left[\sup_{t_n\leq t\leq t_{n+1}}\left|\int_{t_n}^{t}g_i(s)dB(s)\right|\geq\epsilon_n\right]
    \\\leq 4\,\left\{ 1-\Phi\left( C_k (1+n)^{1/4} \sqrt{\log\log (n+e^e)} \right) \right\}, \quad n\geq 1.
\end{multline*}
Now from \cite[Problem~2.9.22]{K&S},
\[
    1-\Phi(x) = \frac{1}{\sqrt{2\pi}}\int_{x}^{\infty}\e^{-u^2/2}du \leq
    \frac{1}{\sqrt{2\pi}}\frac{1}{x}\e^{-x^2/2}, \quad x>0.
\]
Thus, for $n\geq 1$
\begin{multline*}
\mathbb{P}\left[\sup_{t_n\leq t\leq
t_{n+1}}\left|\int_{t_n}^{t}g_i(s)dB(s)\right|\geq\epsilon_n\right]
\\ \leq \frac{4}{\sqrt{2\pi}}\frac{1}{C_k (1+n)^{1/4} \sqrt{\log\log
(n+e^e)}}\e^{-\frac{1}{2}C_k^2 (1+n)^{1/2} \log\log (n+e^e)}.
\end{multline*}
Therefore
\[
    \sum_{n=0}^{\infty}\mathbb{P}\left[\sup_{t_n\leq t\leq t_{n+1}}\left|\int_{t_n}^{t}g_i(s)dB(s)\right|\geq\epsilon_n\right]<+\infty.
\]
The Borel-Cantelli Lemma then gives that
\[
\limsup_{n\to\infty} \frac{\sup_{t_n\leq t\leq
t_{n+1}}\left|\int_{t_n}^{t}g_i(s)dB(s)\right|}{t_n^{5/8}\sqrt{\log\log
t_n}}\leq 1, \quad\text{a.s.}
\]
Therefore \eqref{eq:glim0} holds. We now show  for $i\in \{2,4\}$
that
\begin{equation}\label{eq:lilgi}
    \limsup_{n\to\infty}\frac{\left|\int_{0}^{t_n}g_i(s)dB(s)\right|}{\sqrt{2}t_n^{3/4}\log\log t_n}
    = \frac{\sqrt{\pi} |b|^{1/4}}{\sqrt{3}} , \quad \text{a.s.}
\end{equation}
Define
\[
    X_n^{(i)} := \int_{t_{n-1}}^{t_{n}}g_i(s)dB(s), \quad n\geq 1.
\]
Then
\[
    S_n^{(i)} := \int_{0}^{t_{n}}g_i(s)dB(s) = \sum_{j=1}^{n}X_j^{(i)}.
\]
Now from \eqref{eq:jall}, \eqref{eq.Njtjasy} and \eqref{eq:giasy}
we get
\[
    \sigma_n^2:=\Var[X_n^{(i)}] = \int_{t_{n-1}}^{t_{n}}g_i(s)^2 ds = O(t_n), \quad \text{as $n\to\infty$},
\]
while, from Lemma~\ref{lm:intsub}
\[
    s_n^2:=\Var[S_n^{(i)}] = \int_{0}^{t_{n}}g_i(s)^2 ds \sim \frac{\pi}{3}|b|^{1/2} t_n^{3/2}, \quad \text{as $n\to\infty$}.
\]
and so $\sigma_n/s_n \to0$ as $n\to\infty$. Hence we may apply
Lemma~\ref{lm:norlil} to $S^{(i)}_n$ to obtain
\[
    \limsup_{n\to\infty}\frac{|\int_{0}^{t_{n}}g_i(s)dB(s)|}
    {\sqrt{2\int_{0}^{t_{n}}g_i(s)^2 ds\,\log\log \int_{0}^{t_{n}}g_i(s)^2 ds}} =1, \quad \text{a.s.}
\]
which is equivalent to \eqref{eq:lilgi}.

Lastly observe from \eqref{eq:jall} and \eqref{eq.Njtjasy} that
\begin{equation} \label{eq.deltndivtnsqrt}
    \lim_{n\to\infty}2\sqrt{|b|}\left(\frac{t_{n+1}-t_n}{2\sqrt{1+t_n}}\right) = \frac{\pi}{2.2^k}.
\end{equation}
Scaling \eqref{eq:YtY}, taking limit superiors across the resulting
inequality, and employing \eqref{eq:glim0}, \eqref{eq:lilgi} and
\eqref{eq.deltndivtnsqrt} gives
\begin{equation} \label{eq.contyYtYTn}
    \limsup_{n\to\infty}\frac{\sup_{t_n\leq t\leq t_{n+1}}|\tilde{Y}(t)-\tilde{Y}(t_n)|}{\sqrt{2}t_n^{3/4}\sqrt{\log\log t_n}}
    \leq \frac{\pi}{2^k}\frac{\sqrt{\pi} |b|^{1/4}}{\sqrt{3}}, \quad \text{a.s.}
\end{equation}
Next, define
\[
K_n:=\frac{\sup_{t_n\leq t\leq
t_{n+1}}|\tilde{Y}(t)-\tilde{Y}(t_n)|}{\sqrt{2}t_n^{3/4}\sqrt{\log\log
t_n}}
    + \frac{|\tilde{Y}(t_n)|}{\sqrt{2}t_n^{3/4}\sqrt{\log\log t_n}}.
\]
Since for every $t>0$ there exists $N(t)$ such that $t_{N(t)}\leq
t<t_{N(t)+1}$, it follows from \eqref{eq:Ytlil} that
\[
\frac{\tilde{Y}(t)}{\sqrt{2} t^{3/4}\sqrt{\log \log t}} \leq K_{N(t)}. 
\]
Now, by \eqref{eq.contyYtYTn} and \eqref{eq:YtYn} we have that
\[
\limsup_{n\to\infty}K_n
    \leq \frac{\sqrt{\pi}|b|^{1/4}}{\sqrt{3}}\left(\frac{\pi}{2^k} + 1 \right)
\]
and since $N(t)\to +\infty$ as $t\to\infty$, we have
\[
    \limsup_{t\to\infty}\frac{\tilde{Y}(t)}{\sqrt{2}t^{3/4}\sqrt{\log\log t}} 
    \leq \frac{\sqrt{\pi}|b|^{1/4}}{\sqrt{3}}\left(\frac{\pi}{2^k} + 1 \right)
\]
holding on an almost sure set $\Omega_{k}$. This result also holds
on the almost sure set $\Omega^*=\bigcap_{k\in\Z^+}\Omega_k$ and
hence
\[
 \limsup_{t\to\infty}\frac{\tilde{Y}(t)}{\sqrt{2}t^{3/4}\sqrt{\log\log t}}
 \leq \frac{\sqrt{\pi}|b|^{1/4}}{\sqrt{3}}, \quad \text{a.s.}
\]
Since $\tilde{Y}(t)=\sqrt{\pi}|b|^{1/4}(1+t)^{1/4}Y(t)$, we have
that
\[
 \limsup_{t\to\infty}\frac{Y(t)}{\sqrt{2 t\log\log t}}
 \leq \frac{1}{\sqrt{3}}, \quad \text{a.s.}
\]
Combining this upper bound on the limit superior with
\eqref{eq:Ytlow} gives the required limit superior.

The limit inferior result may be obtained by considering the process
$Z(t)=-Y(t)$. Then
\[
    Z(t) = g_1(t)\int_{0}^{t}g_2(s)dW(s)+g_3(t)\int_{0}^{t}g_4(s)dW(s), \quad t\geq0,
\]
where $W(t):=-B(t)$ is a standard Brownian motion. One then may
apply the foregoing argument to deduce that
\[
    -\liminf_{t\to\infty}\frac{Y(t)}{\sqrt{2t\log\log t}}= \limsup_{t\to\infty}\frac{-Y(t)}{\sqrt{2t\log\log t}}
    =\limsup_{t\to\infty}\frac{Z(t)}{\sqrt{2t\log\log t}}=\frac{1}{\sqrt{3}}, \quad \text{a.s.}
\]
as required.
\end{proof}

The proof of Theorem~\ref{thm:Bes} can now given. It is chiefly
concerned with identifying the leading order terms which contribute
to the overall asymptotic behaviour of $X$. The asymptotic behaviour
of these leading order terms are then known from
Lemma~\ref{lm:uplil}.

\begin{proof}[Proof of Theorem~\ref{thm:Bes}]
By \eqref{eq.xrep}, \eqref{eq:detspec7}, and \eqref{eq:resdecomp8}
the solution $X$ of \eqref{eq.longmemorysdde} has the representation
\begin{equation}\label{eq:XBesrep}
    X(t) = r_7(t)c_7 + r_8(t)c_8 + \sigma r_7(t)\int_{0}^{t}d_7(s)dB(s) + \sigma r_8(t)\int_{0}^{t}d_8(s)dB(s).
\end{equation}
By \eqref{eq:r7r8} and \eqref{eq:Besasy}, $r_7$ and $r_8$ have
asymptotic behaviour given by
\begin{align*}
    r_7(t) &= \frac{1}{\sqrt{\pi}} |b|^{-1/4}(1+t)^{-1/4} \{\cos(2\sqrt{|b|(1+t)} - \pi/4) + O(t^{-1/2}) \} ,
    \quad \text{as $t\to\infty$},\\
    r_8(t) &= \frac{1}{\sqrt{\pi}} |b|^{-1/4}(1+t)^{-1/4} \{\sin(2\sqrt{|b|(1+t)} - \pi/4) + O(t^{-1/2}) \}, \quad \text{ as $t\to\infty$}.
    \end{align*}
Also by \eqref{eq:cntsd5} and \eqref{eq:Besasy}, $d_7$ and $d_8$
have asymptotic behaviour given by
\begin{align*}
    d_7(s)
    &= \sqrt{\pi} |b|^{1/4} (s+1)^{1/4} \left(\sin\left(2\sqrt{|b|(s+1)}-\frac{3}{4}\pi\right) + O(s^{-1/2})\right),
    \quad \text{ as $s\to\infty$}, \\
    d_8(s)
    &= \sqrt{\pi} |b|^{1/4} (s+1)^{1/4} \left(\cos\left(2\sqrt{|b|(s+1)}-\frac{3}{4}\pi\right) + O(s^{-1/2}) \right),
    \quad \text{ as $s\to\infty$}.
\end{align*}
Define the functions $R_7$, $R_8$, $D_7$ and $D_8$ so that, for
$s\geq 0$ and $t\geq 0$ we have
\begin{subequations}\label{eq:rRdD}
\begin{align} \label{eq:R7R8}
r_7(t) &= g_1(t) + R_7(t), \quad r_8(t) = g_3(t) + R_8(t), \\
 \label{eq:D7D8}
 d_7(s) &= g_2(s) + D_7(s), \quad d_8(s) = g_4(s) + D_8(s),
\end{align}
\end{subequations}
where $g_1$,$g_2$,$g_3$ and $g_4$ are as defined in
\eqref{eq:g1g2g3g4}. Notice that $R_7$, $R_8$, $D_7$ and $D_8$ are
continuous functions. Since
\begin{gather*}
R_7(t) = O(t^{-3/4}), \quad R_8(t)=O(t^{-3/4}) \quad \text{ as $t\to\infty$}, \\
D_7(s) = O(s^{-1/4}), \quad D_8(s)=O(s^{-1/4}) \quad \text{ as
$s\to\infty$},
\end{gather*}
it follows that there exists $M>0$ such that
\begin{align*}
|R_7(t)| &\leq  M(1+t)^{-3/4}, \quad |R_8(t)|\leq M(1+t)^{-3/4} \quad t\geq 0, \\
|D_7(s)| &\leq M(1+s)^{-1/4}, \quad |D_8(s)|\leq M(1+s)^{-1/4},
\quad s\geq 0.
\end{align*}

Next, we decompose $X$ according to
\begin{multline}\label{eq:XRDg}
    \frac{X(t)}{\sqrt{2t\log\log t}} = \frac{r_7(t)c_7 + r_8(t)c_8}{\sqrt{2t\log\log t}}
    + \sigma\frac{g_1(t)\int_{0}^{t}g_2(s)dB(s)}{\sqrt{2t\log\log t}}
    + \sigma\frac{g_3(t)\int_{0}^{t}g_4(s)dB(s)}{\sqrt{2t\log\log t}} \\
    + \sigma\frac{R_7(t)\int_{0}^{t}g_2(s)dB(s)}{\sqrt{2t\log\log t}}
    + \sigma\frac{R_8(t)\int_{0}^{t}g_4(s)dB(s)}{\sqrt{2t\log\log t}} \\
    + \sigma\frac{r_7(t)\int_{0}^{t}D_7(s)dB(s)}{\sqrt{2t\log\log t}}
    + \sigma\frac{r_8(t)\int_{0}^{t}D_8(s)dB(s)}{\sqrt{2t\log\log t}} .
\end{multline}
Since $r_7(t)\to 0$ and $r_8(t)\to 0$ as $t\to\infty$, the first
term on the righthand--side of \eqref{eq:XRDg} tends to zero as
$t\to\infty$. The asymptotic behaviour of the second and third terms
is described by Lemma~\ref{lm:uplil}. We now proceed to demonstrate
that the remaining terms have do not contribute to size of the
largest oscillations of $X$.

We start by considering the last two terms on the right hand side of
\eqref{eq:XRDg}. If $\int_{0}^{\infty}D_7(s)^2ds<\infty$ then
because $r_7(t)\to 0$ as $t\to\infty$, we have
\begin{equation}\label{eq:D7asy}
    \lim_{t\to\infty}\frac{r_7(t)\int_{0}^{t}D_7(s)dB(s)}{\sqrt{2t\log\log t}}=0, \quad \text{a.s.}
\end{equation}
On the other hand, if $\lim_{t\to\infty} \int_{0}^{t}D_7(s)^2
ds=+\infty$, by using the estimate on $D_7$, for all $t\geq 0$ we
have
\[
    \int_{0}^{t}D_7(s)^2 ds \leq M^2\int_{0}^{t}(1+s)^{-1/2} ds \leq 2M^2(1+t)^{1/2}.
 \]
Therefore
\[
    \limsup_{t\to\infty}\frac{2\int_{0}^{t}D_7(s)^2ds\log\log\int_{0}^{t}D_7(s)^2ds}{t^{1/2}\log\log t}\leq 4M^2.
\]
Hence by the Law of the iterated logarithm for continuous
martingales, we have
\begin{align*}
\lefteqn{\limsup_{t\to\infty}
\frac{\left|r_7(t)\int_{0}^{t}D_7(s)dB(s)\right|}{\sqrt{2t\log\log t}}}\\
&=\limsup_{t\to\infty}\frac{|r_7(t)|\left|\int_{0}^{t}D_7(s)dB(s)\right|}{\sqrt{2\int_{0}^{t}D_7(s)^2ds\log\log\int_{0}^{t}D_7(s)^2ds}}
\frac{\sqrt{2\int_{0}^{t}D_7(s)^2ds\log\log\int_{0}^{t}D_7(s)^2ds}}{\sqrt{2t\log\log t}}\\
&=\limsup_{t\to\infty}\frac{|r_7(t)|\sqrt{2\int_{0}^{t}D_7(s)^2ds\log\log\int_{0}^{t}D_7(s)^2ds}}{\sqrt{2t\log\log t}}\\
&\leq
M\limsup_{t\to\infty}\frac{t^{-1/4}\sqrt{2\int_{0}^{t}D_7(s)^2ds\log\log\int_{0}^{t}D_7(s)^2ds}}{\sqrt{t^{1/2}\log\log
t}}
\frac{t^{1/4}\sqrt{\log\log t}}{\sqrt{2t\log\log t}}\\
&\leq 2M^2\limsup_{t\to\infty}\frac{\sqrt{\log\log
t}}{\sqrt{2t\log\log t}}=0.
\end{align*}
Hence \eqref{eq:D7asy} holds.
 One may similarly show that
\begin{equation}\label{eq:D8asy}
    \lim_{t\to\infty}\frac{r_8(t)\int_{0}^{t}D_8(s)dB(s)}{\sqrt{2t\log\log t}}=0, \quad \text{a.s.}
\end{equation}
To estimate the asymptotic behaviour of the fourth and fifth terms
on the right hand side of \eqref{eq:XRDg}, we note from
Lemma~\ref{lm:intsub}, we have that
\[
\lim_{t\to\infty} \frac{\int_0^t g_2^2(s)\,ds}{t^{3/2}}
=\frac{1}{3}\pi|b|^{1/2}, \quad \lim_{t\to\infty} \frac{\int_0^t
g_4^2(s)\,ds}{t^{3/2}}=\frac{1}{3}\pi|b|^{1/2}.
\]
Therefore by the Law of the Iterated Logarithm for continuous
martingales we have
\[
    \limsup_{t\to\infty}\frac{\left|\int_{0}^{t}g_2(s)dB(s)\right|}{\sqrt{2}t^{3/4}\sqrt{\log\log t}}
    =\sqrt{\frac{\pi}{3}} |b|^{1/4},
    \quad \text{a.s.}
\]
Therefore, using the estimate on $R_7$ we have
\begin{align*}
\lefteqn{ \limsup_{t\to\infty}
\frac{\left|R_7(t)\int_{0}^{t}g_2(s)dB(s)\right|}{\sqrt{2t\log\log t}}}\\
&\leq M\limsup_{t\to\infty}
\frac{t^{-3/4}\left|\int_{0}^{t}g_2(s)dB(s)\right|}
{\sqrt{2}t^{3/4}\sqrt{\log\log t}}\frac{\sqrt{2}t^{3/4}\sqrt{\log\log t}}{\sqrt{2t\log\log t}}\\
&=M\sqrt{\frac{\pi}{3}} |b|^{1/4}\limsup_{t\to\infty}
\frac{\sqrt{2\log\log t}}{\sqrt{2t\log\log t}}=0.
\end{align*}
Thus,
\begin{equation}\label{eq:R7g2}
    \lim_{t\to\infty}\frac{R_7(t)\int_{0}^{t}g_2(s)dB(s)}{\sqrt{2t\log\log t}} =0, \quad \text{a.s.}
\end{equation}
Similarly it may be shown that
\begin{equation}\label{eq:R8g4}
    \lim_{t\to\infty}\frac{R_8(t)\int_{0}^{t}g_4(s)dB(s)}{\sqrt{2t\log\log t}} =0, \quad \text{a.s.}
\end{equation}
Then due to \eqref{eq:D7asy}, \eqref{eq:D8asy}, \eqref{eq:R7g2},
\eqref{eq:R8g4}, and Lemma~\ref{lm:uplil}, by taking the limit
superior across \eqref{eq:XRDg} we get
\[
    \limsup_{t\to\infty}\frac{X(t)}{\sqrt{2t\log\log t}} = \frac{\sigma}{\sqrt{3}}, \quad \text{a.s.}
\]
Taking the limit inferior and applying these preparatory estimates
along with  Lemma~\ref{lm:uplil} secures the corresponding limit
inferior.
\end{proof}


\begin{proof}[Proof of Theorem~\ref{thm.meanvarbessel}]
By \eqref{eq.autocovarianceaverage} and \eqref{eq:resdecomp8} we have that 
\begin{multline} \label{eq.Varbessel}
\frac{1}{\sigma^2}\text{Var}[X(t)]=\int_0^t r(t,s)^2\,ds \\
= r_7(t)^2\int_0^t d_7(s)^2\,ds + 2r_7(t)r_8(t)\int_0^t d_7(s)d_8(s)\,ds + r_8(t)^2\int_0^t d_8(s)^2\,ds.
\end{multline}
We deduce the asymptotic behaviour of the terms on the righthand side of \eqref{eq.Varbessel}. By the definition of $d_7$ we have the identity  
 \[
 \frac{1}{t^{3/2}}\int_0^t d_7(s)^2\,ds - \frac{1}{t^{3/2}}\int_0^t g_2^2(s)\,ds 
 =  2\frac{1}{t^{3/2}}\int_0^t g_2(s)D_7(s)\,ds +\frac{1}{t^{3/2}}\int_0^t D_7(s)^2.
 \]
 By the definition of $g_2$ and $D_7$, we have that $g_2(t)=O(t^{1/4})$ and $D_7(t) = O(t^{-1/4})$ as $t\to\infty$, so the limit as $t\to\infty$ 
 of the two terms on the right hand side is zero. Since the second term on the left hand side has limit $\pi |b|^{1/2}/3$ as $t\to\infty$, 
 we have 
 \begin{equation} \label{eq.intd72}
 \lim_{t\to\infty}  \frac{1}{t^{3/2}}\int_0^t d_7(s)^2\,ds = \frac{\pi}{3} |b|^{1/2}. 
 \end{equation}
 Similarly, we may establish  
  \begin{equation} \label{eq.intd82}
 \lim_{t\to\infty}  \frac{1}{t^{3/2}}\int_0^t d_8(s)^2\,ds = \frac{\pi}{3} |b|^{1/2}. 
 \end{equation}
 We determine the asymptotic behaviour of the integral in the second term on the right hand side of \eqref{eq.Varbessel}. First, we express $d_7$ and $d_8$ in terms of $g_2$, $g_4$, $D_7$ and $D_8$ to get 
 \begin{multline*}
 \int_0^t d_7(s)d_8(s)\,ds 
 \\
 =\int_0^t g_2(s)g_4(s)\,ds+\int_0^t \left\{ g_2(s)D_8(s)+ g_4(s)D_7(s)+D_7(s)D_8(s)\right\}\,ds.
 \end{multline*}
 Since $g_2(t)=O(t^{1/4})$, $g_4(t)=O(t^{1/4})$, $D_7(t) = O(t^{-1/4})$ and $D_8(t)=O(t^{-1/4})$ as $t\to\infty$, the second 
 integral on the right hand side is of order $t$ as $t\to\infty$. Finally, 
 \[
 \int_0^t  g_2(s)g_4(s)\,ds
 =\frac{1}{2}\pi  \int_0^t |b|^{1/2}  (1+s)^{1/2} \sin\left(4\sqrt{|b|(1+s)}-\frac{3}{2}\pi\right) \,ds.
 \]
Making a substitution in the integral leads to
\begin{multline*}
\int_0^t  |b|^{1/2} (1+s)^{1/2} \sin\left(4\sqrt{|b|(1+s)}-\frac{3}{2}\pi\right) \,ds\\
=\frac{1}{32|b|}\int_{4\sqrt{|b|}-\frac{3}{2}\pi}^{4\sqrt{|b|(1+t)}-\frac{3}{2}\pi }  (u+3\pi/2)^2 \sin(u) \,du.
\end{multline*}
Since the last integral can be evaluated exactly, we see that 
\[
\int_0^t  |b|^{1/2} (1+s)^{1/2} \sin\left(4\sqrt{|b|(1+s)}-\frac{3}{2}\pi\right) \,ds=O(t), \quad\text{ as $t\to\infty$},
\]
so it follows that 
  \begin{equation} \label{eq.intd7d8}
\int_0^t d_7(s)d_8(s)\,ds=O(t), \quad \text{as $t\to\infty$}. 
\end{equation}
We prepare one final estimate; it is on $r_7^2(t)+r_8^2(t)$ as $t\to\infty$. First we 
observe that because $g_1(t)=O(t^{-1/4})$, $g_3(t)=O(t^{-1/4})$, $R_7(t)=O(t^{-3/4})$ and $R_8(t)=O(t^{-3/4})$ as $t\to\infty$, 
it follows that 
\[
2g_1(t)R_7(t)+2g_3(t)R_8(t) + R_7^2(t)+R_8^2(t)=O(t^{-1}), \quad \text{ as $t\to\infty$}.
\]
Therefore
\begin{align*}
r_7^2(t)+r_8^2(t)
&=g_1^2(t)+g_3^2(t)+ 2g_1(t)R_7(t)+2g_3(t)R_8(t) + R_7^2(t)+R_8^2(t)\\
&=\frac{1}{\pi}|b|^{1/2}(1+t)^{-1/2} + O(t^{-1}),
\end{align*}
or 
\begin{equation} \label{eq.r72pr82}
\lim_{t\to\infty} \frac{r_7^2(t)+r_8^2(t)}{t^{-1/2}}=\frac{1}{\pi}|b|^{1/2}.
\end{equation}

Now, we return to estimate the asymptotic behaviour of $\text{Var}[X(t)]$ in \eqref{eq.Varbessel} using the estimates established above. We start 
by rewriting the identity \eqref{eq.Varbessel} according to
\begin{multline*}
\frac{1}{\sigma^2 t}\text{Var}[X(t)]
= \frac{r_7(t)^2}{t^{-1/2}}\left(\frac{\int_0^t d_7(s)^2\,ds}{t^{3/2}}- \frac{\pi|b|^{1/2}}{3}\right) 
\\+ 2\frac{r_7(t)}{t^{-1/4}}\frac{r_8(t)}{t^{-1/4}}\frac{\int_0^t d_7(s)d_8(s)\,ds}{t} \cdot \frac{1}{t^{1/2}}
\\+ \frac{r_8(t)^2}{t^{-1/2}}\left(\frac{\int_0^t d_8(s)^2\,ds}{t^{3/2}}- \frac{\pi|b|^{1/2}}{3}\right)
+ \frac{r_7(t)^2+r_8^2(t)}{t^{-1/2}} \cdot \frac{\pi|b|^{1/2}}{3}.
 \end{multline*}
Since $r_7(t)=O(t^{-1/4})$ and $r_8(t)=0(t^{-1/4})$, by \eqref{eq.intd72} and \eqref{eq.intd82}, the first and third terms on 
the right hand side have each limit zero as $t\to\infty$. Using these estimates on $r_7$ and $r_8$, along with  
\eqref{eq.intd7d8}, confirms that the second term has  zero limit as $t\to\infty$. The fourth term has limit $1/3$ as $t\to\infty$, by 
\eqref{eq.r72pr82}, and therefore we have 
\begin{equation*}
    \lim_{t\to\infty} \frac{\text{Var}[X(t)]}{t} =\frac{1}{3}\sigma^2,
\end{equation*}
as claimed.
\end{proof}

\end{document}